\documentclass[11pt, leqno]{amsart}

\usepackage{amsmath,amsthm,amssymb,amsfonts,MnSymbol,epsfig,graphicx,algorithm,algorithmic}
\usepackage[margin=1in]{geometry}
\usepackage{mathtools}
\usepackage{stmaryrd}
\usepackage{hyperref}
\usepackage{graphicx}
\usepackage{pdfsync}
\usepackage{color, subfigure}
\usepackage[T1]{fontenc}
\usepackage[normalem]{ulem}

\DeclareMathOperator*{\argmin}{arg\,min} 
\DeclareMathOperator*{\RR}{\mathbb{R}}
\DeclareMathOperator*{\length}{L}
\DeclareMathOperator*{\comp}{k}
\DeclareMathOperator*{\supp}{supp}
\DeclareMathOperator*{\diam}{diam}
\DeclareMathOperator*{\conv}{Conv}

\definecolor{mygreen}{rgb}{0.1,0.75,0.2}

\newcommand{\nc}{\normalcolor}

\def\XXint#1#2#3{{\setbox0=\hbox{$#1{#2#3}{\int}$} \vcenter{\hbox{$#2#3$}}\kern-.5\wd0}}

{   \end{list} }

\newcommand{\R}{\mathbb{R}}
\newcommand{\te}{\textrm}
\newcommand{\tacka}{\,\cdot\,}

\numberwithin{equation}{section}

\newtheorem{theorem}{Theorem}[section]

\newtheorem{lemma}[theorem]{Lemma}

\theoremstyle{remark}

\newtheorem{example}[theorem]{Example}

\begin{document}
\title{Multiple penalized principal curves: analysis and computation}
\author{Slav Kirov and Dejan Slep\v{c}ev}
\address{Department of Mathematical Sciences, Carnegie Mellon University, Pittsburgh, PA, 15213, USA.}
\email{skirov@andrew.cmu.edu, slepcev@math.cmu.edu}
\date{\today}

\maketitle

\begin{abstract}
We study the problem of determining the one-dimensional structure that best represents a given data set. More precisely, we take a variational approach to approximating a given measure (data) by curves. We consider an objective functional whose  minimizers are a regularization of principal curves and introduce a new functional which allows for multiple curves. 
We prove existence of minimizers and investigate their properties.
%establishing basic properties, 
%we investigate qualitative properties of the minimizers that include when to expect overfitting, and where along the data connectedness may fail. 
%we investigate the length scales at which minimizers approximate the data, and provide criteria for when overfitting and where disconnections occur. \nc
While both of the functionals used are non-convex, we show that enlarging the configuration space to allow for multiple curves leads to a simpler energy landscape with fewer undesirable (high-energy) local minima.
We provide an efficient algorithm for approximating minimizers of the functional and demonstrate its performance on real and synthetic data.
 The numerical examples illustrate the effectiveness of the proposed approach in the presence of substantial noise, and the viability of the algorithm for high-dimensional data. %We show that the proposed approach is able to find the underlying one-dimensional structure in the presence of substantial noise, and is feasible in very high dimensions.
\end{abstract}
\medskip

%\keywords{principal curves, geometry of data, curve fitting}

\textbf{Classification. }%% MSC codes here, in the form: \MSC code \sep code
49M25,   %	Discrete approximations
65D10,   %Smoothing, curve fitting
62G99, %nonparametric inference
65D18,  % 	Computer graphics, image analysis, and computational geometry 
65K10,   %	Optimization and variational techniques
49Q20 %Variational problems in a geometric measure-theoretic setting
% 49Q10  manifold problems, other than minimal surfaces
 % 	49J55  	Problems involving randomness [See also 93E20]
% 	62G20  		Nonparametric inference: Asymptotic properties

%%%%%%%%%%%%%%%%%%%%%%%%%%%%%%%%%%%%%%%%%%%%%
\section{Introduction}

We consider the problem of finding one-dimensional structures best representing the data given as point clouds. 
This is a classical problem. It has been studied in 80's by Hastie and 
Stuetzle \cite{hs89} who introduced \emph{principal curves} as the curves going through the ``middle'' of the data.  A number of modifications of principal curves, which make them more stable and easier to compute, followed \cite{Delicado, DucStu94, gw13, kegl, sms01}. However there are very few precise mathematical results on  the relation between the properties of principal curves and their variants, and the geometry of the data. 
On the other hand rigorous mathematical setup  has been developed for related problems studied in the context of optimal (transportation) network design \cite{bs02,bs03}. In particular, an objective functional studied in network design, the average-distance functional, is closely related to a regularization of principal curves.

%Our first aim is to carefully investigate a desirable variant of principal curves which is closely related to the average-distance problem. 

Our first aim is to carefully investigate a desirable variant of principal curves suggested by the average-distance problem. The objective functional includes a length penalty for regularization, and we call its minimizers \emph{penalized principal curves}. 
We establish their basic properties (existence and basic regularity of minimizers) and investigate the sense in which they approximate the data and represent the one-dimensional structure. 
One of the shortcomings of principal curves is that they tend to overfit noisy data. Adding a regularization term to the objective functional minimized by principal curves is a common way to address overfitting. The drawback is that doing so introduces bias: when data lie on a smooth curve the minimizer is only going to approximate them. We investigate the relationship between the data and the minimizers and establish how the length scales present in the data and the parameters of the functional dictate the length scales seen in the minimizers. In particular we provide the critical length scale below which variations in the input data are treated as noise and establish the typical error (bias) when the input curve is smooth. We emphasize that the former has direct implications for when penalized principal curves begin to overfit given data.

Our second aim is to introduce a functional %which allows the data to be approximated by more than one curve (\emph{multiple penalized principal curves}). 
that permits its minimizers to consist of more than one curve (\emph{multiple penalized principal curves}).
The motivation is twofold. The data itself may have one-dimensional structure that consists of more than one component, and the relaxed setting would allow it to be appropriately represented. 
%In cases when the data has one-dimensional structure that consists of more than one component, it can appropriately represented. 
%This functional allows one to appropriately represent and approximate data whose one-dimensional structure consists of more than one component. 
The less immediate appeal of the new functional is that it guides the design of an improved scheme for computing penalized principal curves. Namely, for many datasets the penalized principal curves functional has a complicated energy landscape with many local minima. This is a typical situation and an issue for virtually all present approaches to 
nonlinear principal components. As we explain below, enlarging the set over which the functional is considered (from a single curve to multiple curves) and appropriately penalizing the number of components leads to {significantly} better behavior of energy descent methods (they more often converge to low-energy local minima). %Furthermore we present conditions under which one can expect the minimizers to have only one component and thus be a minimizer to the penalized principal curve problem.  

We find topological changes of multiple penalized principal curves are governed by a critical \emph{linear density}. The linear density of a curve is the density of the projected data on the curve with respect to its length. If the linear density of a single curve drops below the critical value over a large enough length scale, a lower-energy configuration  consisting of two curves can be obtained by removing the corresponding curve segment. Such steps are the means by which configurations following energy descent stay in higher-density regions of the data, and avoid local minima that penalized principal curves are vulnerable to. Identification of the critical linear density and the length scale over which it is recognized by the functional further provide insight as to the conditions under and the resolution to which minimizers to can recover one-dimensional components of the data.

We apply modern optimization algorithms based on alternating direction method of multipliers (ADMM) \cite{admm_boyd} and closely related Bregman iterations \cite{GolOsh09, OshBur05} to compute approximate minimizers. We describe the algorithm in detail, discuss its complexity and present 
computational examples that both illustrate the theoretical findings and support the viability of the approach for finding complex one-dimensional structures in point clouds with tens of thousands of points in high dimensions.

\begin{figure}[ht]
\begin{center}
\centerline{\includegraphics[width=0.9\columnwidth]{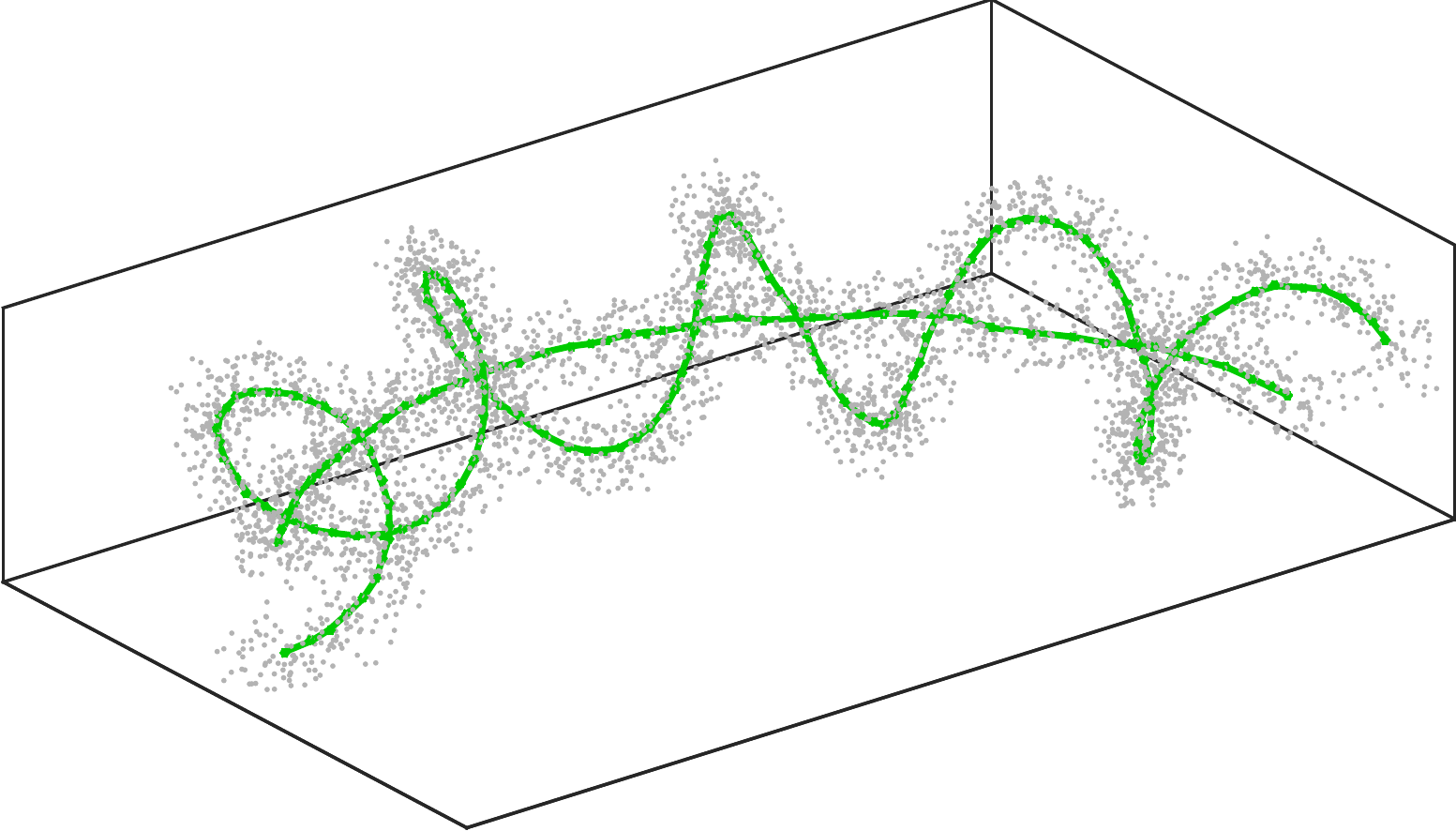}}
\vspace*{-4pt}
\caption{Example of a point cloud generated by noisy samples of two curves (not shown): a section of a circle and a curved helix wrapping around it. The green curves shown  represent the one dimensional approximation of the data cloud obtained by minimizing the proposed functional \eqref{mppc} using the algorithm of Section \ref{sec:num}.}
\end{center}
\label{ex_8}
\end{figure} 

\subsection{Outline} In Section \ref{sec:bp} we introduce the objective functionals (both for single and multiple curve approximation), and recall some of the related approaches. We establish basic properties of the functionals, including the existence of minimizers and their regularity. Under assumption of smoothness we derive the Euler-Lagrange equation for critical points of the functional. We conclude Section \ref{sec:bp} by computing the second variation of the functional.
In Section \ref{sec3} we provide a number of illustrative examples and 
investigate the relation between the length scales present in the data, the parameters of the functional and the length scales present in the minimizers. 
At the end of Section \ref{sec3}  we \nc discuss parameter selection for the functional. In Section \ref{sec:num} we describe the algorithm for computing approximate minimizers of the \eqref{mppc} functional. \nc In Section \ref{sec:ne} we provide some further numerical examples that illustrate the applicability of the functionals and algorithm. Section \ref{sec:dc} contains the conclusion and a brief discussion of and comparison with other approaches for one-dimensional data cloud approximation. Appendix \ref{apA} contains some technical details of an analysis of a minimizer considered in Section \ref{sec3}.

\section{The functionals and basic properties} \label{sec:bp}

In this section we introduce the \emph{penalized principal curves functional} and the 
\emph{multiple penalized principal curves functional}. We recall and prove some of their basic properties. Let $\mathcal{M}$ be the set of finite, compactly supported measures on $\mathbb{R}^d$, with $d\geq 2$ and $\mu(\mathbb{R}^d)>0$.

\subsection{Penalized principal curves}

% L^p-regularized principal curves
% penalized principal curves
% penalized multiple principal curves

Given a measure (distribution of data) $\mu \in \mathcal{M}$, $\lambda > 0$, and $p\geq1$, the \emph{penalized principal curves} are minimizers of
\begin{equation}
\label{ppc}
 E^{\lambda}_\mu(\gamma) := \int_{\mathbb{R}^d} d(x,\Gamma)^p d\mu(x) + \lambda \length(\gamma) 
\tag{PPC}
 \end{equation}
over $\gamma \in \mathcal{C} := \{ \gamma : [0,a] \rightarrow \mathbb{R}^d: a \geq 0,  \gamma \text{ is Lipschitz with } |\gamma'| \leq 1, \; \mathcal{L}^1-\text{ a.e.}\}$, and where $\Gamma := \gamma([0,a])$, $d(x, \Gamma)$ is the distance from $x$ to set $\Gamma$ and $\length(\gamma)$ is the length of $\gamma$:
\[ \length(\gamma) := ||\gamma||_{TV} := \sup \left \{  \sum_{i=2}^n | \gamma(x_i) - \gamma(x_{i-1})| \: : \: 0\leq x_1<x_2<...<x_n \leq a, \,n \in \mathbb{N}  \right \}  \]
 and where $|\cdot|$ denotes the Euclidean norm. The functional is closely related to the average-distance problem introduced by Buttazzo, Oudet, and Stepanov \cite{bs02} having in mind applications to optimal transportation networks \cite{bs03}. In this context, the first term can be viewed as the cost of a population to reach the network, and the second a cost of the network itself. 
There the authors considered general connected one-dimensional sets and instead of length penalty considered a length constraint. 
The penalized functional for one-dimensional sets was studied by Lu and one of the authors in \cite{LuSle13}, and later for curves (as in this paper)  in \cite{LuSle15}. 
%The functional restricted to curves, that we consider, has been considered by Lu and one of the authors \cite{LuSle15}. 
Similar functionals have been considered in statistics and machine learning literature as regularizations of the principal curves problem by Tibshirani \cite{Tib92} (introduces curvature penalization)
Kegl, Krzyzak, Linder, and Zeger \cite{kegl} (length constraint), Biau and Fischer \cite{BiaFis12} (length constraint)
Smola, Mika, Sch\"olkopf, and  Williamson \cite{sms01} (a variety of penalizations including penalizing length as in \eqref{ppc}) and others.

The first term in \eqref{ppc} measures the approximation error, while the second one penalizes the complexity of the approximation. If $\mu$ has smooth density, or is concentrated on a smooth curve, 
 the minimizer $\gamma$ is typically supported on smooth curves. However this is not universally true. Namely it was shown in \cite{Sle14} that minimizers of average-distance problems can have corners, even if $\mu$ has smooth density. An analogous argument applies to \eqref{ppc} (and later introduced \eqref{mppc}). This raises important modeling questions regarding what the best functional is and if further regularization is appropriate. We do not address these questions in this paper.
 %and instead note the relationships we later obtain between $\lambda$, length-scales of the data, and behavior of minimizers provide guidance on how large $\lambda$ should be to ensure linear stability (prevent overfitting) of local minimizers. 
%\grn (the way it was previously written did not read well in the context of the previous sentence) \nc
 %but note that the information on length-scales in the problem that we obtain can provide guidance on how much regularization is needed.  
 
Existence of minimizers of \eqref{ppc} in $\mathcal{C}$ was shown in \cite{LuSle15}. There it was also shown that any minimizer $\gamma_{\min}$ has the following total curvature bound
\[||\gamma_{\min} '||_{TV} \leq \frac{p}{\lambda}  \text{diam(supp}(\mu))^{p-1} \mu(\mathbb{R}^d). \]
The total variation (TV) above allows to treat the curvature as a measure, with delta masses at locations of corners, which is necessary in light of the possible lack of regularity. 
In \cite{LuSle15}, it was also shown that minimizing curves are injective (i.e. do not self-intersect) in dimension $d=2$ if $p\geq 2$.

\subsection{Multiple penalized principal curves}

We now introduce an extension of \eqref{ppc} which allows for configurations to consist of more than one component. Since \eqref{ppc} can be made arbitrarily small by considering $\gamma$ with many components, a penalty on the number of components is needed.
 Thus we propose the following  functional for multiple curves 
\begin{equation}
\label{mppc}
E^{\lambda_1,\lambda_2}_\mu(\gamma) :=  \int_{\mathbb{R}^d} d(x,\Gamma)^p d\mu(x) + \lambda_1 \left( \length(\gamma) + \lambda_2 \left (\comp(\gamma) - 1\right ) \right )
\tag{MPPC}
\end{equation} 
where, we relax $\gamma$ to now be piecewise Lipschitz and $\comp(\gamma)$ is the number of curves used to parametrize $\gamma$. More precisely, we aim to minimize \eqref{mppc} over the admissible set 
\[\mathcal{A} :=  \left \{ \gamma = \{\gamma^i\}_{i=1}^{\comp(\gamma)} : \comp(\gamma) \in \mathbb{N}, \, \gamma^i \in \mathcal{C}, \enspace i=1,...,\comp(\gamma) \right \}. \]
We call elements of $\mathcal A$ \emph{multiple curves}.
One may think of this functional as penalizing both zero- and one-dimensional complexities of approximations to $\mu$. In particular we can recover the (\ref{ppc}) functional by taking $\lambda_2$ large enough.  On the other hand, taking $\lambda_1$ large enough leads to a k-means clustering problem which penalizes the number of clusters, and has been encountered in \cite{ BrKulJor13, KulJor12}.

%A straightforward property is that minimizers cannot have distinct curves with endpoints within distance $\lambda_2$ of each other.

The main motivation for considering (\ref{mppc}), even if only one curve is sought, has to do with the non-convexity of the \eqref{ppc}. We will see that numerically minimizing \eqref{mppc} often helps evade undesirable (high-energy) local minima of the \eqref{ppc} functional.  In particular \eqref{mppc} can be seen as a relaxation of \eqref{ppc} to a larger configuration space.  The energy descent for \eqref{mppc} allows for curve splitting and reconnecting which is the mechanism that enables one to evade local minima of \eqref{ppc}.
\bigskip

\subsection{Existence of minimizers of \eqref{mppc}}
We show that minimizers of (\ref{mppc}) exist in $\mathcal{A}$. We follow the approach of \cite{LuSle15}, where existence of minimizers was shown for (\ref{ppc}). We first cover some preliminaries, including defining the distance between curves. If $\gamma_1,\gamma_2 \in \mathcal{C}$ with respective domains $[0,a_1], [0,a_2]$, where $a_1 \leq a_2$, we define the extension of $\gamma_1$ to $[0,a_2]$ as

\[ \tilde{\gamma}_1(t) = 
\begin{cases}
\gamma_1(t) & \text{if } t \in [0,a_1] \\
\gamma_1(a_1) & \text{if } t \in (a_1,a_2].
\end{cases} \]
We let 
$$d_{\mathcal{C}}(\gamma_1,\gamma_2) = \max_{t \in [0,a_2]} |\tilde{\gamma}_1(t) - \gamma_2(t)|.$$

We have the following lemma, and the subsequent existence of minimizers. 

\begin{lemma}
Consider a measure $\mu \in \mathcal{M}$ and $\lambda_1,\lambda_2 > 0, p \geq 1$. 
\begin{enumerate}
\item[(i)] For any minimizing sequence $\{\gamma_n\}$ of (\ref{mppc}) 
\begin{enumerate}
\item $\limsup_{n \to \infty} \comp(\gamma_n) \leq \frac{1}{\lambda_1 \lambda_2} (\diam \supp(\mu))^p$, and
\item  $\limsup_{n \to \infty} \length(\gamma_n) \leq \frac{1}{\lambda_1} (\diam \supp(\mu))^p$
\end{enumerate}
\item[(ii)] {There exists a minimizing sequence $\{\gamma_n\}$ of (\ref{mppc}) such that $\forall n, \, \Gamma_n$ is contained in $\conv(\mu)$, the convex hull of the support of $\mu$.}
\end{enumerate}
\end{lemma}

\begin{proof}
The first property follows by taking a singleton as a competitor. The second follows from projecting any minimizing sequence onto $\conv(\mu)$. Doing so can only decrease the energy, as shown in \cite{bs03,LuSle15}. The argument relies on the fact that projection onto a convex set decrease length. 
\end{proof}

\begin{lemma}
Given a positive measure $\mu \in \mathcal{M}$ and $\lambda_1,\lambda_2 > 0, p \geq 1$, the functional (\ref{mppc}) has a minimizer in $\mathcal{A}$. Moreover, the image of any minimizer is contained in the convex hull of the support of $\mu$. 
\end{lemma}

\begin{proof}
The proof is an extension of the one found in \cite{LuSle15} for \eqref{ppc}. Let $\{\gamma_n\}_{n \in \mathbb{N}}$ be a minimizing sequence in $\mathcal{A}$. Since the number of curves $\comp(\gamma_n)$ is bounded, we can find a subsequence (which we take to be the whole sequence) with each member having the same number of curves $k$. We enumerate the curves in each member of the sequence as $\gamma_n = \{\gamma_n^i\}_{i=1}^k $. We assume that each curve $\gamma_n^i$ is arc-length parametrized for all $n\in \mathbb{N},\,i \leq k$. Since the lengths of the curves are uniformly bounded, let $L = \sup_{n,i} \length(\gamma_n^i)$, and extend the parametrization for each curve in the way defined above. Then for each $i \leq k$, the curves $\{\gamma_n^i\}_{n \in \mathbb{N}}$ satisfy the hypotheses of the Arzela-Ascoli Theorem. Hence for each $i \leq k$, up to a subsequence $\gamma_n^i$ converge uniformly to a curve $\gamma^i: [0,L] \to \RR^d$. Diagonalizing, we find a subsequence (which we take to be the whole sequence) for which the aforementioned convergence holds for all $i\leq k$. Moreover, the limiting object is a collection of curves which are 1-Lipschitz since all of the curves in the sequence are. Thus $\gamma := \{\gamma^i\}_{i=1}^k  \in \mathcal{A}$.

The mapping $\Gamma \mapsto \int_{\mathbb{R}^d} d(x,\Gamma)^p d\mu(x)$ is continuous and $\Gamma \mapsto \length(\Gamma)$ is lower-semicontinuous with respect to convergence in $\mathcal{C}$. Thus $\liminf_{n \to \infty} E^{\lambda_1,\lambda_2,p}_\mu(\gamma_n) \geq E^{\lambda_1,\lambda_2,p}_\mu(\gamma)$, and so $\gamma$ is a minimizer. 
\end{proof}

%%%%%%%%%%%%%%

\subsection{First variation} \label{sec:fv}
In this section, we compute the first interior variation of the \eqref{ppc} functional and state under what conditions a smooth curve is a critical (i.e. stationary) point. 
In the case of multiple curves, one can apply the following analysis to each curve separately.

Let $\mu$ be a compactly supported measure. We will assume the curve $\gamma:[a,b] \to \R^d$ is $C^2$. Although minimizers may have corners as mentioned earlier, our use of the first variation is aimed at understanding how the parameters of the functional relate to the length scales observed in minimizers. 
We generally expect that minimizers will be $C^2$ except at finitely many points, and hence that our analysis applies to intervals between such points. 

To compute the first variation we only perturb the interior of the curve, and not the endpoints. Without loss of generality we assume that
 $|\gamma_s| = 1$, where $\gamma_s$ denotes the partial derivative in $s$.  We consider variations of $\gamma$ of the form $\gamma(s,t) = \gamma(s) + tv(s)$, where $v \in C^2([a,b], \R^d)$, $v(a)= v(b) =0$. We note that one could allow for $v(a)$ and $v(b)$ to be nonzero (as has been considered for example in \cite{Sle14}), but it is not needed for our analysis.  We can furthermore assume (by reparameterizing the curves if necessary) that $v$ is orthogonal to the curve: $v(s) \cdot \gamma_s(s) = 0 \enspace \forall s \in [a,b]$. 
  
  Let $\Gamma_t := \gamma([a,b], t)$,  the image of $\gamma(\tacka,t)$. 
  To compute how the energy \eqref{ppc} is changing when $\gamma$ is perturbed we need to describe how the distance of points to $\Gamma_t$ is changing with $t$. 
For any $x \in \supp \mu$ let $\Pi_t(x)$ be a point on $\Gamma_t$ which is closest to $x$, that is let
$\Pi_t(x)$ be a minimizer of $|x-y|$ over $y \in \Gamma_t$. For simplicity, we assume that $\Pi_t(x)$  is unique for all $x \in \supp \mu$. The general case that the closest point is non-unique can also be considered \cite{LuSle15, Sle14}, but we omit it here. A further reason this assumption is inconsequential is that since $\mu$ is absolutely continuous with respect to the Lebesgue measure 
$\mathcal{L}^d$,  the set of points where $\Pi_t$  is non-unique has $\mu$-measure zero \cite{ManMen02}.
We call $\Pi_t$ the projection of data onto $\Gamma_t$. For $x \in \supp \mu$ let $ g(x,t) := d(x, \Gamma_t)^2 = |x -\Pi_t(x)|^2$. Then 
\begin{align}
 \frac{\partial g}{\partial t} & =  - 2 (x - \Pi_t(x)) \cdot \gamma_t  \notag \\
 \frac{\partial^2 g}{\partial t^2} & = 2 \left( |\gamma_t |^2  - (x - \Pi_t(x)) \cdot \gamma_{tt} - 
\frac{(\gamma_t  \cdot \gamma_s - (x - \Pi_t(x)) \cdot \gamma_{st} )^2}{|\gamma_s|^2 - (x - \Pi_t(x))\cdot  \gamma_{ss}}   \right) \label{dd2} 
\end{align}
where $\gamma$ and its derivatives are evaluated at $(s(t),t)$, where $s  = \gamma(\tacka,t)^{-1}(\Pi_t(x))$, that is $ s(t)= \argmin_{r \in [a,b]} d(x, \gamma(t,r)) $. 
Note that for any $s \in (a,b)$ the set of points projecting onto $\gamma(s,t)$ satisfies
$\Pi_t^{-1}(\gamma(s,t)) \subset \gamma_s(s,t)^\perp$. 
Taking the derivative in $t$, and changing coordinates so that the approximation-error term is written as double integral, we obtain
$$ \frac{dE}{dt} = \int_a^b \left ( \lambda_1 \frac{\gamma_s}{|\gamma_s|} \cdot \gamma_{st} - 2 \, \alpha(s) \int_{\Pi_t^{-1}(\gamma)} \, (x - \Pi_t(x)) \cdot \gamma_t \, |1- \vec{\mathcal{K}}(s) \cdot (x - \Pi_t(x))| \, d\mu_s(x)  \right ) ds$$
where we have suppressed notation for dependence on $t$ (and in some places $s$), $|1+\vec{\mathcal{K}}(s) \cdot (\gamma(s)-x)|$ is the Jacobian for change of coordinates, and $\vec{\mathcal{K}}$ is the curvature vector of $\gamma$.  Here we have used the disintegration theorem (see for example pages 78-80 of \cite{Del79}) to rewrite an integral for $\mu$ over $\RR^n$ as an iterated integral along slices orthogonal to the curve (more precisely over the set of points that project to a given point on the curve).
We have denoted by $\alpha$ the linear density of the projection of $\mu$ to $\Gamma_t$, pulled back to the parameterization of $\gamma$; that is
$\alpha:=d (\gamma(\tacka,t)^{-1} \circ \Pi_t)_\# \mu / d \mathcal{L}^1$. By $\mu_s$ we denote the probability measure supported on the slice $\Pi_t^{-1}(\gamma(s,t))$.
Integrating by parts we obtain
$$  \left. \frac{dE}{dt} \right|_{t=0} = \int_a^b \left (- \lambda_1 \vec{\mathcal{K}} \cdot \gamma_{t} - 2 \, \alpha(s) \int_{\Pi_t^{-1}(\gamma)} \, (x - \Pi_t(x)) \cdot \gamma_t \, |1- \vec{\mathcal{K}} \cdot (x - \Pi_t(x))| \, d\mu_s(x)  \right ) ds. $$ % - \gamma_s \cdot \gamma_t \big|_a^b.$$
We conclude $\gamma$ is a stationary configuration if and only if
\begin{equation}
\label{cp_cond}
 \lambda_1 \vec{\mathcal{K}}(s) = - 2 \, \alpha(s) \int_{\Pi_t^{-1}(\gamma)} \, (x - \Pi_t(x))\, |1-\vec{\mathcal{K}}(s) \cdot (x - \Pi_t(x))| \, d\mu_s(x)
 \end{equation}
for $\mathcal{L}^1-$ a.e. $s \in (a,b)$.

%%%%%%%%%%%%%%%%%
\subsection{ Second variation.}
In this section we compute the second variation of \eqref{ppc} for the purpose of providing conditions for linear stability. That is, we focus on the case that a straight line segment is a stationary configuration (critical point), and find when it is stable under the considered perturbations (when the second variation is greater than zero).  
%For simplicity we consider the situation that a straight line segment is (part of a) steady state and obtain conditions for its linear stability.
This has important implications for determining when the penalized principal curves start to overfit the data, and is further investigated in the next section.

If $\gamma$ is a straight line segment, $\vec{\mathcal{K}} = 0$, and (\ref{cp_cond}) simplifies to  
$$  \gamma(s) = \bar{x}(s) := \int_{\Pi_t^{-1}(\gamma(s,t))} x \,d\mu_s(x) $$
for $\mathcal{L}^1- $ a.e. $s \in (a,b) $ such that $\alpha(s) \neq 0$. This simply states that a straight line is a critical point of the functional if and only if almost every point on the line it is the mean of points projecting there. In other words, the condition is equivalent to $\gamma$ being a principal curve (in the original sense). 

The second variation of the length term is 
\begin{align} \label{length_2var}
\begin{split}
 \frac{d^2}{dt^2} \length(\gamma) & =  \int_a^b \left ( \frac{\gamma_{st}}{|\gamma_s|} - \frac{\gamma_s}{|\gamma_s|^2} \left ( \frac{\gamma_{s}}{|\gamma_s|} \cdot \gamma_{st} \right ) \right ) \cdot \gamma_{st} + \frac{\gamma_{s}}{|\gamma_s|} \cdot \gamma_{stt} \,ds \\
& = \int_a^b  \frac{1}{|\gamma_s|} \left ( |\gamma_{st}|^2 - \left ( \frac{\gamma_{s}}{|\gamma_s|} \cdot \gamma_{st} \right )^2 \right ) + \frac{\gamma_{s}}{|\gamma_s|} \cdot \gamma_{stt} \, ds 
\end{split}
 \end{align}
We note that $0 = (\gamma_s \cdot \gamma_t)_s = \gamma_{ss} \cdot \gamma_t + \gamma_s \cdot \gamma_{st}$, and therefore $\gamma_s \cdot \gamma_{st} = 0$, so that the second variation of the length term becomes just $|\gamma_{st}|^2$. Thus using (\ref{dd2}) we get 
\begin{equation}
\label{2var}
 \left. \frac{d^2E}{dt^2} \right|_{t=0} = \int_a^b \left ( \lambda_1 |\gamma_{st}|^2 + 2 \, \alpha(s) \int_{\Pi_t^{-1}(\gamma)} \left (|\gamma_t|^2 - \left ( (x - \Pi_t(x)) \cdot \gamma_{st} \right )^2 + \right )d\mu_s(x)  \right ) ds.
 \end{equation}

%%%%%%%%%%%%%%%%%%%%%%%%%%%%%%%%%%%%%%%%%%%%
\section{Relation between the minimizers and the data}  \label{sec3}

In this section, our goal is to relate the parameters of the functional, the length-scales present in the data, and the length-scales seen in the minimizers. To do so we consider examples of 
data and corresponding minimizers, use the characterization of critical points of \eqref{mppc}, and perform linear stability analysis.

\subsection{Examples and properties of minimizers}
Here we provide some insight as to how minimizers of (\ref{mppc}) behave. 
We start by characterizing minimizers in some simple yet instructive cases. In the first couple of cases we focus on the behavior of single curves, and then  investigate when minimizers develop multiple components.

\subsubsection{Data on a curve.}
\label{statcurv}
Here we study the bias of penalized principal curves when the data lie on a curve without noise. If $\mu$ is supported on the image of a smooth curve, and a local minimizer $\gamma$ of \eqref{ppc} is sufficiently close to $\mu$, one can obtain an exact expression for the projection distance. More precisely, suppose that for each $ s \in (a,b)$, only one point in $\supp \mu$ projects to it. That is  $\forall s \in (a,b)$, the set of $ x_s \in \supp \mu \,$ such that
$ s$ minimizes  $| x_s - \gamma(\hat s)|$ over $  \hat{s} \in [a,b]$ is a singleton. 
Then (\ref{cp_cond}) simplifies to 
%\begin{equation}
%\label{s_curvature}
$$ \lambda_1 \mathcal{K}(s) = 2 \alpha( s) h  (1+\mathcal{K}(s) h ) $$
%\end{equation}
where $h:= |\gamma(s) - x_s|$, $\mathcal{K}$ denotes the unsigned scaler curvature of $\gamma$, and $\alpha$ is the projected linear density. Consequently 

\begin{equation}
\label{dist_eq}
h = \frac{1}{2 \mathcal{K}} \left ( \sqrt{1 + 2 \frac{\lambda_1 \mathcal{K}^2}{\alpha}} - 1 \right )
\approx
\begin{cases}
\sqrt{\frac{ \lambda_1}{2 \alpha}} & \te{if } \frac{1}{\mathcal{K}} \ll \sqrt{\frac{\lambda_1}{\alpha}}  \medskip \\
\frac{\lambda_1 \mathcal{K}}{2 \alpha} \quad & \te{if } \frac{1}{\mathcal{K}} \gg \sqrt{\frac{\lambda_1}{\alpha}}.
\end{cases}
\end{equation}
Note that always $ h \leq \sqrt{\frac{\lambda_1}{2 \alpha}} $. We illustrate the transition of the projection distance $l$
indicated in \eqref{dist_eq} with the example below.

\begin{example}\emph{Curve with decaying oscillations.}
We consider data uniformly spaced on the image of the function $\frac{x}{5} \sin(-4\pi\log(x))$, which ensures that the amplitude and period are decreasing with the same rate, as $x \to 0^+$. In Figure \ref{ex_27}, the linear density of the data is constant (with respect to arc length) with total mass $1$, and solution curves are shown for two different values of $\lambda_1$. For $x$ small enough the minimizing curve  is flat, as it is not influenced by oscillations whose amplitude is less than $\sqrt{\frac{\lambda_1}{\alpha}}$. As the amplitude of oscillations grows beyond the smoothing length scale the minimizing curves start to follow them. As $x$ gets larger and $\mathcal K$ becomes smaller, 
the projection distances at the peaks start to scale linearly with $\lambda_1$, as predicted by \eqref{dist_eq}. 
 %\blue
  Indeed, as $\mathcal{K}$ decreases to zero the ratio of the curvature of the minimizer to that of the data curve approaches one and  $\alpha$ converges to a constant,. Hence from \eqref{dist_eq} follows  that the ratio of the projection distances at the peaks converges to the ratio of the 
 $\lambda_1$ values.
%\blue One may observe in Figure \ref{ex_27} that the ratio of projection distances gets closer to $\frac{1}{4}$, which equals the ratio of the $\lambda_1$ values.
%\grn  (Since $\mathcal{K}$ corresponds to the curvature of the minimizing curves, the limit of the ratio may not be exactly $\frac{1}{4}$, as \eqref{dist_eq} could otherwise suggest.)
% [Changed since the previous wording was too strong. Do you think this is ok?]
 %\red Slav I modified the original  wording so that it better justifies the stronger statement. 
% \nc

\begin{figure}[!htbp]
\begin{center}
\centerline{\includegraphics[width=0.7\columnwidth]{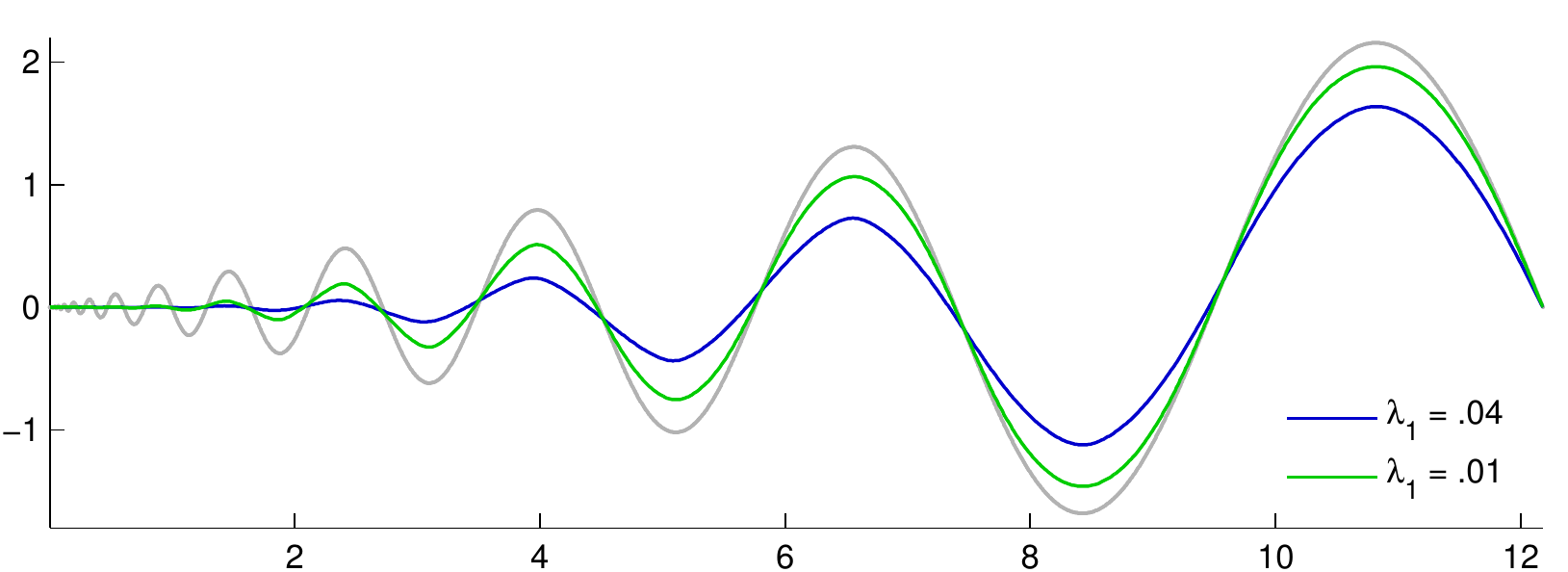}}
\centerline{\includegraphics[width=0.7\columnwidth]{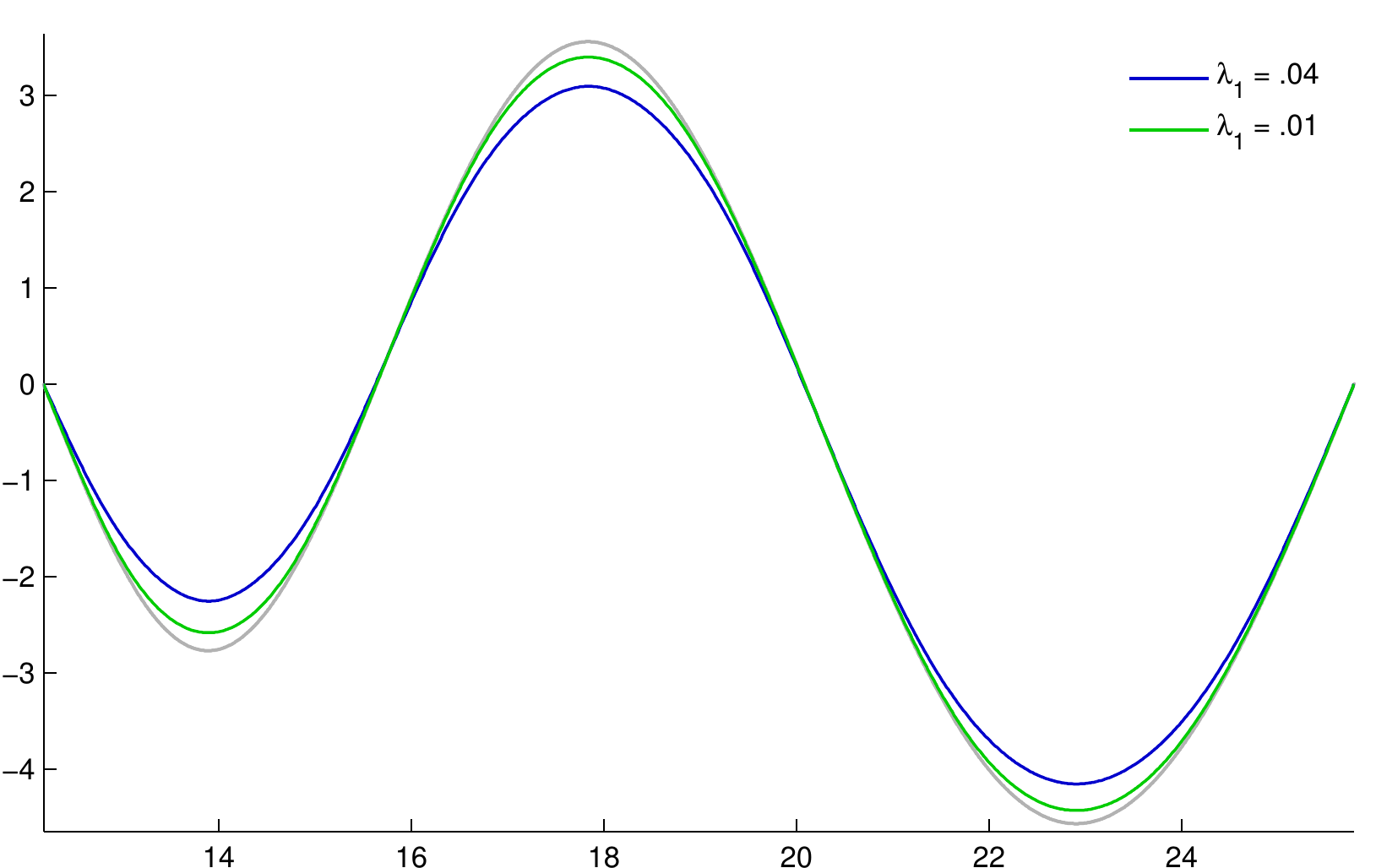}}
\caption{Numerical results shown for $n = 3000$ uniformly spaced data points (in gray) on the image of $\frac{x}{5} \sin(-4\pi\log(x))$ for $x \in [.001,e^{3.25}]$, and two different values of $\lambda_1$.}
\label{ex_27}
\end{center}
\end{figure} 

\end{example}

\subsubsection{Linear stability.}
\label{sec_linstab}
 In this section we establish conditions for the linear stability of penalized principal curves. For simplicity we consider the case when $\supp \mu \subset \R^2$. Suppose that $\gamma:[0,L] \to \R^2$ is arc-length parametrized and a stationary configuration of \eqref{ppc}, and that for some $0\leq a<b \leq L$, $\gamma([a,b])$ is a line segment. As previously, we let $\alpha$ denote the projected linear density of $\mu$ onto $\gamma$. 

We evaluate the second variation \eqref{2var} over the interval $[a,b]$, where the considered variations of $\gamma$ are $\gamma_t(s) = v(s) = (v_1(s),v_2(s))$, where $\gamma_s \cdot \gamma_t = 0$. Since $\gamma$ is a line segment on $[a,b]$, we can consider coordinates where $v_1(s) = 0$. We then have
$$ \left. \frac{d^2E}{dt^2} \right|_{t=0} = \int_a^b  \lambda_1 (v_2')^2 + 2 \alpha (s) \int_{\Pi_t^{-1}(\gamma)} \left (v_2^2 - \left ( v_2' (x - \Pi_t(x))  \right )^2 \right )d\mu_s(x)  \, ds.$$
We define the \emph{mean squared projection distance} 
\begin{equation}
\label{mspd}
H(s) := \left (\int_{\Pi_t^{-1}(\gamma)} (x - \Pi_t(x)) ^2 d\mu_s(x) \right )^\frac{1}{2}
\end{equation}
and obtain
\begin{equation}
\label{lin_stab}
 \left. \frac{d^2E}{dt^2} \right|_{t=0} = \int_a^b  \left ( \lambda_1 - 2 \alpha(s) H(s)^2 \right ) (v_2')^2 + 2 \alpha (s) v_2^2 \, ds.
 \end{equation}
We see that if $\lambda_1 \geq 2 \alpha(s) H(s)^2$ for almost every $s \in [a,b]$, then $\left. \frac{d^2E}{dt^2} \right|_{t=0} > 0$ and so $\gamma$ is linearly stable. 

On the other hand, suppose that $\lambda_1 < 2 \alpha(s) H(s)^2$ on some subinterval -- without loss of generality we take it to be the entire interval $[a,b]$. Consider the perturbation given by $v_2(s) = \sin(ns)$. Then the RHS of \eqref{lin_stab} becomes 
%\begin{equation} \label{lininst}
 $$ n^2 \int_a^b \left( \lambda_1 - 2 \alpha(s) H(s)^2 \right) \cos^2(ns) ds +  2  \int_a^b \alpha(s) \sin^2(ns)ds $$
% \end{equation}
and we see the first term dominates (in absolute value) the second for $n$ large enough. Hence line segment 
\begin{equation} \label{linstabcond}
 \gamma \te{ is linearly unstable on intervals where  } \lambda_1 < 2 \alpha H^2.
\end{equation} 

In the following examples we examine linear stability for some special cases of the data $\mu$. 

\begin{example}\emph{Parallel lines.}
We start with a simple case in which data, $\mu$, lie uniformly on two parallel lines. In Figure \ref{stab_lines} we show computed local minimizers starting with a slight perturbation of the initial straight line configuration, using the algorithm later described in Section \ref{sec:num}. The data lines are of length 2, so that $\alpha = 0.5$ for the straight line configuration. The parameter $\lambda_1=0.16$ and hence the condition for linear instability \eqref{linstabcond} of the straight line steady state becomes $0.4 < H$. The numerical results show that indeed straight line steady state becomes unstable when $H$ becomes slightly larger than $0.4$. 
%[I decided to go with your first suggestion, varying $H$ instead of $\lambda_1$ (since we already have an example with varying $\lambda_1$). I think it does not look that bad, although I am not set on the way the figure currently looks. I can make $H$ larger in both, or I can get rid of the left image entirely, and just say the minimizer for that case is nearly indistinguishable from the initial curve. Let me know what you think.]
\begin{figure}[!htb]

\minipage{0.5\textwidth}
  \includegraphics[width=\linewidth]{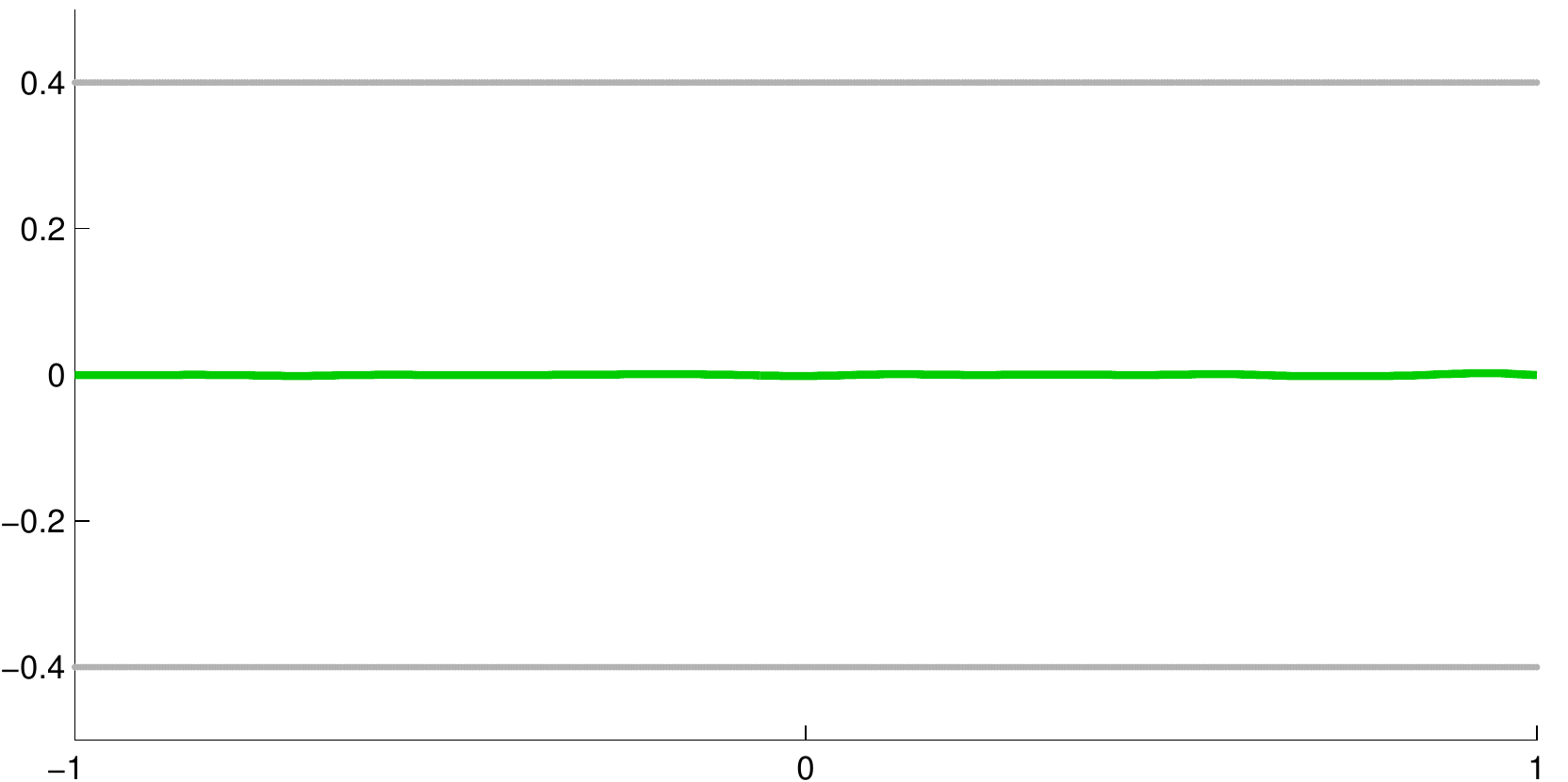}
\endminipage\hfill
\minipage{0.5\textwidth}
  \includegraphics[width=\linewidth]{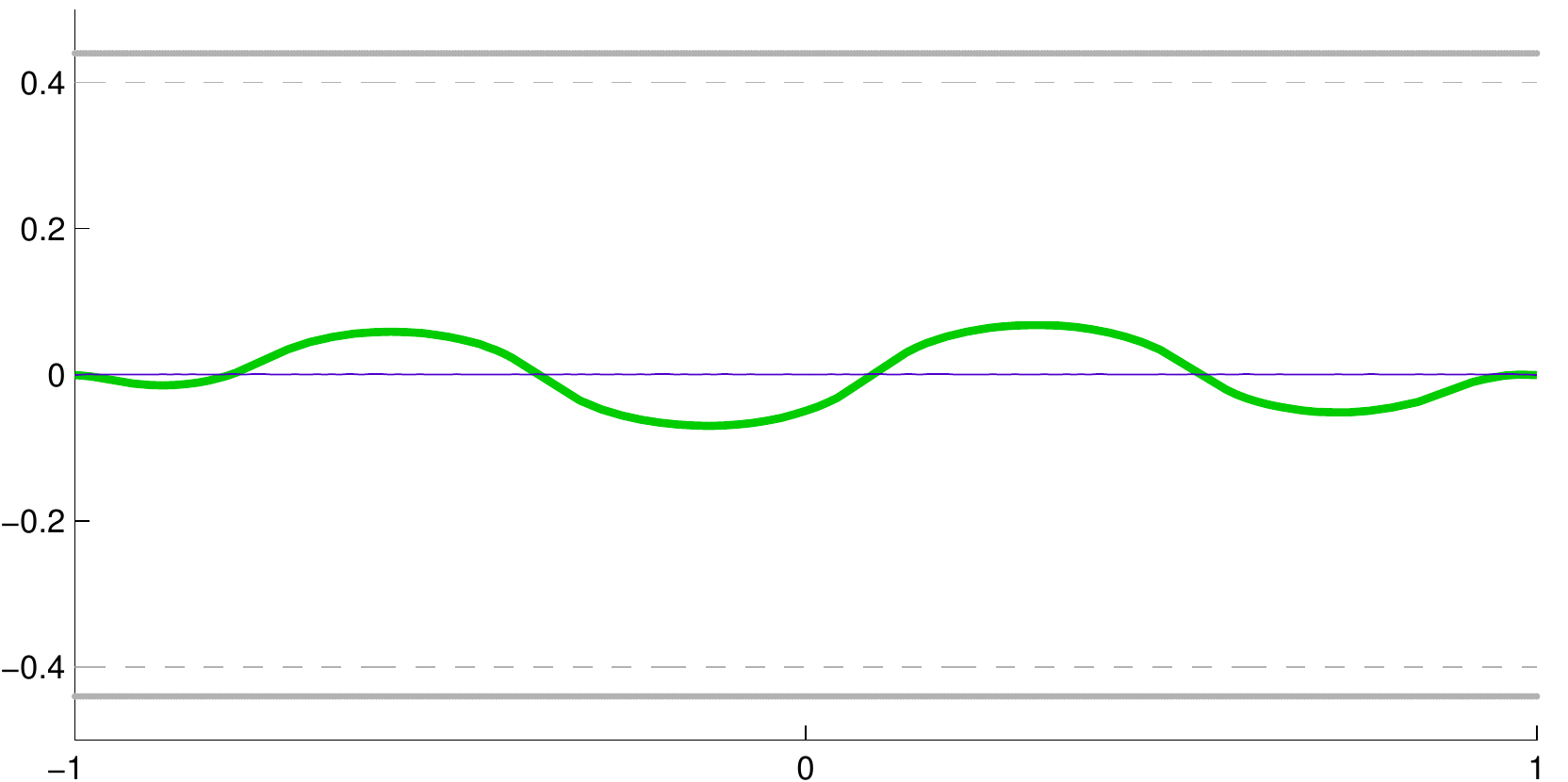}
\endminipage\hfill
\caption{The data are gray line segments at height $H=\pm 0.4$ on the left image and 
$H=\pm 0.44$ on the right image.  We numerically computed the local minimizers (green) of \eqref{mppc} among curves with fixed endpoints at $(-1,0)$ and $(1,0)$,
starting with slight perturbation of the line segment $[-1,1] \times \{0\}$.} 
\label{stab_lines}
\end{figure}
\end{example}

\begin{example}\emph{Uniform density in rectangle.} \label{udr}
Consider  a probability measure, $\mu$, with uniform density over $ [0,L] \times [0,2h]$ with $L \gg h$. Linear instability of the line segment  $\{\frac{1}{h}\} \times [0,L]$ (which is a critical point of \eqref{ppc}) can be seen as indication of when a local minimizer starts to overfit the data. It follows from \eqref{mspd} that $H^2 = \frac{1}{3} h^2$, and from \eqref{linstabcond} that $\lambda_1^* = \frac{2}{3L} h^2$ is the critical value for linear stability. 

In Figure \ref{ex_24}, we show the resulting local minimizers of \eqref{ppc} when starting from a small perturbation of the straight line, for several values of $\lambda_1$, for $h=\frac12$ and $L=4$. The results from the numerical experiment appear to agree with the predicted critical value of $\lambda_1^* =  1/24$, as the computed minimizer corresponding to $\lambda_1 = 1/27$ has visible oscillations, while that of $\lambda_1 = 1/23$ does not.

\begin{figure}[ht]
%\vskip 0.2in
\begin{center}
\centerline{\includegraphics[width=0.95\columnwidth]{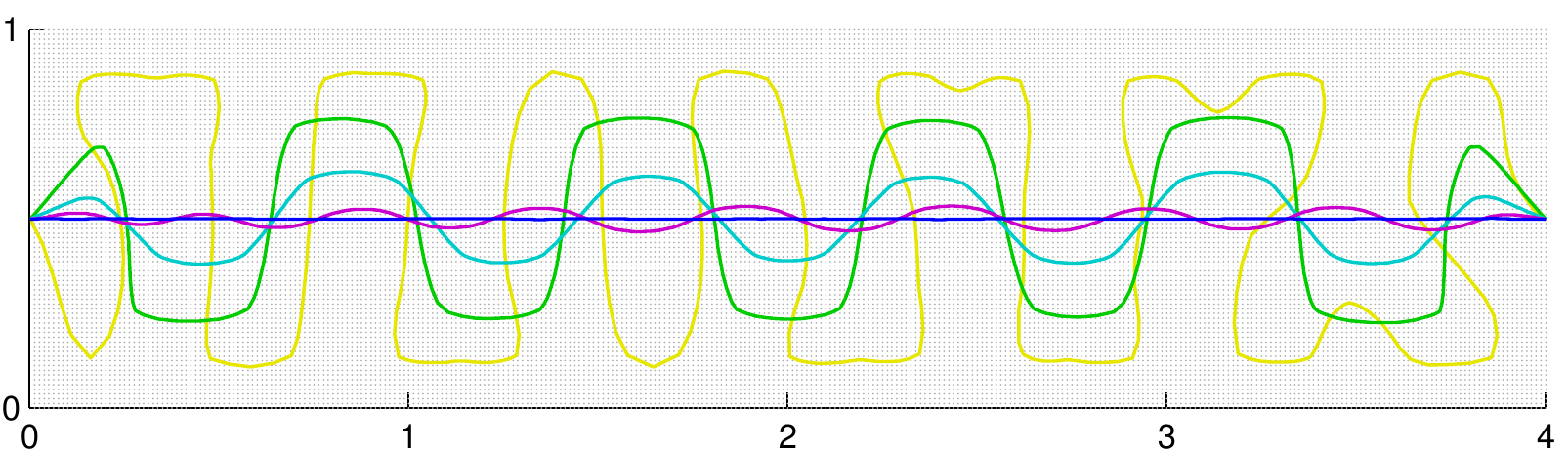}}
\caption{Numerical results showing local minimizers of \eqref{ppc} for various values of $\lambda_1$. 
The data are a grid of $n = 361\times81$ uniformly spaced points with total mass  equal to $1$.  Curves with decreasing amplitude correspond to $\lambda_1 = 1/1000, 1/150, 1/50, 1/27, 1/23$.  Recall that  the critical value for linear stability is $\lambda_1^* = 1/24$. The initial curve used for all results was a randomly perturbed straight line segment $ [0,4] \times \{\frac12\} $. % $y_0$ with $\max y_0 - \min y_0 \approx 0.005$, and 
The endpoints were kept fixed at $(0,0.5),(4,0.5)$ to avoid boundary effects.}
\label{ex_24}
\end{center}
%\vskip -0.2in
\end{figure} 
\end{example}

To illustrate how closely the curves approximate that data we consider average mean projection distance, $H$, for various values of $\lambda_1$. We expect that condition for linear stability \eqref{linstabcond}, which was derived for straight-line critical points applies, approximately, to curved minimizers. In particular we expect that curves where $H$ is larger than (approximately)  $\sqrt{\frac{\lambda_1}{2 \alpha}}$ will not be minimizers and will be evolved further by the algorithm. Here we investigate numerically if for minimizers $H \approx \sqrt{\frac{\lambda_1}{2 \alpha}}$, as is the case in one regime of \eqref{dist_eq}. Our findings are presented on  Figure \ref{lin_stab_proj_dist}.
\begin{figure}[!h]
\centerline{\includegraphics[width=0.6\linewidth]{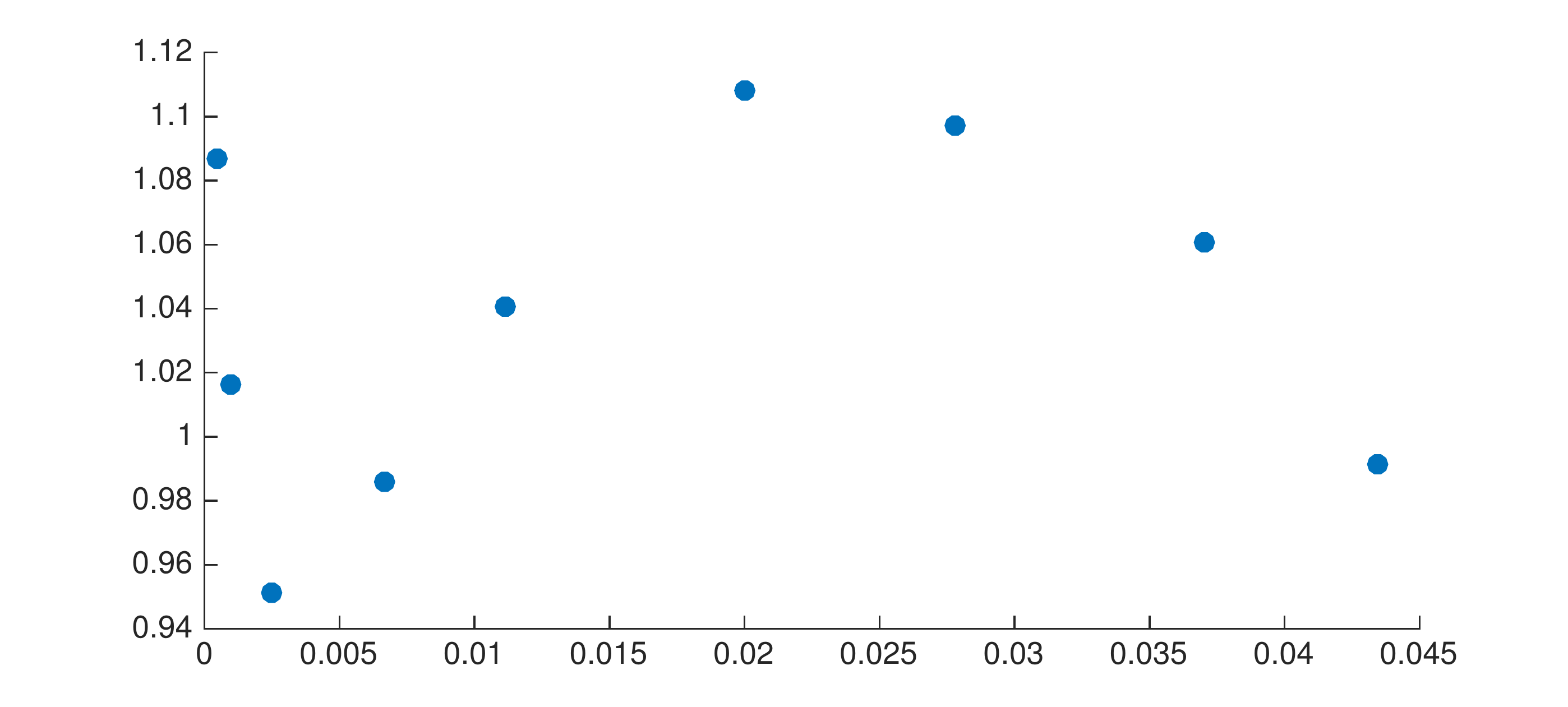}
\put(-300,70){\large $\frac{H}{\sqrt{\frac{\lambda_1}{2 \alpha}}}$}
\put(-130,-5){\large $\lambda_1$}
}
\caption{We compare the average mean projection distance $H$ (defined as average of \eqref{mspd}) to $\sqrt{\frac{\lambda_1}{2 \alpha}}$ (smoothing length scale)  for the Experiment \ref{udr}. We consider a somewhat broader set of $\lambda_1$ values than on Figure \ref{ex_24}.
We observe good agreement with the expectation, partly motivated by \eqref{linstabcond},  that $H \sim \sqrt{\frac{\lambda_1}{2 \alpha}}$. }
\label{lin_stab_proj_dist}
\end{figure}

%To illustrate how close the data are to the minimizing curve we plot the typical distance from data to the curve as a function of $\sqrt{\lambda_1}$ in Figure \ref{lin_stab_proj_dist}.
%\begin{figure}[!ht]
%\begin{center}
%\centerline{\includegraphics[width=0.4\linewidth]{lin_stab_proj_dist_sqrtlambda.pdf}}
%\caption{ A plot of the average mean projection distance versus $\sqrt{\lambda_1}$ for the solution curves shown in Figure \ref{ex_24}. We expect that the average mean projection distance should be roughly $\approx \sqrt{\frac{\lambda_1}{\alpha}}$, where $\alpha$ is the projected density onto the solution curve. Note that $\alpha$ increases as $\lambda_1$ increases, since the length of $\gamma$ decreases. This explains why the slope of the plotted curve decreases as higher values of $\lambda_1$ are reached.}
%\label{lin_stab_proj_dist}
%\end{center}
%\end{figure} 

\begin{example}\emph{Vertical Gaussian noise.}
\label{vgn}
Here we briefly remark on the case that $\mu$ has Gaussian noise with variance $\sigma^2$ orthogonal to a straight line. We note that the mean squared projection distance $H$ is just the standard deviation $\sigma$. Therefore linear instability (overfitting) occurs if and only if $\lambda_1 < 2 \alpha \sigma^2$.
\end{example}
\nc

\subsubsection{Role of $\lambda_2$.}
\label{lambda2}
We now turn our attention to the role of $\lambda_2$ in \eqref{mppc}. Our goal is to understand when do transitions in the number of curves  in  minimizers occur. 

By direct inspection of \eqref{mppc}, it is always energetically advantageous to connect endpoints of distinct curves if the distance between them is less than $\lambda_2$. Similarly, it is never advantageous to disconnect a curve by removing a segment which has length less than $\lambda_2$. Thus $\lambda_2$ represents the smallest scale at which distinct components can be detected by the \eqref{mppc} functional. When distances are larger than $\lambda_2$, connectedness is governed by the projected linear density $\alpha$ of the curves, as we investigate with the following simple example. 

\begin{example}\emph{Uniform density on line.} 
\label{udl}
In this example, we consider the measure $\mu$ to have uniform density $\alpha$ on the line segment $[0,L]\subset \mathbb{R}$. We relegate the technical details of the analysis to Appendix \ref{apA}; here we report the main conclusions. By \eqref{app:alpha} there is a critical density 
\[  \alpha^* = \left ( \frac{4}{3} \right )^{2}  \frac{\lambda_1}{\lambda_2^2} \]
such that if $\alpha > \alpha^*$ then the minimizer $\gamma$ has one component and is itself a line segment contained in $[0,L]$.  It is straightforward to check that $\gamma$ will be shorter than $L$ by a length of $h = \sqrt{\lambda_1/\alpha}$ on each side. Note that at the endpoints $H^2  = h^2/3$, which is less than the upper bound at interior points predicted by \eqref{dist_eq}.

On the other hand, if $\alpha < \alpha^*$ and $L$ is long enough then the minimizer consists of regularly spaced points on $[0,L]$ with space between them approximately (because of finite size effects)
\begin{equation}
\label{opt_gap}  
\te{gap} \approx 2  \left ( \frac{3 \lambda_1 \lambda_2}{4 \alpha} \right )^{\frac{1}{3}}.
\end{equation}
An example of this scenario is provided in Figure \ref{ex_1}.

\begin{figure}[ht]
%\vskip 0.2in
\begin{center}
\centerline{\includegraphics[width=\columnwidth]{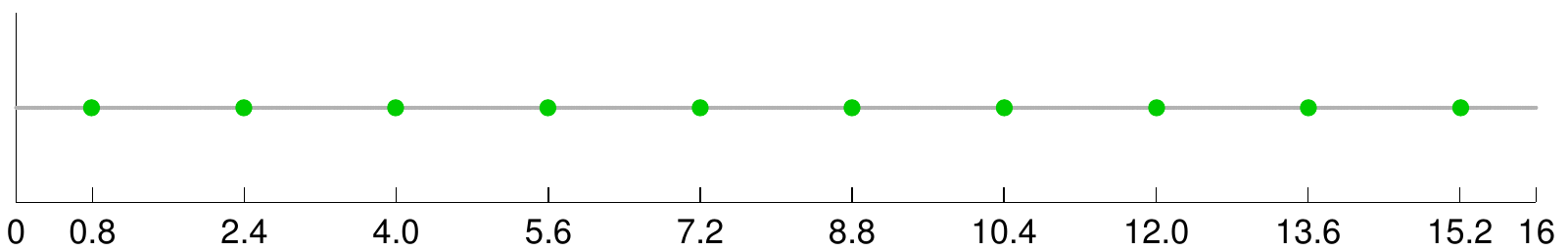}}
\caption{A minimizer for $n = 1000$ uniformly spaced points on a line segment, with total mass $1$. Here $\lambda_1 = 1/16,\, \lambda_2 = .6$ and the critical value for connectedness is $\lambda_2^* = 4/3$.  The optimal gap between the points is 1.6, compared to the approximation of $\approx 1.53$ given by (\ref{opt_gap}). The discrepancy is due to the finite length of the line segment considered in the example. }
\label{ex_1}
\end{center}
%\vskip -0.2in
\end{figure} 
\end{example}

\subsection{Summary of important quantities and length scales.} \label{sec:ls}
Here we provide an overview of how length scales present in the minimizers are affected by the parameters $\lambda_1$ and $\lambda_2$, and the geometric properties of data. We identify key quantities and length scales that govern the behavior of minimizers to \eqref{mppc}. We start with those that dictate the local geometry of penalized principal curves. 

%The following two quantities govern the geometric behavior observed in minimizers of \eqref{mppc}:
%length scales seen:
%The following length scales below explain the relation between one-dimensional data and the minimizers. 
\begin{itemize}
\item[$\sqrt{\frac{\lambda_1}{2\alpha}}$] ---  smoothing length scale (discussed in Sections \ref{statcurv} and \ref{sec_linstab} and illustrated in Example \ref{udr}). This scale represents the resolution at which data will be approximated by curves. 
Consider data generated by smooth curve with data density per length $\alpha$ and added noise (high-frequency oscillations, uniform distribution in a neighborhood of the curve, etc.). Then the noise will be "ignored" by the minimizer as long as its average amplitude (distance in space from the generating curve) is less than a constant multiple (depending on the type of noise) of  $\sqrt{\frac{\lambda_1}{2\alpha}}$. In other words  $\sqrt{\frac{\lambda_1}{2\alpha}}$ is the length scale over which the noise is averaged out. Noise below this scale is neglected by the minimizer, while noise above is interpreted as signal that needs to be approximated. For example, if we think of data as drawn by a pen, then $2 \sqrt{\frac{\lambda_1}{2\alpha}}$ is the widest the pen tip can be, for the line drawn to be considered a line by the algorithm.
% Slav, please do not change the above. If you want we can discuss this point.
%In Section \ref{statcurv}, we established that when data lie on a smooth curve (without noise) the projection distance $h$ to a minimizer is at most $\sqrt{\frac{\lambda_1}{2 \alpha}}$. In Section \ref{sec_linstab}, we found that linear instability occurs if the mean squared projection distance $H$ is greater than 
%$\sqrt{\frac{\lambda_1}{2\alpha}}$. Using linear instability, one can further informally argue that when the data do contain noise, $H^2$ will still be at most $\frac{\lambda_1}{2 \alpha}$. Minimizing curves must be linearly stable over all subintervals of their domain, and in particular, over small enough subintervals where the corresponding curve segments can be approximated by line segments. In Figure \ref{lin_stab_proj_dist}, we show the mean projection distance versus $\sqrt{\lambda_1}$ of the curves for the linearly stability Example \ref{udr}. 

\item[$\frac{\lambda_1 \mathcal K}{\alpha}$] --- bias or approximation-error length scale (discussed in Section \ref{statcurv}). 
Consider again data generated by smooth curve with data density per length $\alpha$ and curvature $\mathcal K$. If the curvature  of the curve is small (compared to $\sqrt{\frac{\lambda_1}{\alpha}}\,$) and reach is comparable to $1/\mathcal K$, then the distance from the curve to the minimizer is going to scale like $\frac{\lambda_1 \mathcal K}{\alpha}$. That is the typical error in reconstruction of a smooth curve that a minimizer makes (due to the presence of the length penalty term) scales like $\frac{\lambda_1 \mathcal K}{\alpha}$. 
%\blue
%When data lie on a smooth curve that is ``followed'' by a local minimizer with small curvature $\mathcal{K}$, the projection distance will be at most this quantity. 
%\grn [I changed this since $\mathcal{K}$ denotes the curvature of the minimizer, and not the data curve. We also don't define or use 'reach' anywhere else, so it might confuse readers who are unfamiliar with it. \red Slav I went back to the original explanation. Note that the assumptions are  such that the distance from the curve to the minimizer is  expected to be small compared to $1/\mathcal{K}$. Note that when one smoothly (in a quantitative sense) shifts the curve by am amount small compared to $1/\mathcal{K}$then the curvature does not change much. So the data curve and the minimizer will have similar curvature. 
%While this is not rigorous at this level, I much prefer an a-priory heuristic estimate on the bias that one can get only knowing the data curve than a perhaps more accurate estimate that depends on the unknown minimizing curve. ]
%\nc
\end{itemize}
\medskip

In addition to the above length scales, the following quantities govern the topology of multiple penalized principal curves:

\begin{itemize}
\item[$\lambda_2$] --- connectivity threshold (discussed in Section \ref{lambda2}). This length scale sets the minimum distance between distinct components of the solution. Gaps in the data of size $\lambda_2$ or less are not detected by the minimizer. Furthermore, this quantity provides the scale over which the following critical density is recognized. 

\medskip

\item[$ \frac{\lambda_1}{\lambda_2^2}$] --- linear density threshold (discussed in Example \ref{udl} and Appendix \ref{app:alpha}). 
Consider again data generated (possibly with noise) by a smooth curve (with curvature small compared to $\sqrt{\frac{\lambda_1}{\alpha}}$) with data density per length $\alpha$. If $\alpha$ is smaller than $  \alpha^* = \left ( \frac{4}{3} \right )^{2}  \frac{\lambda_1}{\lambda_2^2} + O(\mathcal K) $, then it is cheaper for the data to be approximated by a series of points than by a continuous curve.
That is if there are too few data points the functional no longer sees them as a continuous curve. 
If $\alpha > \alpha^*$, then the minimizers of \eqref{ppc} and \eqref{mppc} are expected to coincide, while if $\alpha < \alpha^*$, then the minimizer of \eqref{mppc} will consist of points spaced at distance  about 
$ \left( \frac{ \lambda_1 \lambda_2}{ \alpha} \right )^{\frac{1}{3}}$.
%\blue
% %Note that this length scale interpolates between the connectivity length scale and the length scale over which the noise is being averaged.
%In the case that data lie on a line segment with uniform density $\alpha$, we found that minimizers are connected if and only if $\alpha \geq \alpha^* := \frac{16}{9} \frac{\lambda_1}{\lambda_2^2}$. When there is noise, the critical linear density is the same due to the orthogonality of projections and the squared distance penalty in the functional. Furthermore, we note that non-straight configurations may be energetically unstable when $\alpha > \alpha^*$, since projection distances after disconnecting will be at most those in the straight line case. Hence $\alpha^*$ sets the threshold for disconnections to occur in minimizers. 
%\grn [Since we only did the analysis for a straight line, I think it's worth pointing out that the criteria is only one-sided when the minimizer is curved]
%\red Slav, please leave the wording as is. Note that I modified the definition of $\alpha^*$ to indicate 
%that for curved data it is approximate. 
%\nc
Note that the condition $\alpha < \alpha^*$ can also be written as $\sqrt{\frac{\lambda_1}{\alpha}} < \frac{3}{4} \lambda_2$, and thus the minimizer can be expected to consist of more than component if the connectivity threshold is greater than the smoothing length scale. 
\end{itemize}

 We also remark the following scaling properties of the functionals. Note that $E_{a\mu}^{\lambda_1,\lambda_2} = a E_\mu^{\lambda_1/a, \lambda_2}$ for any $a>0$. Thus, when the total mass of data points is changed, $\lambda_1$ should scale like $|\mu|$ to preserve minimizers. Alternatively, if $\mu_L(A) := \mu(\frac{A}{L})$ for every $A\subseteq \mathbb{R}^d$ and some $L>0$, one easily obtains that $E_{\mu_L}^{\lambda_1,\lambda_2}(L\gamma) = L^2 E_{\mu}^{\lambda_1/L,\lambda_2/L}(\gamma)$.

\subsection{Parameter selection}
Understanding the length scales above can guide one in choosing the parameters $\lambda_1, \lambda_2$. 
Here we discuss a couple of approaches to selecting these parameters. 
%Here we discuss a few approaches to selecting these parameters, depending on what kind of information the user may have about the data, or the properties one desires for the approximating curves. 
We will assume that the data measure $\mu$ has been normalized, so that it is a probability measure. 

A natural quantity to specify is a critical density $\alpha^*$, which ensures that the linear density of any found curve will be at least $\alpha^*$. From Section \ref{sec:ls} it follows that setting $\alpha^*$ imposes the following constraint on the parameters: $\frac{16}{9}  \frac{\lambda_1}{\lambda_2^2} = \alpha^*$. Alternatively, one can set $\alpha^*$ if provided a bound on the desired curve length -- if one is seeking a single curve with approximately constant linear density and length $l$ or less, then set $\alpha^* = l^{-1}$. 

There are a couple of ways of obtaining a second constraint, which in conjunction with the first determine values for $\lambda_1,\lambda_2$. 

\subsubsection{Specifying critical density $\alpha^*$ and desired resolution $H^*$.} 
One can set a desired resolution for minimizers by bounding the mean squared projection distance $H$. 
If $\alpha^*$ is set to equal the minimum of $\alpha$ along the curves then, the spatial resolution $ H $ from the data to minimizing curves is at most $ \sqrt{\frac{\lambda_1}{2 \alpha^*}}$.
Consequently, if one specifies $\alpha^*$ and desires spatial resolution $H^*$, or better, the desired parameters are:  
\[\lambda_1 = 2 \alpha^* H^{*2} \quad \te{ and } \quad \lambda_2 = \frac{4 \sqrt 2}{3} H^*. \]

Choosing proper $H^*$ depends on the level of noise present in the data. In particular, $H^*$ needs to be at least the mean squared height of vertical noise in order to prevent overfitting.

\subsubsection{Specifying critical density $\alpha^*$ and $\lambda_2$.} One may be able to choose $\lambda_2$ directly, as it specifies the resolution for detecting distinct components. In particular, there needs to be a distance of at least $\lambda_2$ between components, in order for them to detected as separate. Once set, $\lambda_1 = \frac{9}{16}  \alpha^* \lambda_2^2$.

Typically one desires the smallest (best) resolution $\lambda_2$, that does not lead to $\alpha^*$ larger than desired. Even if a single curve is sought, taking a smaller value for $\lambda_2$ can ensure less frequent undesirable local minima. 
 One case of this is later illustrated in Example \ref{parabola}, where local minimizers can oscillate within the parabola.
%Taking $\lambda_2$ small enough would not allow for the low density vertical segments to exist in a local minimizer. 
\nc

\begin{example} \emph{Line segments.}
Here we provide a simple illustration of the role of  parameters, using data generated by three line segments with noise. The line segments are of the same length, and the ratio of the linear density of data over the segments is approximately 4:2:1 (left to right). In addition, the first gap is larger than the second gap. Figure \ref{line_segs} shows how the minimizers of \eqref{mppc} computed depend on parameters used. In the Subfigures \ref{fig:threelinesa}, \ref{fig:threelinesb} \ref{fig:threelinesc} we keep $\lambda_1$ fixed while decreasing $\lambda_2$. As the critical gap length is decreased, and equivalently having more components in the minimizer becomes cheaper, the gaps in the minimizer begin to appear. It no longer sees the data representing one line but two or three separate lines. 
the only difference between functionals in  Subfigures \ref{fig:threelinesc} and  Subfigures \ref{fig:threelinesd} is that $\lambda_1$ is increased from $0.008$ to $0.024$. This results in length of the curve becoming more expensive. In Subfigure \ref{fig:threelinesd} we see that, due to low data density per length ($\alpha$), the minimizer approximates the two data patches to the right by singletons rather than curves.

%
%
%. As the linear density threshold increases, the number of detected components increases, where the two segments closest to each other are differentiated last. The last regime in Figure \ref{line_segs}(d) illustrates that increasing $\lambda_1$ and keeping $\lambda_2$ unchanged has the affect of increasing both the linear density threshold $\alpha^*$ and the smoothing length scale. 
\begin{figure}[!ht]\centering
\begin{subfigure}[$\lambda_1 = 0.008, \,  \lambda_2 = 0.7$]
{\includegraphics[width=.48\linewidth]{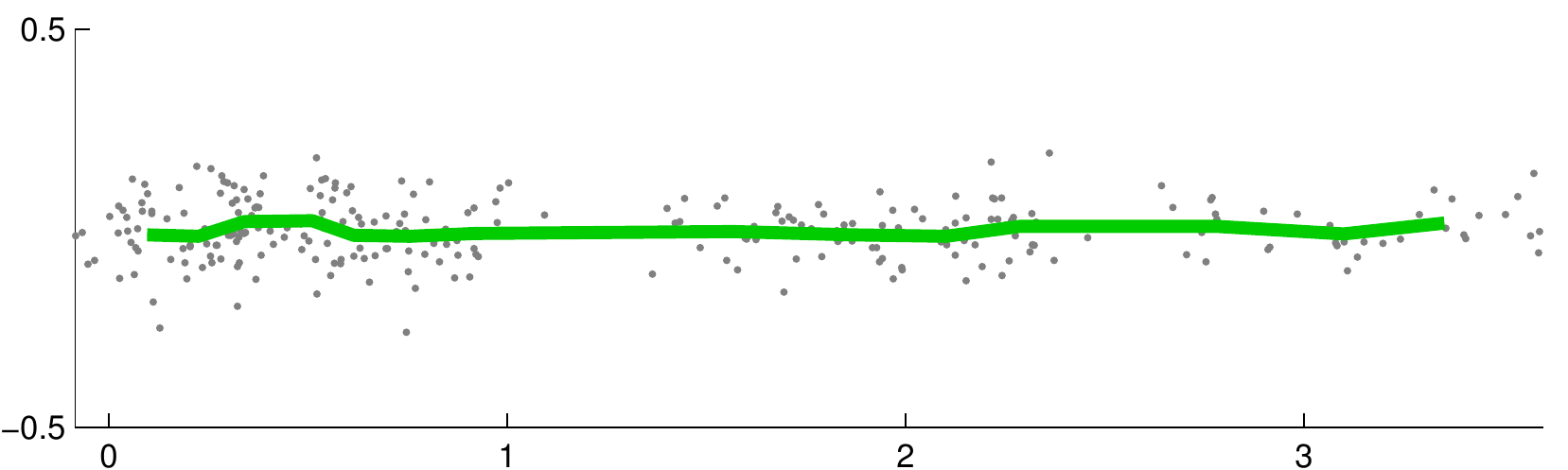}
\label{fig:threelinesa}}
\end{subfigure}
\hfill
\begin{subfigure}[$\lambda_1 = 0.008, \,  \lambda_2 = 0.5$ ]
{\includegraphics[width=.48\linewidth]{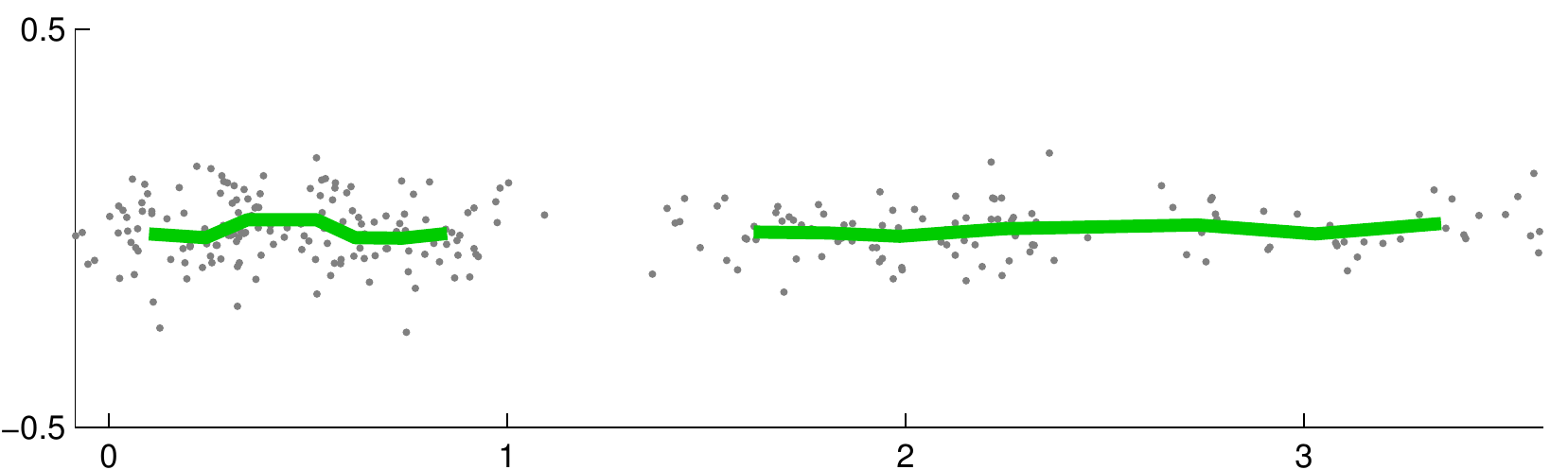}
\label{fig:threelinesb}}
\end{subfigure}
\begin{subfigure}[$\lambda_1 = 0.008, \,  \lambda_2 = 0.35$]
{\includegraphics[width=.48\linewidth]{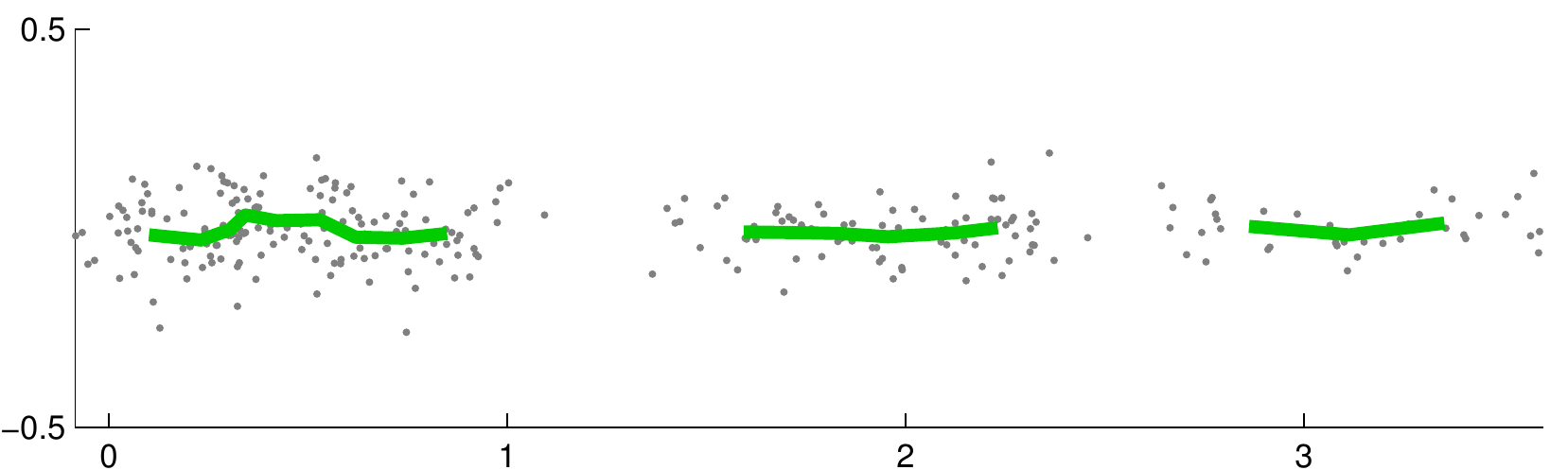}
\label{fig:threelinesc}}
\end{subfigure}
\hfill
\begin{subfigure}[$\lambda_1 = 0.024, \,  \lambda_2 = 0.35$]
{
\includegraphics[width=.48\linewidth]{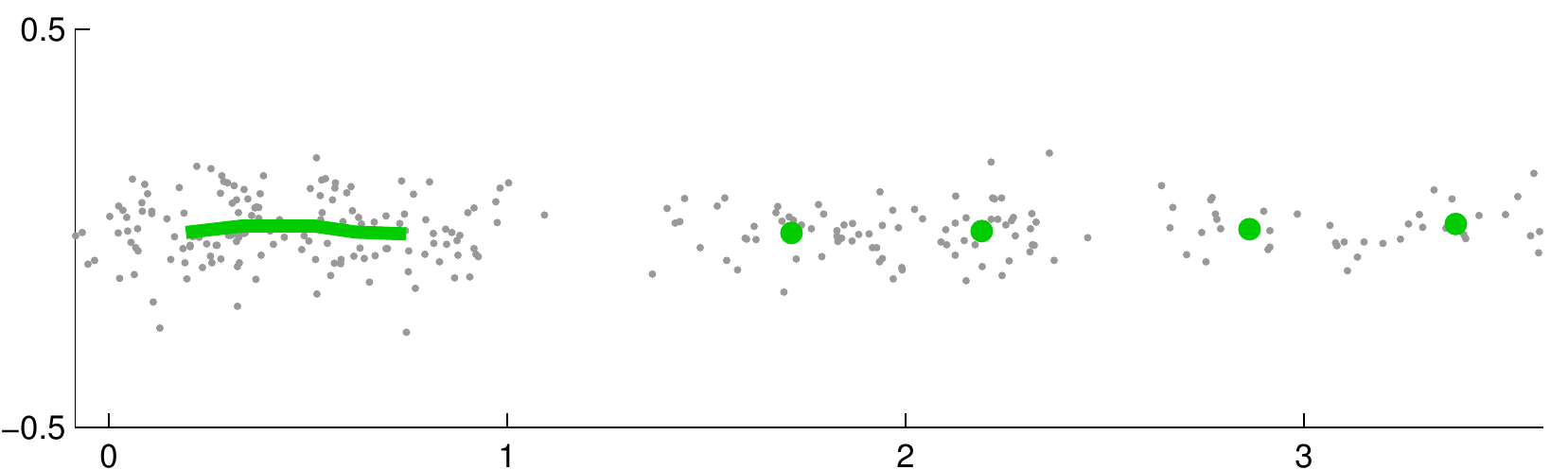}
\label{fig:threelinesd}}
\end{subfigure}
\caption{Minimizer of \eqref{mppc} shown for different parameter settings. $\lambda_1$ and $\lambda_2$. }
\label{line_segs}
\end{figure}

\end{example}
\nc

% 
%  here it would be nice to do an example. Perhaps on a river network or an image of a tree (leafless). See, and search for related images. Note that the nice rever networks are computed by knowing the terrain elevation, not from  satellite images alone.
% \begin{verbatim}
% http://cutcaster.com/vector/100597336-Cherry-tree-in-winter/
% http://www.dreamstime.com/stock-photo-cherry-tree-image8693070
%
% http://www.nature.com/ngeo/journal/v5/n10/pdf/ngeo1593.pdf
% http://onlinelibrary.wiley.com/doi/10.1029/2008EO100001/abstract
% http://onlinelibrary.wiley.com/doi/10.1029/2010WR010090/full
% http://hydrosheds.cr.usgs.gov/index.php
%
% http://rivix.com/index.php
% http://rivix.com/Gallery/Images/BC_River_Colored.png
%}
%\end{verbatim}
%\nc

%%%%%%%%%%%%%%%%%%%%%%%%%%%%%%%%%%%%%%%%%%%%%%%%%%%%%%%%
\section{Numerical algorithm for computing multiple penalized principal curves} \label{sec:num}

For this section we assume the data measure $\mu$ is discrete, with points $x_1,x_2,...,x_n \in \mathbb{R}^d$ and corresponding weights $w_1,w_2,...,w_n \geq 0$. The weights are uniform ($1/n$) for most applications, but we make note of our flexibility in this regard for cases when it is convenient to have otherwise.

For a piecewise linear curve $y = (y_1,...,y_m)$, we consider projections of data to $y_i$'s only. Hence, we approximate $d(x_i,y) \approx \min \{|x_i-y_j| : j=1,...,m\}$, unless otherwise stated. (Notation: when $a$ is a vector, as are $x_i,y_j$ in the previous line, $|a|$ denotes the Euclidean norm). Before addressing minimization of (\ref{mppc}), we first consider (\ref{ppc}) where $y$ represents a single curve.
 The discrete form is
\begin{equation}
\label{adp_discrete}
  \sum_{j=1}^m \sum_{i\in I_j} w_i |x_i-y_j|^2 + \lambda_1 \sum_{j=1}^{m-1} |y_{j+1}-y_j| \end{equation}
where 
\begin{equation} \label{Ij}
I_j := \{ i \: : \:  (\forall k =1,...,m) \;\,  |x_i - y_j| \leq |x_i - y_{k}|\} 
\end{equation}
 represents the set indexes of data points  for which $y_j$ is the closest among $\{y_1, \dots, y_m\}$.
%Slav: Do we check if $I_j$ are disjoint? If we do not do it, then we should omit the lines below. \grn Yes, the dsearchn function in MATLAB ensures this.
In case that the closest point is not unique an arbitrary assignment is made so that $I_1,...,I_m$ partition $\{1,...,n\}$ (for example set $\tilde I_j = I_j \backslash \bigcup_{i=1}^{j-1} I_i$).

 %projecting to $y_j$ (in case of ties the assignment is unique and arbitrary so that $I_1,...,I_m$ partition $\{1,...,n\}$).
%(in the case of ties, the assignment is arbitrary so that $I_j \cap I_{j'} = \emptyset$, $\forall j \neq j'$, and $\bigcup _{j=1}^m I_j = \{1,...,n\}$

\subsection{Basic approach for minimizing  \ref{ppc}}
Here we restrict our attention to performing energy decreasing steps for the (\ref{ppc}) functional. We emphasize again that this minimization problem is non-convex. The projection assignments $I_1,...,I_m$ depend on $y$ itself. However, if the projection assignments are fixed, then the resulting minimization problem is convex. This suggests the following EM-type algorithm outlined in Algorithm \ref{alg_basic}.

\begin{algorithm}%[tb]
\label{alg_ppc}
   \caption{Computing local minimizer of \eqref{ppc}}
   \label{alg_basic}
\begin{algorithmic}
   \STATE {\bfseries Input:} data $x_1,...x_n$, weights $w_1,...,w_n$, initial curve $y_1,...,y_m$, $\lambda_1 > 0$
   \REPEAT
   \STATE 1. compute $I_1,...,I_m$ defined in \eqref{Ij}
   \STATE 2. minimize (\ref{adp_discrete}) for $I_1,...,I_m$ fixed as described in Section \ref{ppcadmm}
   \UNTIL{convergence}
\end{algorithmic}
\end{algorithm}

 Note that if the minimization of (\ref{adp_discrete}) is solved exactly, then Algorithm \ref{alg_basic} converges to a local minimum in finitely many steps (since there are finitely many projection states, which cannot be visited more than once). 

\subsubsection{Minimize functional with projections fixed} \label{ppcadmm}
We now address the minimization of (\ref{adp_discrete}) with projections fixed (step 2 of Algorithm \ref{alg_basic}). One may observe that that his subproblem resembles that of a regression, and in particular the fused lasso \cite{Tib05}.

To perform the minimization we apply the alternating direction method of multipliers (ADMM) \cite{admm_boyd}, which is equivalent to split Bregman iterations \cite{GolOsh09} when the constraints are linear (our case) \cite{Ess09}. We rewrite the total variation term as $||Dy||_{1,2} := \sum_{i=1}^{m-1} |(Dy)_i|$, where $D$ is the difference operator, $(Dy)_i = y_{i+1}-y_i$ and $|\cdot|$ again denotes the Euclidean norm. An equivalent constrained minimization problem is then 
\[
\min_{y,z \,:\, z=Dy} \quad \sum_{j=1}^m \sum_{i\in I_j} w_i |x_i-y_j|^2 + \lambda ||z||_{1,2}
\]
Expanding the quadratic term and neglecting the constant, we obtain
\begin{equation}
\label{constr}
\min_{y,z \,:\, z=Dy} \quad ||y||_{\bar{w}}^2 - 2(y,\bar{x})_{\bar{w}}+ \lambda ||z||_{1,2}
\end{equation}
where notation was introduced for total mass projecting to $y_j$ by $\bar{w}_j = \sum_ {i\in I_j} w_i$, center of mass $\bar{x}_j = \frac{1}{\bar{w}_j} \sum_{i\in I_j} w_i x_i$, and weighted inner product $(y,\bar{x})_{\bar{w}} = \sum_{j=1}^m \bar{w}_j (y_j, \bar{x}_j)$. One iteration of the ADMM algorithm then consists of the following updates:

\begin{enumerate}
\item {$ \,y^{k+1} = \text{argmin}_y \quad ||y||_{\bar{w}}^2 - 2(y,\bar{x})_{\bar{w}} + \frac{\rho}{2} ||Dy-z^k + b^k||^2$}
\item {$ \,z^{k+1} = \text{argmin}_z  \quad  \lambda ||z||_{1,2} + \frac{\rho}{2} ||Dy^{k+1}-z + b^k||^2$}
\item {$ \, b^{k+1} = b^k + Dy^{k+1}-z^{k+1} $}
\end{enumerate}
where $\rho > 0$ is a parameter that can be interpreted as penalizing violations of the constraint. As such, lower values of $\rho$ tend to make the algorithm more adventurous, though the algorithm is known to converge to the optimum for any fixed value of $\rho > 0$. 

The minimization in the first step is convex, and the first order conditions yield a tridiagonal system for $y$. The tridiagonal matrix to be inverted is the same for all subsequent iterations, so only one inversion is necessary, which can be done in $\mathcal{O} (md)$ time. In the second step, $z$ decouples, and the resulting solution is given by block soft thresholding
\[ z_i^{k+1} =
\begin{cases}
v_i^k - \frac{\lambda}{\rho} \, \frac{v_i^k}{||v_i^k||} \quad & \text{if} \quad ||v_i^k|| > \frac{\lambda}{\rho}  \\
0 &  \text{else} \\
 \end{cases}  \]
 where we have let $v_i^k = (Dy^{k+1})_i + b_i^k$. We therefore see that ADMM applied to (\ref{constr}) is very fast.
 
Note that one only needs for the energy to decrease in this step for Algorithm \ref{alg_basic} to converge to a local minimum. This is typically achieved after one iteration of ADMM. In such cases few iterations may be appropriate, as finer precision typically gets lost once projections are updated. On the other hand, the projection step is more expensive, requiring $\mathcal{O}(nmd)$ operations to compute exactly. It may be worthwhile to investigate how to optimize alternating these steps, as well as more efficient methods for updating projections especially when changes in $y$ are small. In our implementation we exactly recompute all projections, and if the resulting change in energy is small, we minimize (\ref{adp_discrete}) to a higher degree of precision (apply more iterations of ADMM before again recomputing projections). 

%For our purposes though, we apply the above iterations until the energy decreases, and then exactly recompute all projections. 

\subsection{Approach to minimizing \ref{mppc}}
\label{sec_mppc_numerics}
We now discuss how we perform steps that decrease the energy of the modified functional \eqref{mppc}. 
We allow $y = y_1,...,y_m$ to consist of any number, $k$, of curves, and we denote them $y^1 = (y_1,...,y_{m_1}), \,  y^2 = (y_{m_1+1},...,y_{m_1+m_2}),..., \, y^k = (y_{m-m_{k}+1},...,y_m)$, where $m_1+m_2+...+m_k = m$. The indexes of the curve ends are $s_c = \sum_{j=1}^c m_j$ for  $c=1,...,k$, and we set $s_0=0$.
The discrete of form of \eqref{mppc} can then be written as 
\begin{equation}
\label{mppc_discrete}
  \sum_{j=1}^m \sum_{i\in I_j} w_i |x_i-y_j|^2 + \lambda_1 \sum_{c=0}^{k-1} \sum_{j=1}^{m_{c+1}} |y_{s_c+j+1}-y_{s_c+j}| + \lambda_1 \lambda_2 (k-1).
\end{equation}
Our approach to (locally) minimizing the problem over $y,\, k,\, m_1,...,m_k$ is to split the functional into parts that are decreased over different variables. Keeping $k,\, m_1,...,m_k$ constant and minimizing over $y_1,...,y_m$ we can decrease \eqref{mppc_discrete} by simply applying step 1 and step 2 of Algorithm \ref{alg_basic} to each curve $y^i, \, i=1,...,k$ (note that step 2 can be run in parallel). To minimize over $k,\, m_1,...,m_k$ we introduce topological routines below that disconnect, connect, add, and remove curves based on the resulting change in energy. 

\subsubsection{Disconnecting and connecting curves} \label{sec:disccon}
Here we describe how to perform energy decreasing steps by connecting and disconnecting curves. We first examine the energy contribution of an edge $\{i,i'\} := [y_i,y_{i'}]$. To do so we compare the energies corresponding to whether or not the given edge exists. 
It is straightforward to check that the energy contribution of the edge $\{i,i'\}$ with respect to the continuum functional \eqref{mppc} is
$$\Delta E_{i,i'} := \lambda_1 |y_{i'}-y_i| - \lambda_1 \lambda_2 - \sum_{j \in I_{i,i'}} w_j \min \left( |y_i - \Pi_{i,i'}(x_j)|,|y_{i'} - \Pi_{i,i'}(x_j)| \right)^2$$
where $I_{i,i'}$ is the set of data points projecting to the edge $\{i,i'\}$, and $\Pi_{i,i'}$ is the orthogonal projection onto edge $\{i,i'\}$. Our connecting and disconnecting routines will be based on the sign of $\Delta E_{i,i'}$.
We note that above criterion is based on the variation of the continuum functional rather than its discretization 
\eqref{mppc_discrete}, in which projections to the vertices only (not edges) are considered.
Our slight deviation here is motivated by providing a stable criterion that 
%would not change if we had a finer discretization
is invariant to further discretizations of the line segment $[y_i,y_{i'}]$.
 % which is subject to change as explained later in sections \ref{reparam}, \ref{res_crit}. 
While we use the discrete functional to simplify computations in approximating the optimal fitting of curves, we will connect and disconnect curves based on the continuum energy \eqref{mppc}. 

We first discuss disconnecting. We compute the energy contribution for each existing edge and if $\Delta E_{i,i'} < 0 $, then we remove edge $\{i,i'\}$. Note this condition can only be true if the length of the edge is at least $\lambda_2$. It may happen that all edge lengths are less than $\lambda_2$, but that the energy may be decreased by removing a sequence of edges, whose total length is greater than $\lambda_2$. Thus, in addition to checking single edges, we implement an analogous check for sequences of edges. The energy contribution of a sequence of $k$ edges $\{i,i+1\},\{i+1,i+2\},...,\{i+k-1,i+k\}$ (including the corresponding interior vertices $y_{i+1},...,y_{i+k-1}$) is given by
\begin{align*}
 \Delta E_{i:i+k}  := \, & \lambda_1 \left ( \sum_{l=0}^{k-1} |y_{i+l+1}-y_{i+l}| -\lambda_2 \right ) \\
 & + \sum_{l=0}^{k-1} \sum_{j \in I_{i+l,i+l+1}} w_j \left ( \left (x_j - \Pi_{i+l,i+l+1}(x_j) \right )^2 - \left (\min\{|x_j - y_i|, |x_j - y_{i+k}|\}\right )^2 \right). 
 \end{align*}
%where $\Pi_{-\{i+l,i+l+1\}}$ is the projection onto edge $[y_{i+l}, y_{i+l+1}]$. 
The routine for checking such edge sequences is outlined in Algorithm \ref{alg_cutting}.

\begin{algorithm}%[tb]
   \caption{Removing appropriate edge sequences}
   \label{alg_cutting}
\begin{algorithmic}
   \STATE {\bfseries Input:} data $x_1,...x_n$, weights $w_1,...,w_n$, connected curve $y_1,...,y_m$, projections $I$, $\lambda_1,\lambda_2 > 0$
   \STATE  set $i = 1$
   \REPEAT
   \STATE  set $k=1$, $len = |y_{i+1}-y_i|$
   \REPEAT
   \STATE increment $k = k+1$, $len = len + |y_{i+k}-y_{i+k-1}|$
   \UNTIL {$len > \lambda_2$ (or $i+k =m$, in which case break)}
   \STATE compute $\Delta E_{i:i+k}$
   \IF {$\Delta E_{i:i+k} < 0$}
   \STATE remove edge sequence $\{i,i+1\},\{i+1,i+2\},...,\{i+k-1,i+k\}$
   \STATE advance $i = i+k-1$
   \ENDIF
   \STATE increment $i = i+1$
   \UNTIL{$i>m-1$}
\end{algorithmic}
\end{algorithm}

Connecting is again based on the energy contribution of potential new edges. We use a greedy approach to adding the edges. That is, we compute $\Delta E_{i,i'}$ for each potential edge $\{i,i'\}$, and add them in ascending order, connecting curves until no admissible energy-decreasing edges exist. We note that finding the globally optimal connections is essentially a traveling salesman problem, which is NP-hard.  More sophisticated algorithms could be used here, but the greedy search is simple and has satisfactory performance.

\subsubsection{Management of singletons:} \label{sec:single}
Here we describe the procedures for topological changes via adding and removing components of the multiple curves. This is achieved by adding singletons (curves whose range is just a single point in $\RR^d$), growing them into curves, and by removing singletons. Even if one is only interested in recovering one-dimensional structures, singletons may play a vital role. 
In particular, any low-density regions of the data (background noise or outliers) can often be represented by singletons in a minimizer of \eqref{mppc}, allowing the curves to be much less affected in approximating the underlying one-dimensional structure.
%Singletons can represent low-density regions of data, and divert the attraction of curves away from outliers and background clutter. This is especially important when low-density regions exist far from the one-dimensional structures. 

Below we provide effective routines for energy-decreasing transitions between configurations involving singletons.
For checking whether (and where) singletons should be added, we examine each point $y_i$ individually. If $y_i$ is itself not a singleton, we compute the expected change in energy resulting from  disconnecting $y_i$ from its curve, placing it at the mean $\bar{x}_i$ of the data that project to it, and reconnecting the neighbors of $y_i$, so the number of components only increases by one. The change in the fidelity term will be exactly \nc $-\bar{w}_i(\bar{x}_i-y_i)^2$, where $\bar{w}_i = \sum_{j \in I_j}w_j$ is the total mass projecting to $y_i$. Thus we add a singleton when 
$$ \lambda_1 \lambda_2 < \bar{w}_i(\bar{x}_i-y_i)^2 + \lambda_1 \left (|y_i-y_{\max(1,i-1)}| + |y_i-y_{\min(m,i+1)}| - |y_{\max(1,i-1)}-y_{\min(m,i+1)}| \right).$$

 If $y_i$ is itself a singleton, then one cannot exactly compute the change in the energy due to adding another singleton in its neighborhood without knowing the optimal positions of both singletons. We restrict our attention to the data which project onto $y_i$, and note that if those points are the only ones that project to the new singleton, then adding the singleton may be advantageous only if the fidelity term associated to $y_i$ is greater than $\lambda_1 \lambda_2$. If that holds, we perturb $y_i$ in the direction of one of its data points, place a new singleton opposite to $y_i$ with respect to its original position, and apply a few iterations of Lloyd's k-means algorithm (with $k=2$) to the data points that projected to $y_i$.
% (one iteration consists of computing the projections, and then setting the location of the points to the mean of the data projecting to them). 
We keep the two new points if and only if the energy decreases below that of the starting configuration with only $y_i$.

A singleton $y_i$ gets removed if doing so decreases the energy. That is if
$$ \lambda_1 \lambda_2 > \sum_{j\in I_i} w_j \left( |x_j-y_i|^2 - d(x_j,y_{-i})^2 \right ) $$
where $d(x_j,y_{-i}) := \min \{|x_j-y_{i'}| : i' \in [m], \, i' \neq i \}$.

Since singletons are represented by just a single point and cannot grow by themselves, we also check whether transitioning from singleton to short curve is advantageous. To do so we enforce that the average projection distance $\tilde{d}_i$ to a singleton $y_i$ is less than $\sqrt{\lambda_1/\tilde{\alpha}}$, which represents the expected spatial resolution, where $\tilde{\alpha} = \bar{w}_i / (4 \tilde{d}_i)$ is an approximation to the potential linear density. Thus we add a neighboring point to $y_i$ if  
$$\tilde{d}_i = \sum_{j\in I_i} w_j |x_j-y_i| > \lambda_1/ {\bar{w}_i}. $$
Since this is based on an approximation, we also explicitly compute the posterior energy to make sure that it has indeed decreased, and only in this case keep the change. 

Note that for each singleton $y_j$, minimizing the discrete energy $\eqref{adp_discrete}$ with projections fixed corresponds to placing $y_j$ at its center of projected mass ${\bar{w}_i}$. Hence for singletons Algorithm \ref{alg_basic} reduces to Lloyd's k-means algorithm. 

In summary, we have fast and simple ways to perform energy decreasing steps involving the $\lambda_2$ term of the functional. Even when minimizers are expected to be connected, performing these steps may change the topological structure of the curve, keeping it in higher density regions of the data, and consequently evading several potential local minima of the original functional (\ref{ppc}).

\subsection{Re-parametrization of $y$}
\label{reparam}
%In addition to addressing local minima regarding the global structure of the curve, it is also important to give consideration to potential local minima that occur on a smaller scale -- namely, the discrete parametrization of the curve.
In applying the algorithm described thus far, it may, and often does, occur that some regions of $y$ are represented with fewer points $y_i$ than others, even if an equal amount of data are projected to those regions.
That is, there is nothing that forces the nodes $y_i$ to be well spaced along the discretized curve. 
%This is made possible by that fact that each $y_i$ only sees the data that projects to itself. 
To address this, we introduce  criteria that $l_i \bar{w_i} $ be roughly constant for $i=1,...,m$, where $l_i =\frac{1}{2} \sum \{ |y_i-y_j| : j \in \{i-1,i+1\} \cap [1,m] \}$ and $\bar{w_i}$ is the total weight of points projecting to $y_i$. This condition is motivated  by finding for fixed $m$ the optimal spacing of $y_i$'s that minimizes the fidelity term of the discrete energy \eqref{mppc_discrete}, under the assumption that the data are  distributed with slowly changing density  in a rectangular tube around straight line $y$.

\subsection{Criteria for well-resolved curves}
\label{res_crit}
Here we discuss criteria for when a curve can be considered well-resolved with regard to the number of points $m$ used to represent it. One would like to have an idea of what conditions give an acceptable degree of resolution, without requiring $m$ too large and significantly increasing computational time. We suggest two such conditions. 

One is related to the objective of obtaining an accurate topological representation of the minimizer, specifically the number of components. In order to have confidence in recovering components at a scale $\lambda_2$, the spacing between consecutive points on a discretized curve should be of the same scale. Thus we impose that the average of the edge lengths is at most $\frac{\lambda_2}{2}$.

Another approach for determining the degree of resolution of a curve is to consider its curvature. One may calculate the average turning angle and desire that it be less than some value (e.g. $\frac{\pi}{10}$). If $\lambda_2$ is not small enough, the first condition will not guarantee small turning angles, and so we include this criterion as optional in our implementation.
We note that in light of the possible lack of regularity of minimizers \cite{Sle14}, it would not be reasonable to limit the maximal possible turning angle. 

If either of the above criteria are not satisfied, we add more points to the curves where we expect they would decrease the discrete energy the most. Consistent with the criteria above for re-parametrization of the curves, we add points along the curve where $l_i \bar{w_i} $ is the largest. 

\subsection{Initialization} \label{sec:ini}
Finally, we discuss initialization. While the procedures described above enable the algorithm to evade many undesirable local minima, initialization can still impact the quality of the computed local minimizers. One of the  simple ideas that we found to work very well is to initialize using singletons.
We note that when the number of singletons is a fixed number $k$ then minimizing \eqref{mppc}
reduces to minimizing the $k$-means functional. Thus to position the singletons for fixed $k$ 
we use the standard Lloyd's algorithm to find the $k$-means cluster centers. We denote the \eqref{mppc} energy of the $k$-means centers by $E(k)$.
To determine a suitable value of $k$ we  perform a line search by starting with $k=1$ and  double it as long as $E(k)$ decreases, and then halve the intervals until  
a (local) minimizer $k$ is found.  We list the steps in Algorithm \ref{alg_in}. 
%\grn [Did we settle on this initialization method? There was a hybrid method doing only the first phase here, followed by running ADMM and the checkSingletons routine until convergence, which I remember being faster and leading to (slightly) lower energy in most cases. The differences were not drastic, but I will double check since it was a couple of weeks ago.] \red Slav, I would prefer us not to mix initialization  amd the subsequent algorithm. Thus I would prefer initialization that does not use the main algorithm and ADMM.  i\nc
 
\begin{algorithm}
   \caption{Initializing with singletons}
\begin{algorithmic}
    \STATE {\bfseries Input:} data $x_1,...x_n$, weights $w_1,...,w_n$, and $\lambda_1, \lambda_2 > 0$.
    \STATE Set $k=1/2$, $E(k) = + \infty$
   \REPEAT
   \STATE Let $k = 2k$
   \STATE Compute the  $k$-means centers  $C_k= \{c_1, \dots, c_k\}$, and energy $E(k) :=  E_\mu^{\lambda_1, \lambda_2}(C_k)$
   \UNTIL ${E(k) > E(k/2)}$
   \STATE Let $k' = \lfloor \frac{k}{2} \rfloor$, $k'' = k'$
   \REPEAT
   \STATE Let $k'' = \lfloor \frac{k''}{2} \rfloor$, $k = k' + k''$
   \STATE Compute the $k$-means centers  $C_k= \{c_1, \dots, c_k\}$, and energy $E(k) :=  E_\mu^{\lambda_1, \lambda_2}(C_k)$
   \IF {$E(k) < E(k')$}
   \STATE $k' = k$
   \ENDIF 
   \UNTIL $k''=1$
   \STATE {\bfseries Output:} $ y = C_{k'}$
\end{algorithmic}
\label{alg_in}
\end{algorithm}
\nc

\subsection{Overview} Thus far we have described all of the main pieces of our algorithm to compute local minimizers of \eqref{mppc}. Here we describe how we put these pieces together. Algorithm \ref{alg_basic}, which includes ADMM for decreasing the discrete energy \eqref{adp_discrete}, 
computes approximate local minimizers of \eqref{ppc}. To approximate local minimizers of \eqref{mppc}, we break up the minimization into separate parts. One consists of a ``local'' step that updates the placement of each curve, and is accomplished by running the ADMM step of Algorithm \ref{alg_basic} on each curve. On the other hand, the inclusion of routines to disconnect, connect, add, and remove curves allows us to perform energy-decreasing steps of \eqref{mppc} in a more global topological fashion.
 
%Hence to compute local minimizers of \eqref{mppc}, we perform the (potentially) topology-changing routines outlined in \ref{sec:disccon}, \ref{sec:single} on a regular basis throughout the steps of Algorithm \ref{alg_basic}. 
%The optimal frequency at which each of these is run is unknown to us, but we have found that running the connecting and disconnecting routines at different times every 10 iterations has worked well. In most cases running the routine that checks singletons at a comparable frequency is also suitable. However, when overfitting occurs and $\lambda_2$ is small enough, several singletons may get added before being suboptimally connected. Thus a more cautious approach could be warranted for the addition of singletons. Checking (and modifying if necessary) the parametrization of the curves meanwhile can be done more often, and in our current implementation we run the routine every 5 iterations. 

We provide an general outline for finding local minimizers of \eqref{mppc} in Algorithm \ref{alg_mppc}. The (potentially) topology-changing routines outlined in \ref{sec:disccon}, \ref{sec:single} are run on a regular basis throughout the steps of Algorithm \ref{alg_basic}. 
In particular we run them every $top{\_}period = 10$ iterations and we run the reparameterization of curves every $reparam{\_}period = 5$ iterations.
The performance for different values, as well as for different order of operations was similar. 
%\blue 
%Some justification though for the given ordering is as follows. In the case that the initialization is poor, disconnecting first allows to split up curves before getting further attracted to a local minimum. Checking singletons before connecting makes certain any unfavorable singletons will be removed, and reduces the risk of suboptimal connections. Finally, re-parametrizing after connecting ensures new edges will be more finely represented. 
%\nc

\begin{algorithm}%[tb]
   \caption{Computing local minimizer of \eqref{mppc} $\:$ [Main Loop]}
   \label{alg_mppc}
\begin{algorithmic}
   \STATE {\bfseries Input:} data $x_1,...x_n$, weights $w_1,...,w_n$, initial curve $y_1,...,y_m$, $\lambda_1, \lambda_2 > 0$.
   \STATE set iter = 0
   \REPEAT
   \STATE 1. iter = iter + 1
   \STATE 2. compute $I_1,...,I_m$, defined in \eqref{Ij}
   \STATE 3. run ADMM on non-singleton curves to decrease energy (\ref{mppc_discrete}) as described in Section \ref{ppcadmm} 
   \STATE 4. replace $y_j$ by the center of mass of data points projecting to $y_j$: $y_j = \bar{x}_j$
   \IF  {iter + 4 = 0 (mod top{\_}period)}
   \STATE remove appropriate edge sequences as described in Section \ref{sec:disccon} and Algorithm \ref{alg_cutting}
   \ELSIF  {iter + 3 = 0 (mod top{\_}period)}
   \STATE add or remove appropriate singletons as described in Section \ref{sec:single}
   \ELSIF  {iter + 2 = 0 (mod top{\_}period)}
   \STATE add appropriate connections as described in Section \ref{sec:disccon}
   \ELSIF  {iter + 1 = 0 (mod reparam{\_}period)}
   \STATE add points and re-parametrize the curves if needed as described in Sections \ref{reparam}, \ref{res_crit}
   \ENDIF 
   \UNTIL{convergence}
\end{algorithmic}
\end{algorithm}

% The initial segments are then expected to grow and eventually connect along the data (under appropriate choices for $\lambda_1, \lambda_2$). 

%%%%%%%%%%%%%%%%%%%%%%%%%%%%%%%%%%%%%%%%%%%%%%%%%%
\subsection{Further numerical examples} \label{sec:ne}

We present a couple of further computational examples which illustrate the behavior of the functionals and the algorithm. For some of the examples, we include comparisons with results from other approaches including the Subspace Constrained Mean Shift algorithm and diffusion maps. 
%In this section we also include some comparisons with 
\begin{example} \emph{Parabola.}
\label{parabola}
We begin with an example that illustrates the cutting and reconnecting mechanism used in the Algorithm \ref{alg_mppc} for finding minimizers of \eqref{mppc}. We use data that are uniformly distributed on the graph of the parabola $x = y^2$ for $y \in [-3,3]$ and set $\lambda_1=0.12$ and $\lambda_2=4/3$. For illustration, we first run the Algorithm \ref{alg_mppc} for minimizing \eqref{ppc} (the  same as main loop of Algorithm \ref{alg_mppc} without allowing any topological changes) starting  form a small perturbation of the line segment $[0,9] \times \{0\}$. The result is shown on Figure \ref{parab1}. We then turned on the cutting routine, described in Algorithm \ref{alg_cutting}. The segments to be cut are indicated on Figure \ref{parab1} as dashed lines. Figure  \ref{parab2} shows a subsequent configuration, after a few steps of ADMM relaxation, but prior to reconnecting. Edges that are about to be added in the reconnection step (described in Section \ref{sec:disccon}) are shown as dashed blue lines. 
\begin{figure}[!htb]
 \centering
 \subfigure[Local minimizer of \eqref{ppc} with edges to be cut indicated.]{\includegraphics[width=0.32 \textwidth]{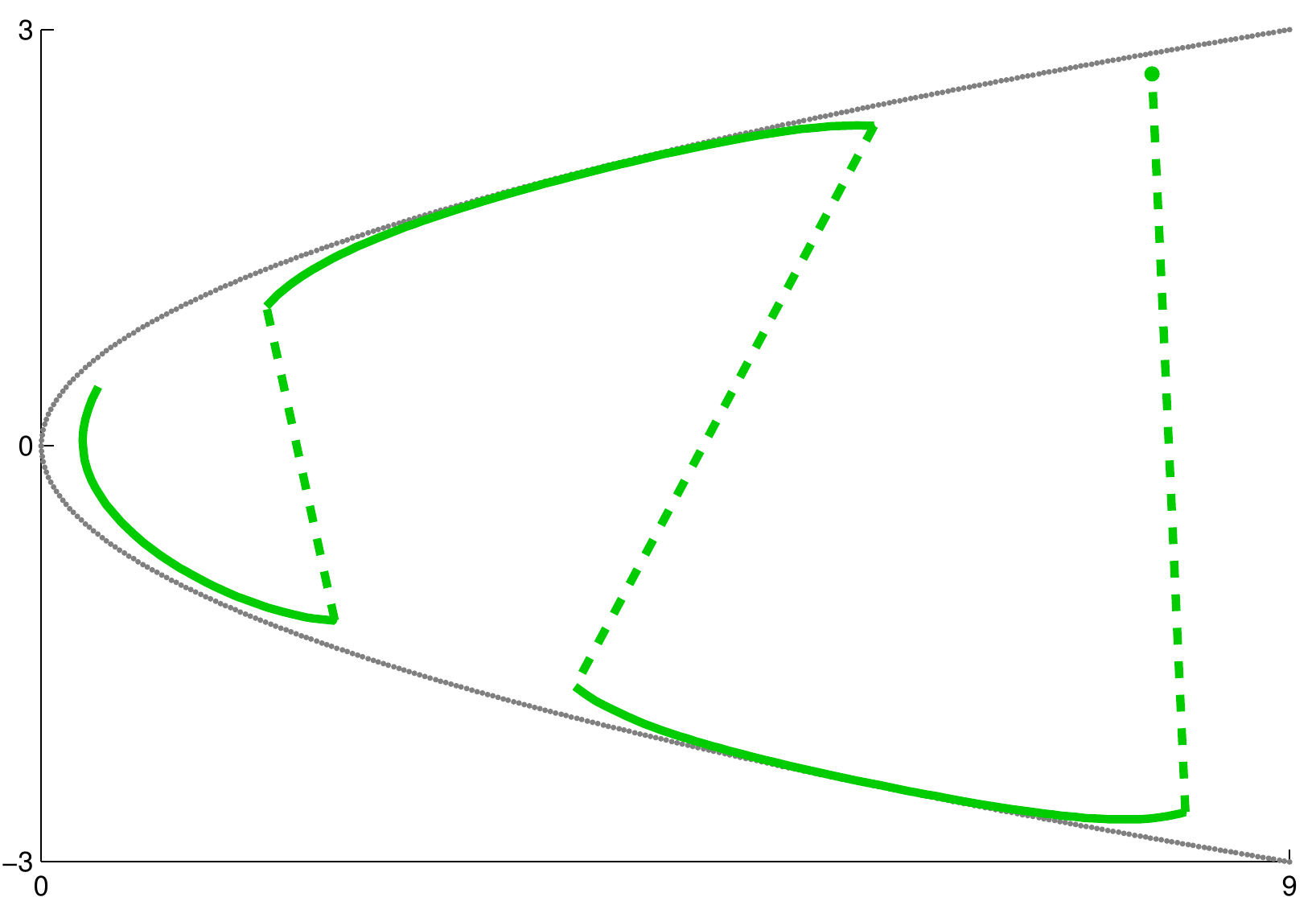} \label{parab1} }  
  %\hspace*{20pt} \nolinebreak
  \subfigure[Subsequent configuration, with edges to be added indicated.]{\includegraphics[width=0.32 \textwidth]{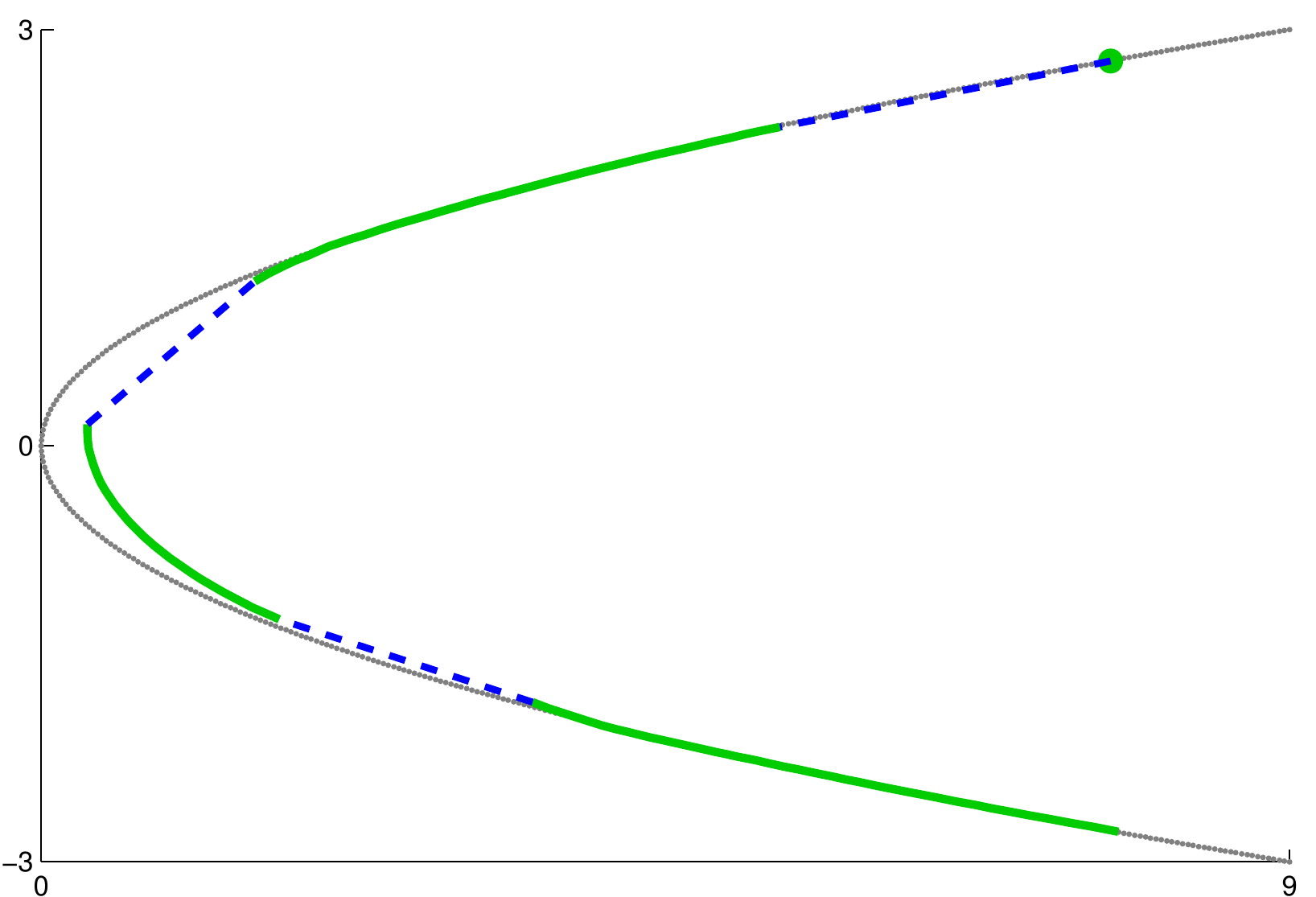}   \label{parab2}} 
  \subfigure[Minimizer of \eqref{mppc}]{\includegraphics[width=0.32 \textwidth]{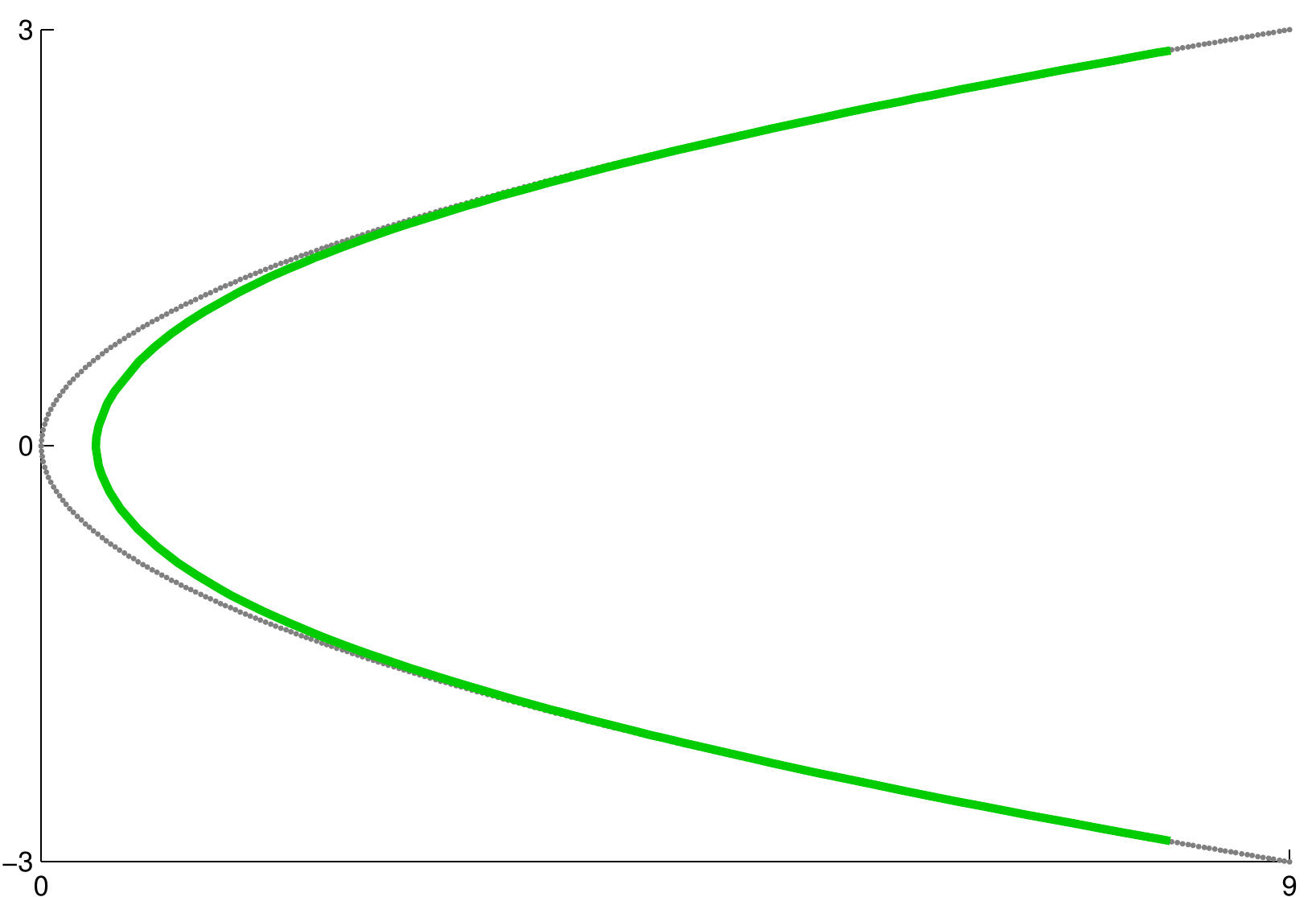}   \label{parab3}} 
\caption{}
\label{fig:parab}
\end{figure} 
\end{example}
\nc

\begin{example} \emph{Noisy spiral.}
\label{spiral}
Here we consider data generated as noisy samples of the spiral
$t \mapsto (t\cos(t),t\sin(t))$, $t \in [3,14]$, shown as a dashed line in Figure \ref{ex_6b}. 2000 points are drawn uniformly with respect to arc length along the spiral. 
For each of these points, noise drawn independently from the normal distribution $1.5 \, \mathcal{N}_2(0,1)$ is added. 
In Figure \ref{ex_6}, we show the results of algorithms for minimizing \eqref{ppc} and \eqref{mppc}. 
The initialization used for both experiments is a diagonal line corresponding to the fist principal component.
The descent for \eqref{ppc} does not allow for topological changes of the curve and subsequently gets attracted to a local minimum. 
Meanwhile, Algorithm \ref{alg_mppc} for minimizing \eqref{mppc} is able to recover the geometry of the data, via disconnecting and reconnecting the initial curve. 

\begin{figure}[!htb]
 \centering
 \subfigure[Local minimizer of \eqref{ppc}]{\includegraphics[width=0.44 \textwidth]{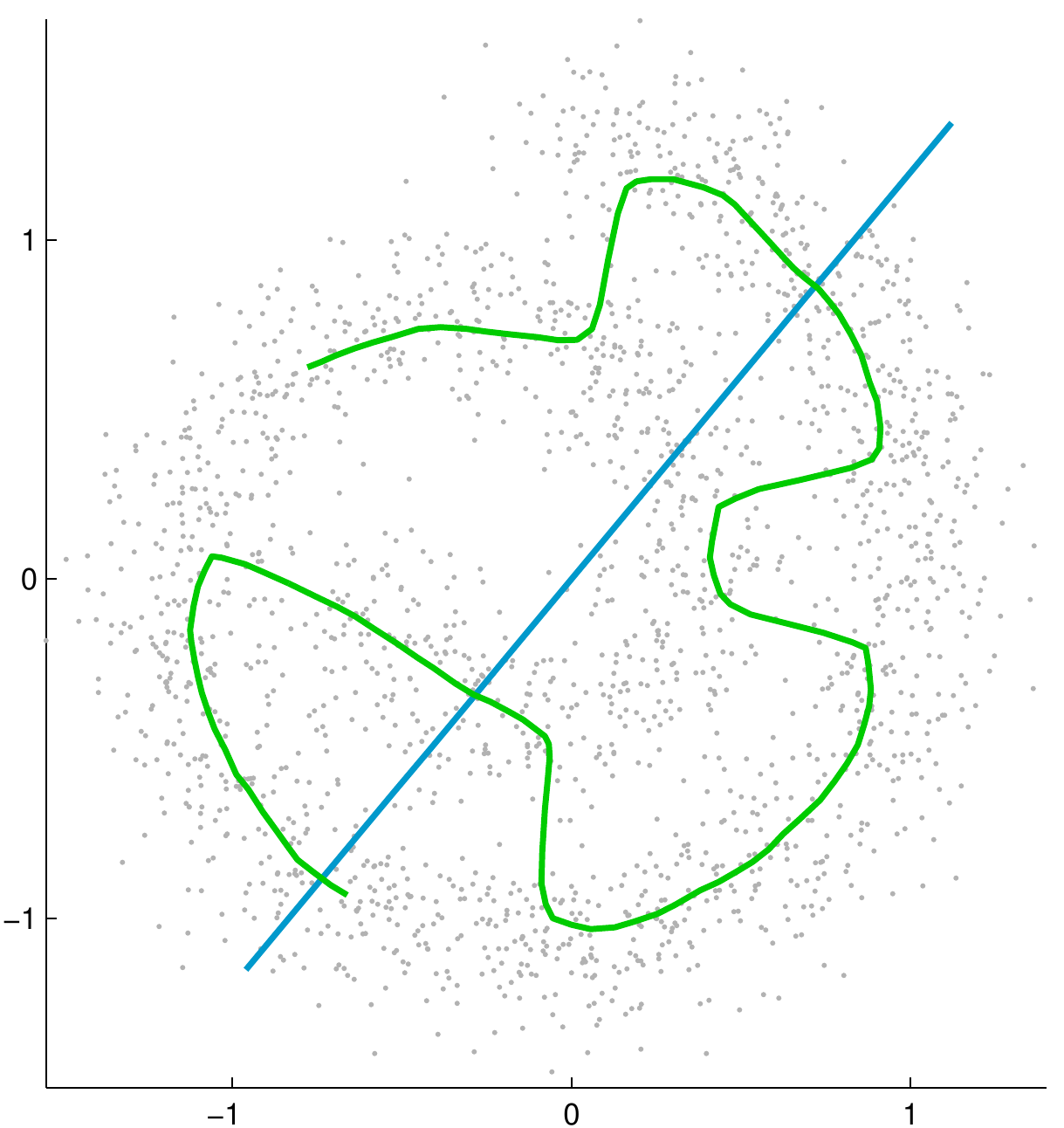} \label{ex_6a} }  
  \hspace*{20pt} \nolinebreak
  \subfigure[Local minimizer of \eqref{mppc}]{\includegraphics[width=0.44 \textwidth]{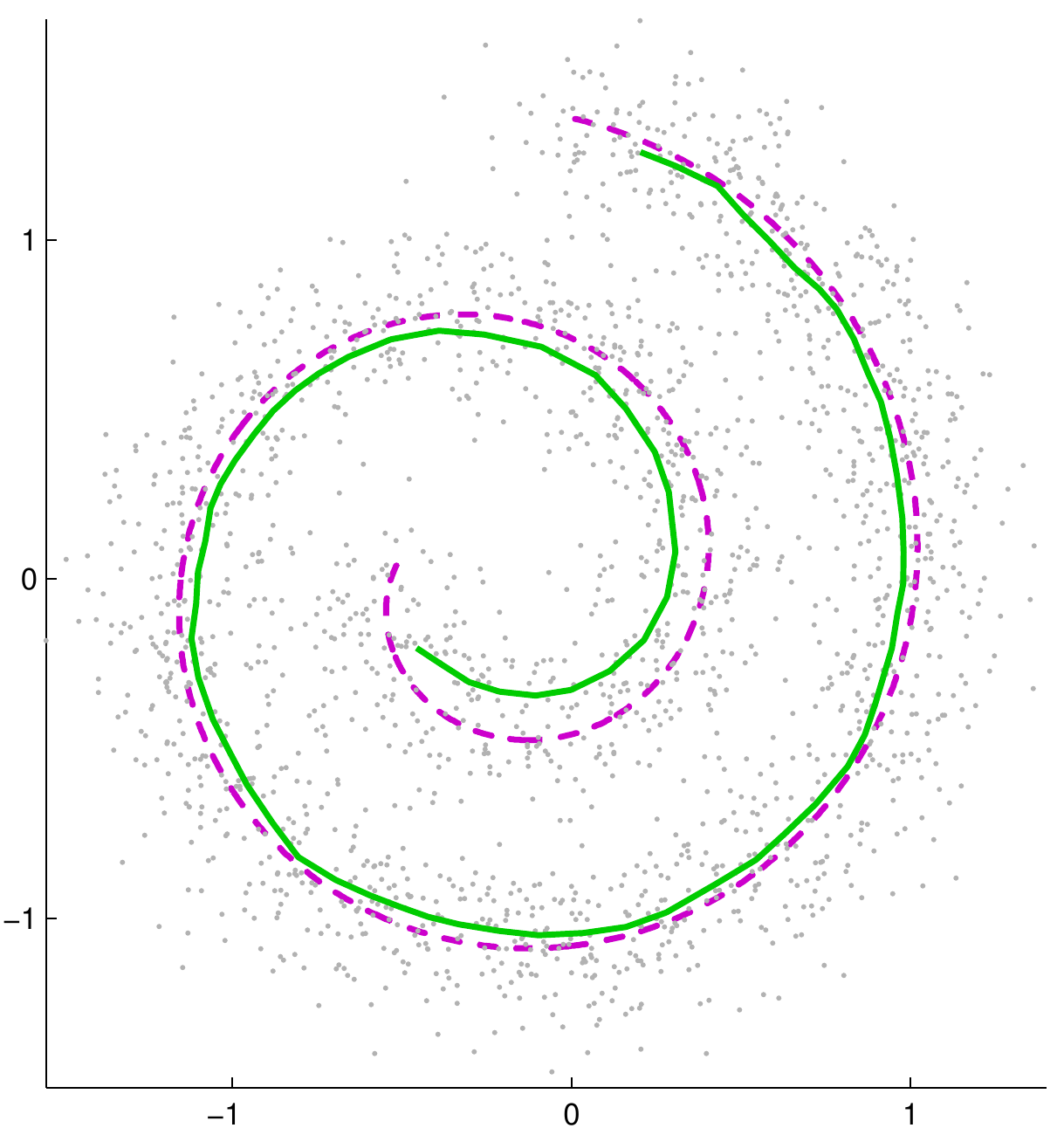}   \label{ex_6b}} 
\caption{Numerical results on data generated by a spiral (in purple) plus Gaussian noise. On the left, \eqref{ppc} is used to find the local minimizer given by the green curve, using the first principal component (blue) as initialization. On the right, \eqref{mppc} (with the same initialization) is minimized to find the green curve, using critical linear density $ \alpha^\star = .09$. In both cases $\lambda_1 = .01$.}
\label{ex_6}
\end{figure}

For this dataset we also include results of the Subspace Constrained Mean Shift (SCMS) algorithm \cite{oe11}, also studied in \cite{CGW15, chen15cosmic, genovese14} as means to find one dimensional structure in data. SCMS seeks to find the ridges (of an estimate) of the underlying probability density of the data.  The ridge set of a function $F : \R^d \to \R$ is defined as the set where $\nabla F$ is an eigenvector of the Hessian of $F$ and  the eigenvalues of all remaining eigenvectors are negative (the point is a local maximum along all orthogonal directions). In practice, given a random sample one uses a  kernel density estimator (KDE) t o approximate the probability distribution.  SCMS algorithm takes a set of points as input, and successively updates each point until it converges to a ridge point of the KDE of a specified bandwidth. The output is then a list of unordered points that approximate the ridge set. 
\begin{figure}[!htbp]
 \centering
% \subfigure[Output of SCMS]{\includegraphics[width=0.42 \textwidth]{scms1.pdf} \label{scms1} }  
  \subfigure[SCMS with bandwidth 0.15]{\includegraphics[width=0.42 \textwidth]{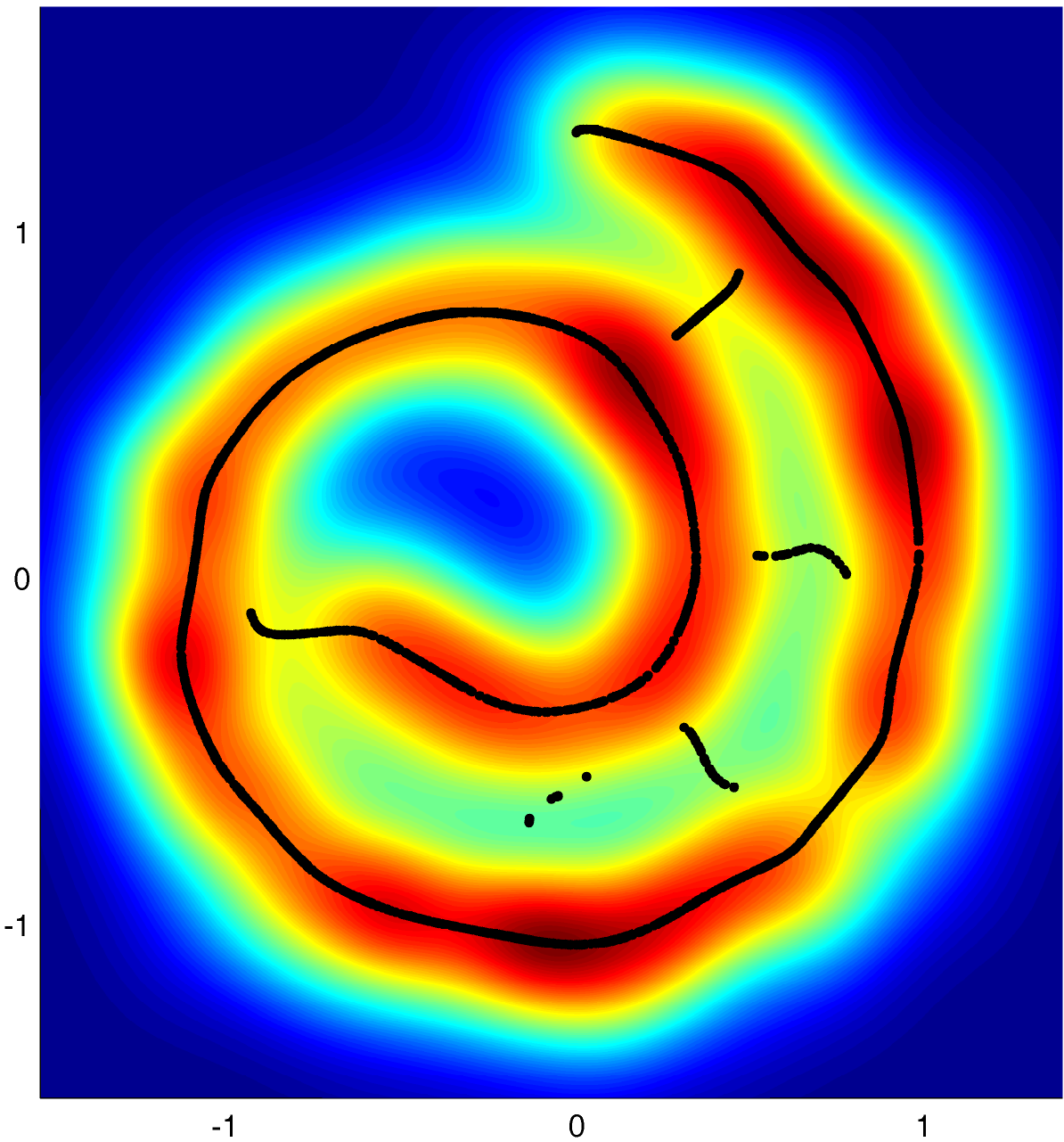}   \label{scmsh15}} 
  \hspace*{40pt} \nolinebreak
  \subfigure[SCMS with bandwidth 0.2]{\includegraphics[width=0.42 \textwidth]{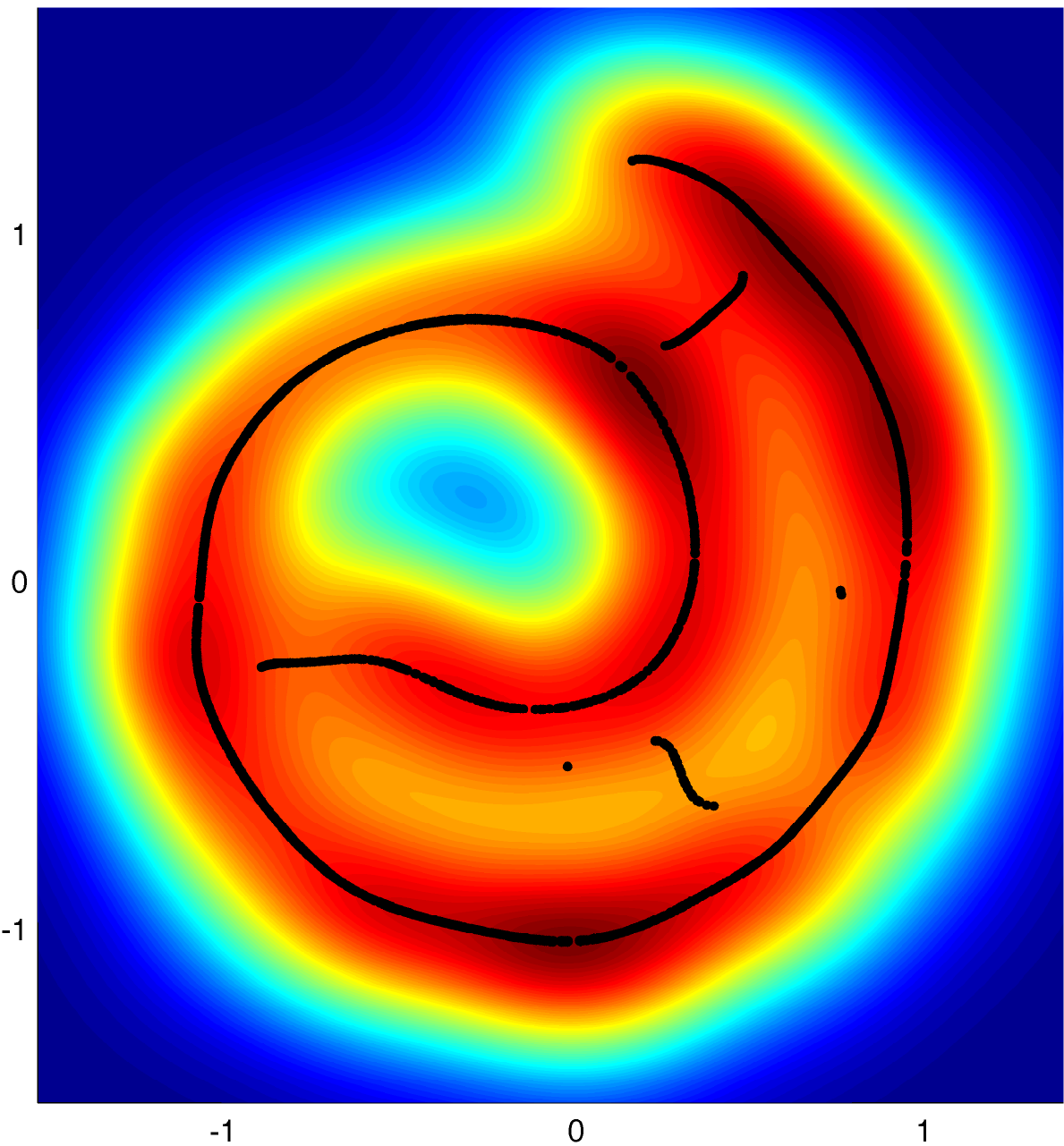}   \label{scmsh17}} 
\caption{There are more undesirable ridges for small bandwidths. Decreasing the bandwidth further can also result in the gaps in the desirable spiral filament. There are fewer undesirable ridges at higher bandwidth, but they have higher density. Increasing the bandwidth further introduces a significant bias of the main ridge, compared to the generating curve.}
\label{scms}
\end{figure} 
We apply SCMS using a Gaussian kernel density estimate (KDE) for two different bandwidth, 
which both give good results. The algorithm is initialized with
 2500 points on a 50x50 mesh in the range of the data. As shown in Figure \ref{scms}, the algorithm does output points approximating the underlying one-dimensional structure. 
 We note however that mathematically there is a number of ridges of KDE going between 
 the layers of the spiral. The SCMC algorithm captures those with high enough density and large enough "basin if attraction".  We note that as the kernel bandwidth increases, the number undesirable ridges decreases, but their intensity increases (the density at the remaining ridges is higher). Removing points on the mesh that have density below a given threshold has been suggested for noisy data \cite{chen15cosmic, genovese14}, and doing so can improve the results here by eliminating some of the undesirable ridges. However this introduces a parameter (density threshold) that needs to be chosen carefully (see appendix A of \cite{chen15cosmic}).
 %Nevertheless, the experiment illustrates that ridges alone do not always correspond to one-dimensional structure of the data, even when it is significantly smoothened out in the density estimation process. 
\end{example}

\begin{example} \emph{Noisy grid with background clutter.}
\label{grid}
In the following example we illustrate the robustness of the proposed approach to background noise. 
We use data in $\mathbb{R}^3$ with an underlying grid-like structure, shown in Figure \ref{ex_12}. The data consist of 2400 points generated by four intersecting lines with Gaussian noise, plus 2400 more points uniformly sampled from the background $[0,3]\times[0,3]\times[-.75,.75]$. 

Since the linear density of data in the background noise is less than that of the intersecting lines, the computed minimizer approximates the data in the background by isolated points (in green). For the parameters we used, this is predicted by the discussion of the density threshold in Section \ref{sec3}. By choosing $\lambda_1$ and $\lambda_2$ so that the critical density threshold $\alpha^*=  \left ( \frac{4}{3} \right )^{2}  \frac{\lambda_1}{\lambda_2^2} $ is between the linear density of the background noise and the linear density of the lines, the background noise will be represented by isolated points, which allows the curves to appropriately approximate the intersecting lines. 

We note that although the algorithm succeeds in approximating the one-dimensional structure of the data, it is not able to recover the intersections due to the simpler structure of configurations we consider in \eqref{mppc}. 
In such cases where our approach cannot identify the global topology, we presume it may be possible to use the obtained approximation as input for other approaches that aim to recover the topology of the data \cite{SiMeCa07}. 

\end{example} 
\begin{figure}[!htbp]
%\vskip 0.2in
\vspace{-14pt}
\begin{center}
\centerline{\includegraphics[width=0.84\textwidth]{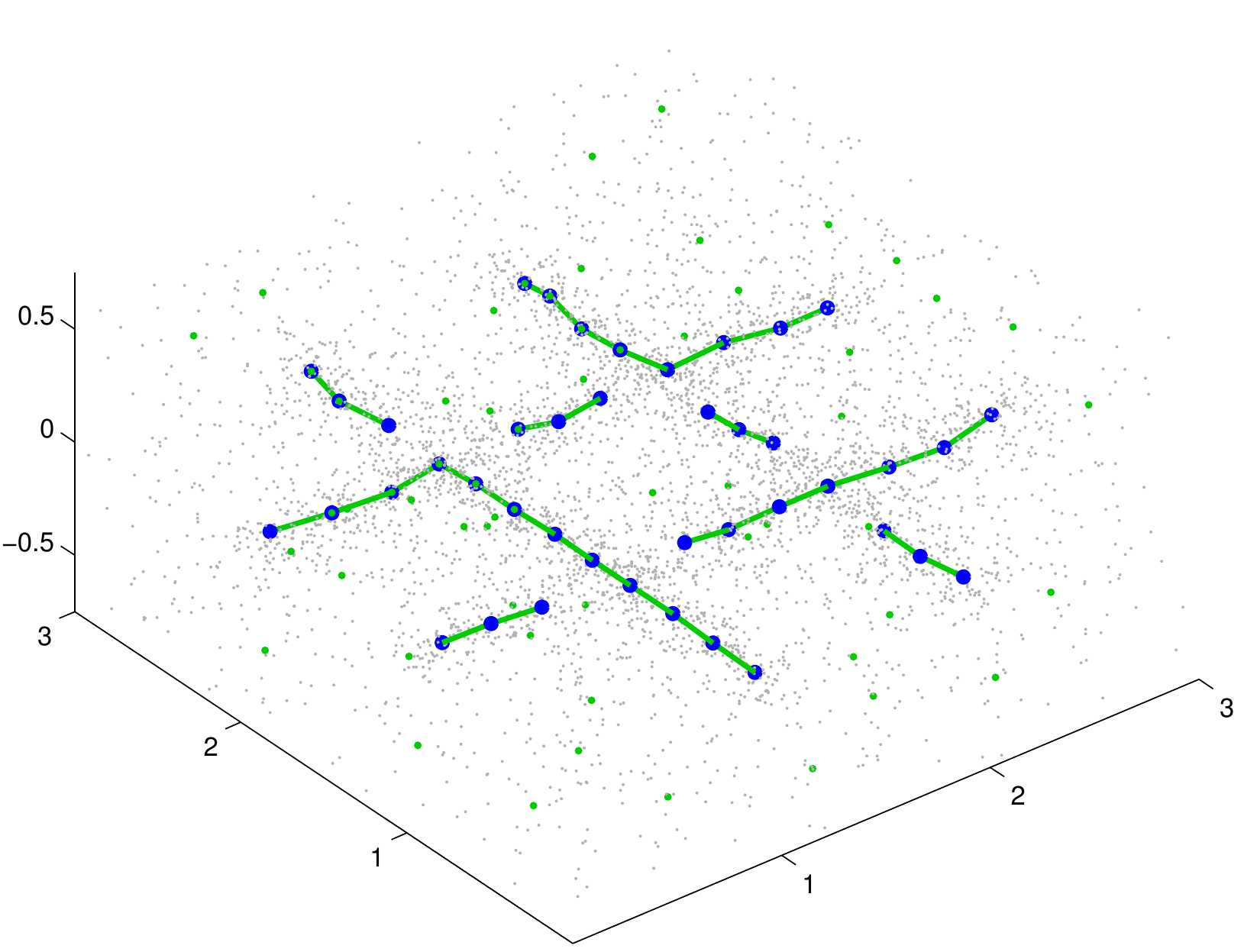}}
\caption{Computed minimizer on data generated from four intersecting lines forming a grid with Gaussian noise and background clutter in $\RR^3$, with $\lambda_1 = 7\times10^{-4}, \, \lambda_2 = 0.2$.
The minimizer consists of curves and isolated points (green). The larger blue dots represent the discretization (points $y_i$) of curves which are a part of the minimizer.  }
\label{ex_12}
\end{center}
%\vskip -0.2in
\label{fig:window}
\end{figure} 

\begin{example} \label{zebrafish} \emph{Zebrafish embryo images.} 
Here we demonstrate performance of the algorithm on a high-dimensional dataset that consists of grayscale images. 

In \cite{dsilva15} Dsilva et al. develop a technique for finding the temporal order of still images of a developmental process. They consider the problem where both the time ordering and the angular orientation of the images are unknown. 
To be able to handle both variables simultaneously they use vector diffusion maps \cite{SingerWu}.  One of the tests they performed to validate their approach was on images taken from a time-lapse movie that captures zebrafish embryogenesis [\url{https://zfin.org/zf_info/movies/Zebrafish.mov}] (Karlstrom and Kane \cite{k96}). 

In this case the angle of rotation is fixed, recovering the temporal order can be done using diffusion maps \cite{CoiLaf06} alone, see Figure \ref{zf_color}. Here we demonstrate that these images can also be ordered using our method.

As in \cite{dsilva15}, we apply our algorithm to 120 consecutive frames (roughly corresponding to seconds 6-17 in the movie) of 100x100 pixels in order to test how well it can recover the development trajectory. Thus each image is represented as a point in $\mathbb{R}^{10,000}$. We note that there is almost no noise in the dataset, but emphasize that the goal here is to recover a single curve passing through data whose true order is not provided to the algorithm.

After normalizing the data, we run our algorithm with parameters $\lambda_1 = 10^{-3}$ and $\lambda_2 = 2$. The low value for $\lambda_1$ is appropriate given that there is virtually no noise. The high value of $\lambda_2$ ensures that a single curve is found, and so the functional \eqref{ppc} is also being minimized.
 Our algorithm outputs a curve that correctly ranks all of the original images. Figure \ref{zf_images} shows a random sample of the images used, along with their found true ordering. In Figure \ref{zf_curve}, we visualize the found curve in $\mathbb{R}^3$ using the first three principal components. 
\begin{figure}
\begin{center}
\begin{tabular}{@{}cccccccccc@{}}
\includegraphics[width = .08\textwidth]{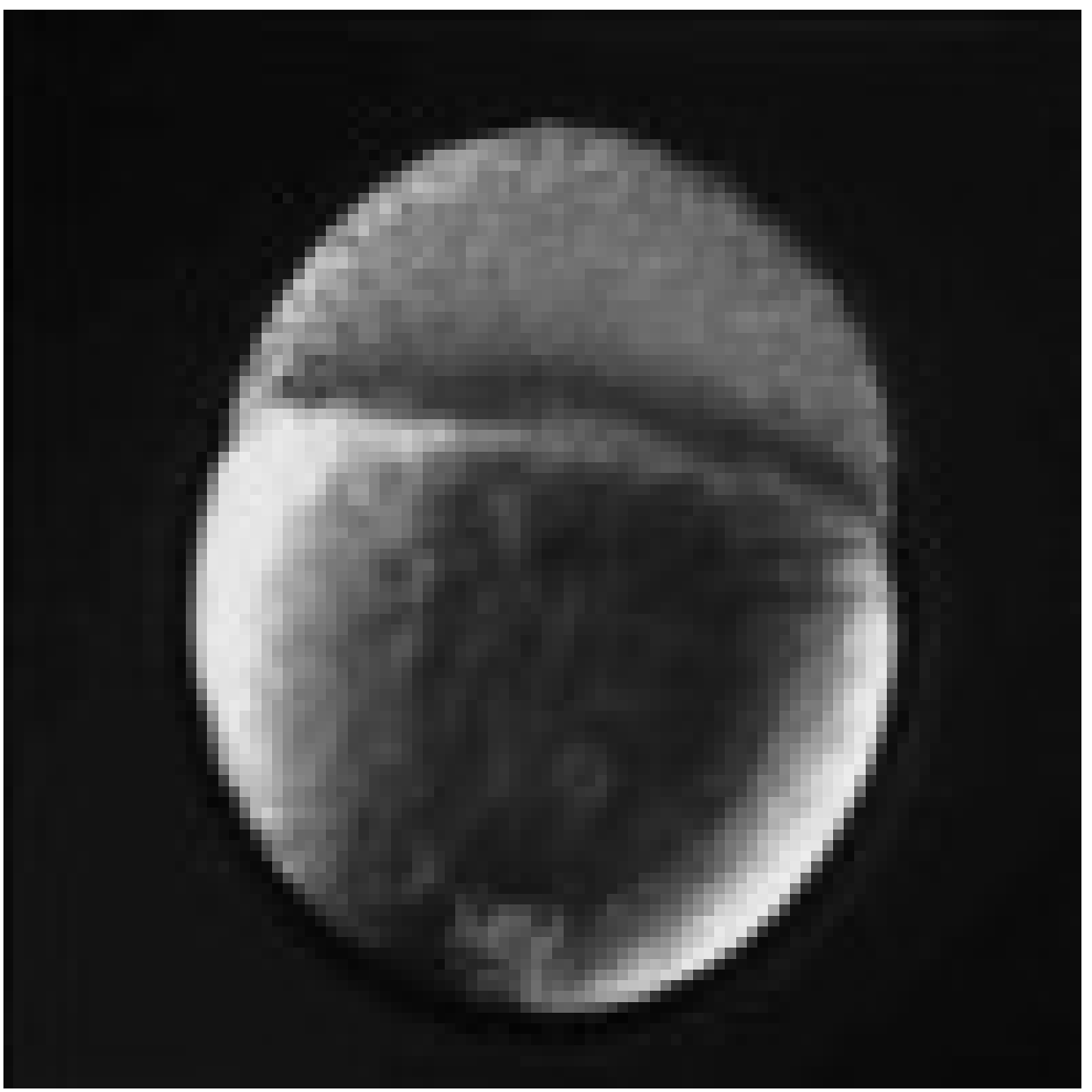} &
\includegraphics[width = .08\textwidth]{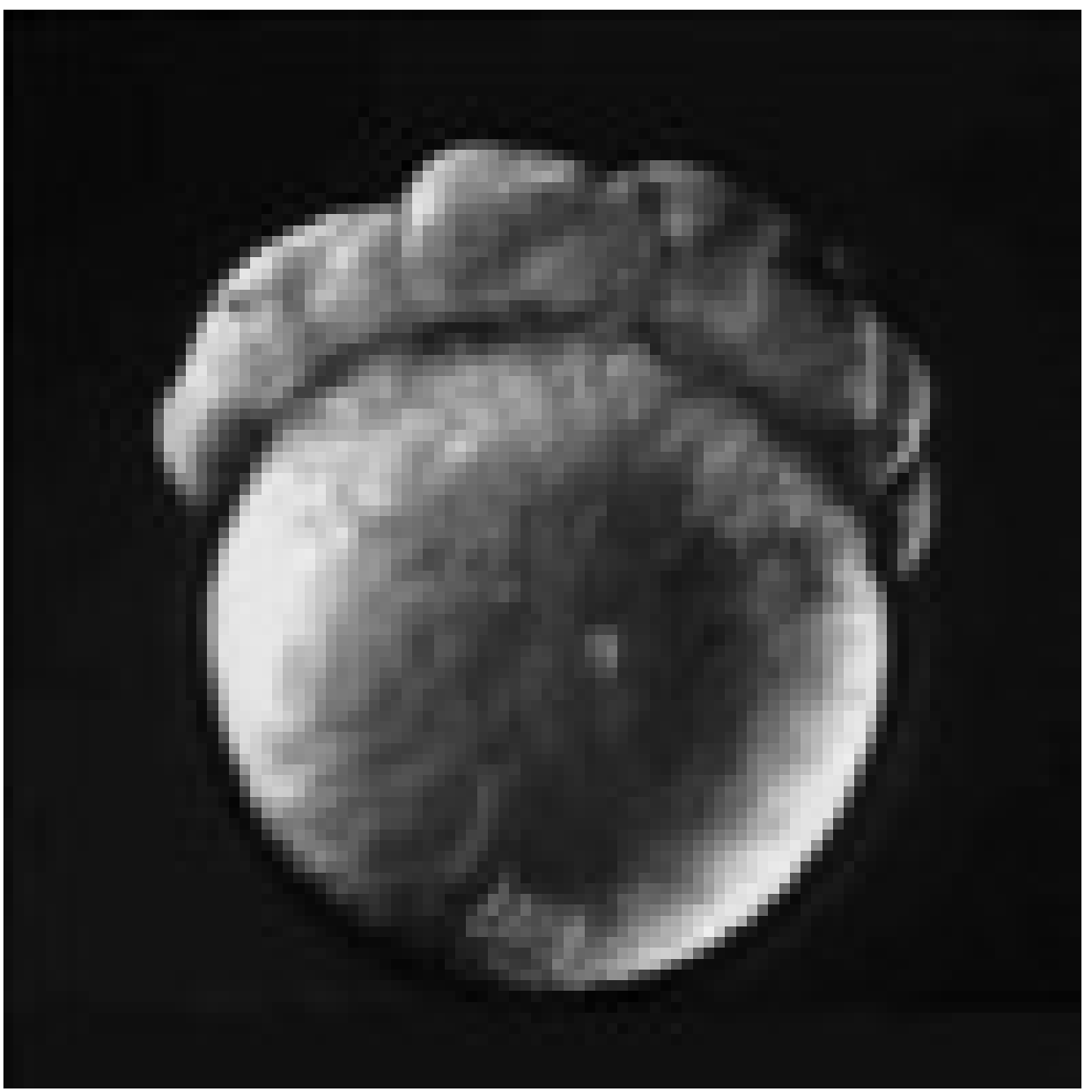} & 
\includegraphics[width = .08\textwidth]{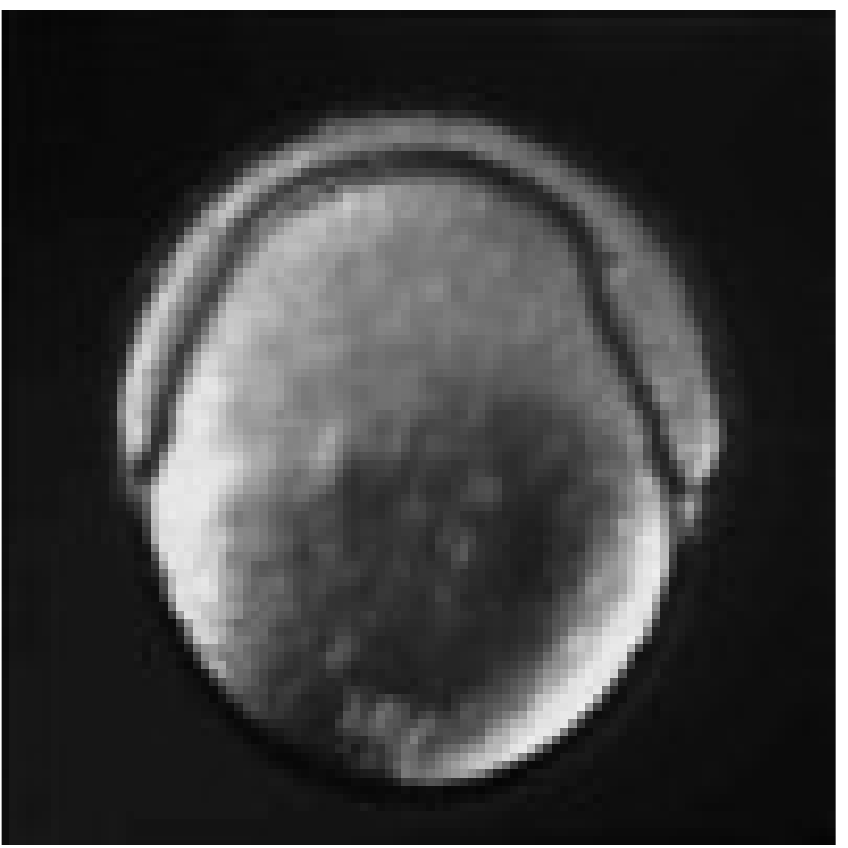} &
\includegraphics[width = .08\textwidth]{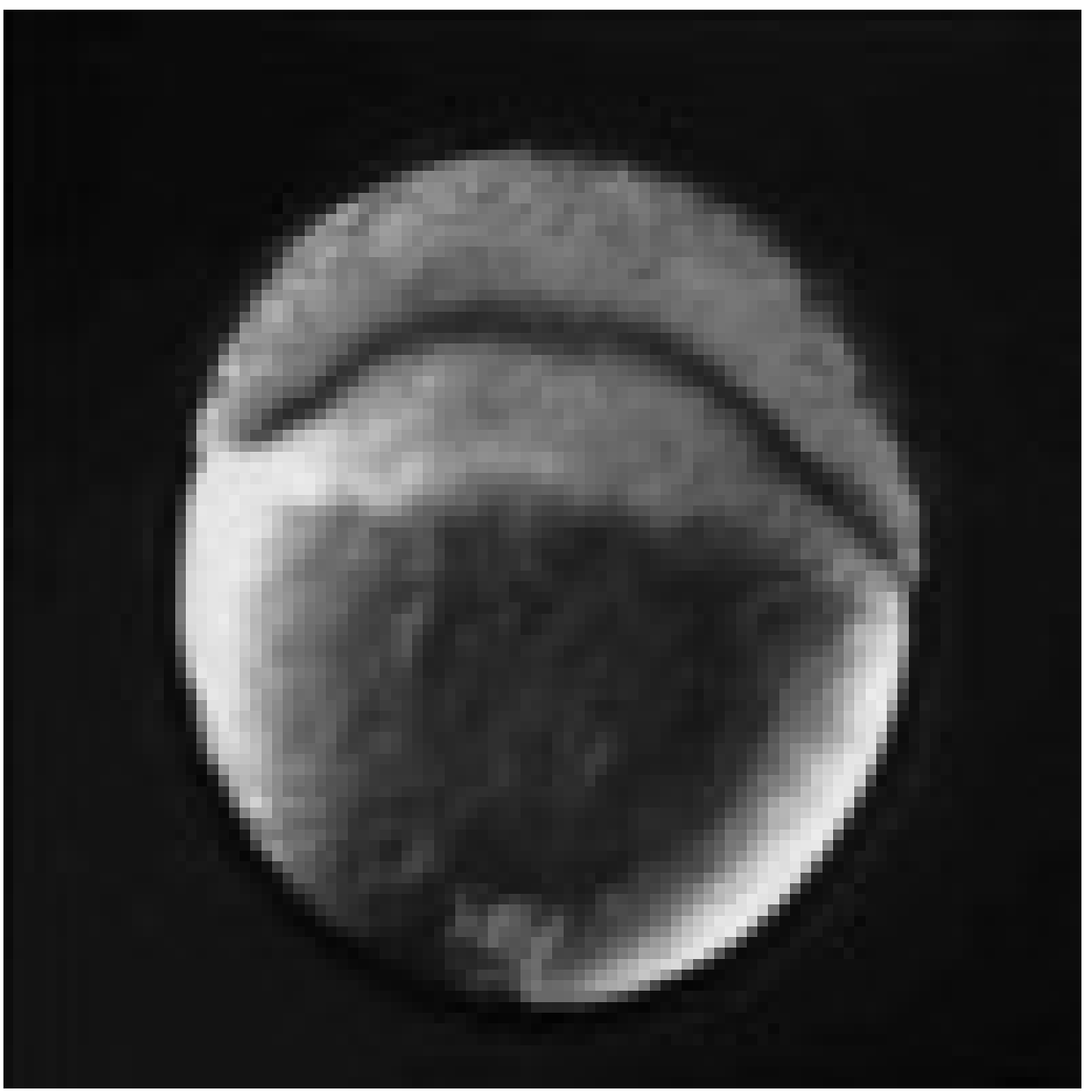} & 
\includegraphics[width = .08\textwidth]{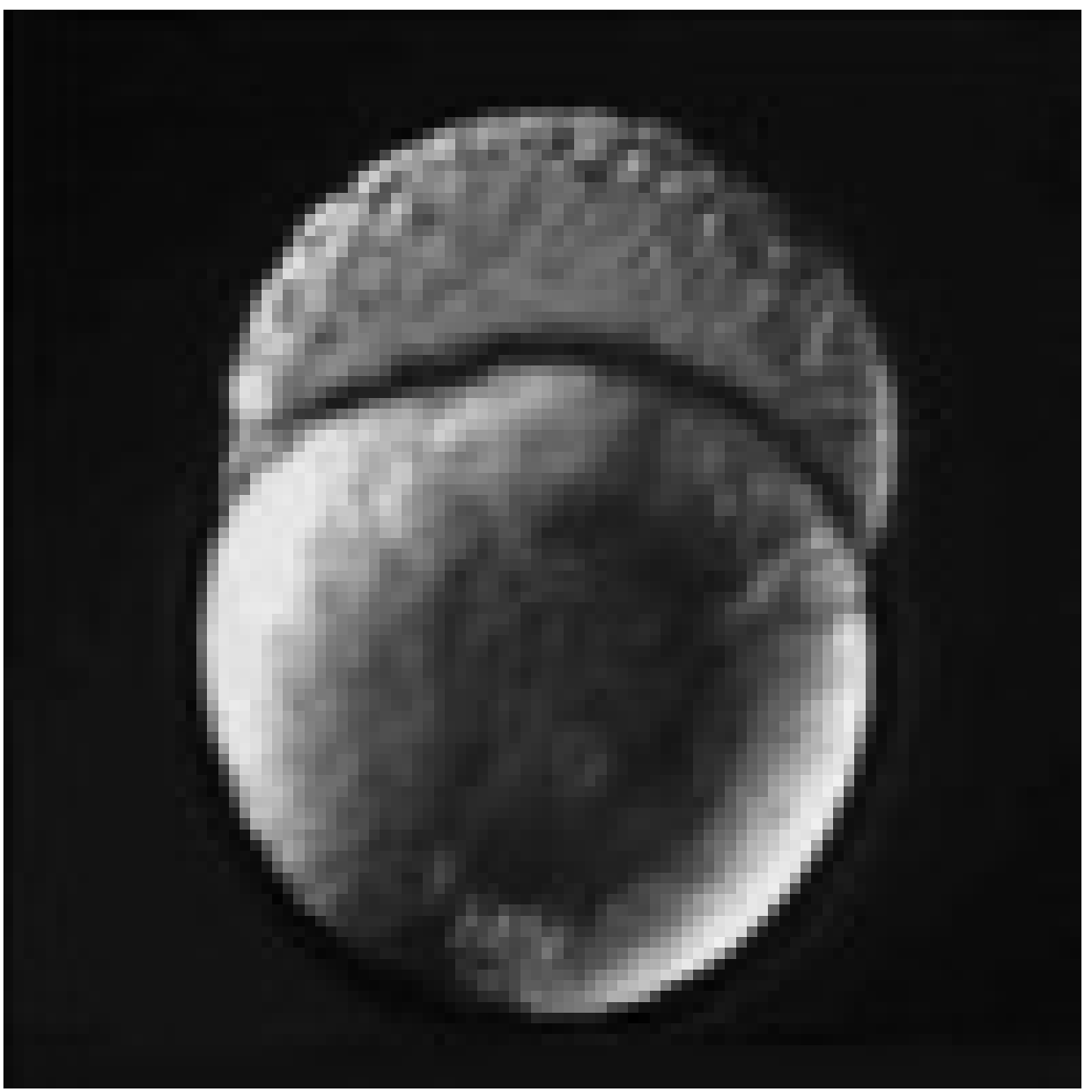} &
\includegraphics[width = .08\textwidth]{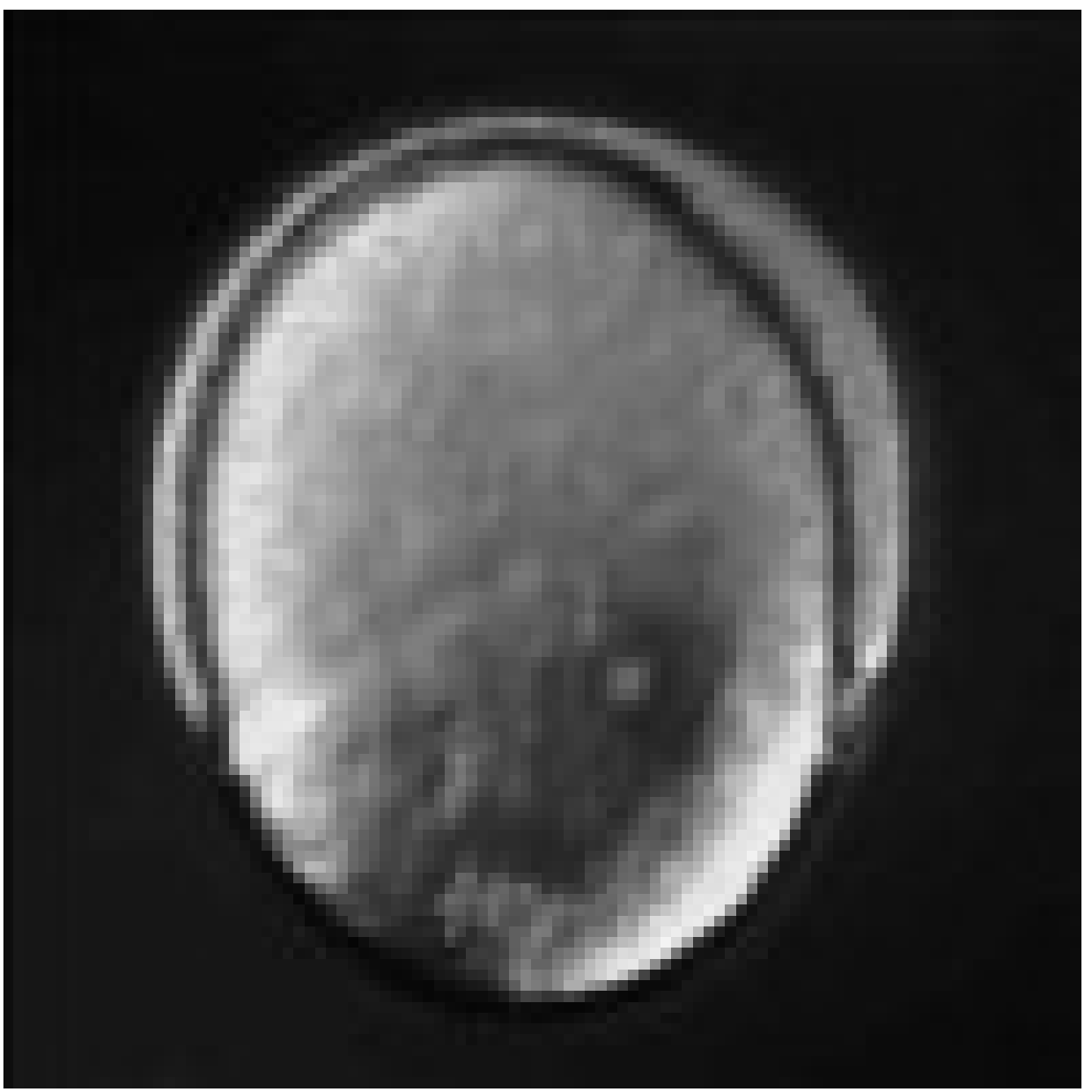} &
\includegraphics[width = .08\textwidth]{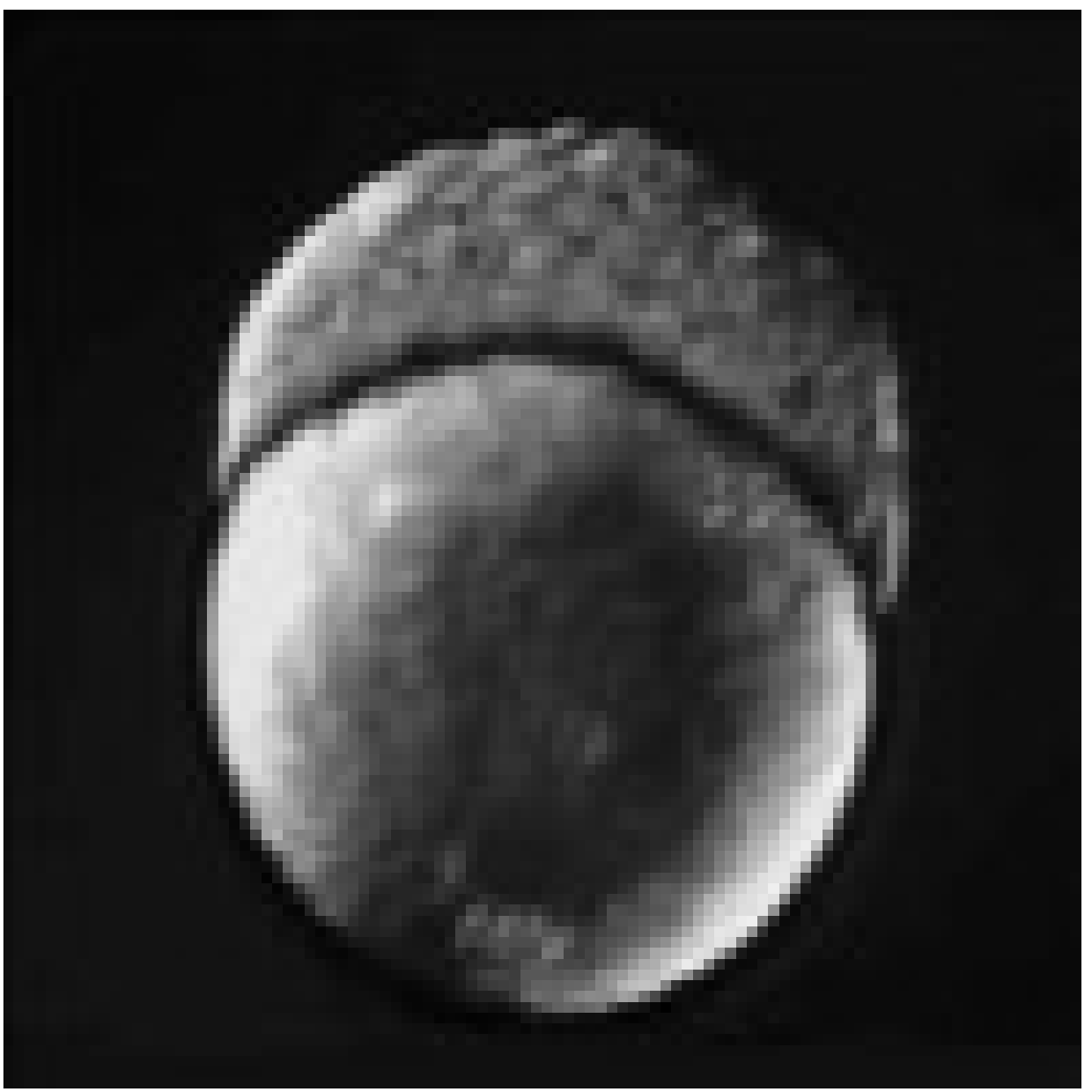} &
\includegraphics[width = .08\textwidth]{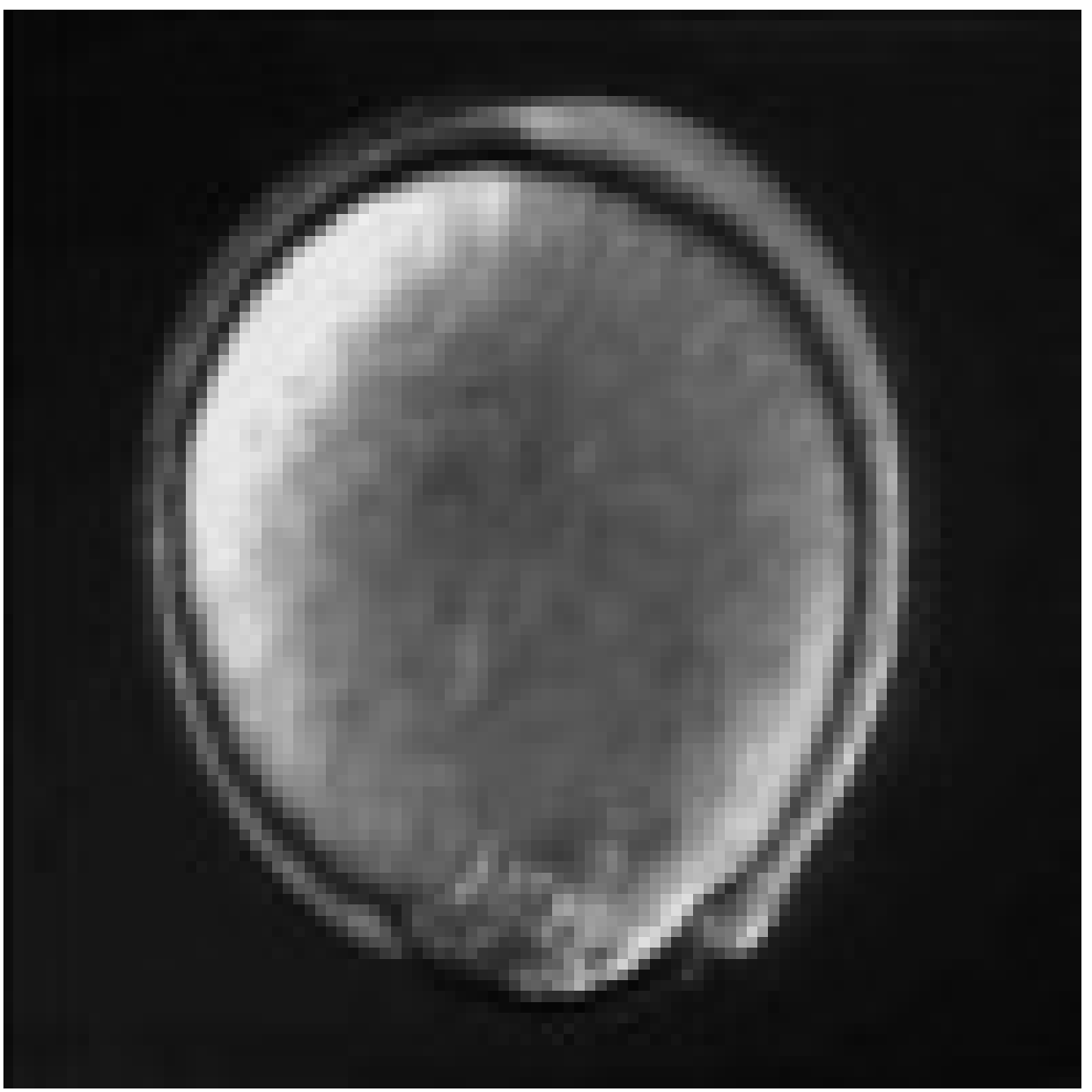} &
\includegraphics[width = .08\textwidth]{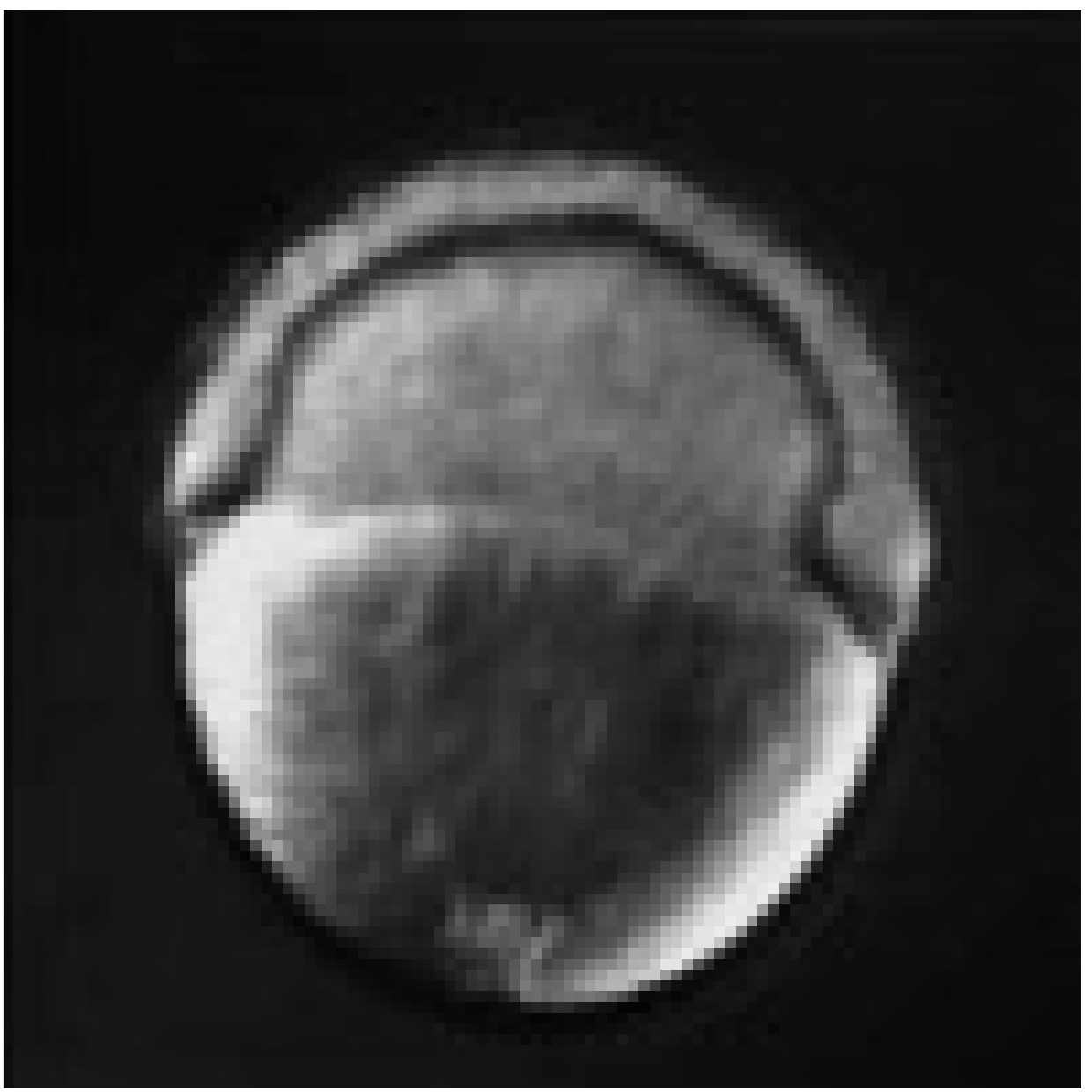} &
\includegraphics[width = .08\textwidth]{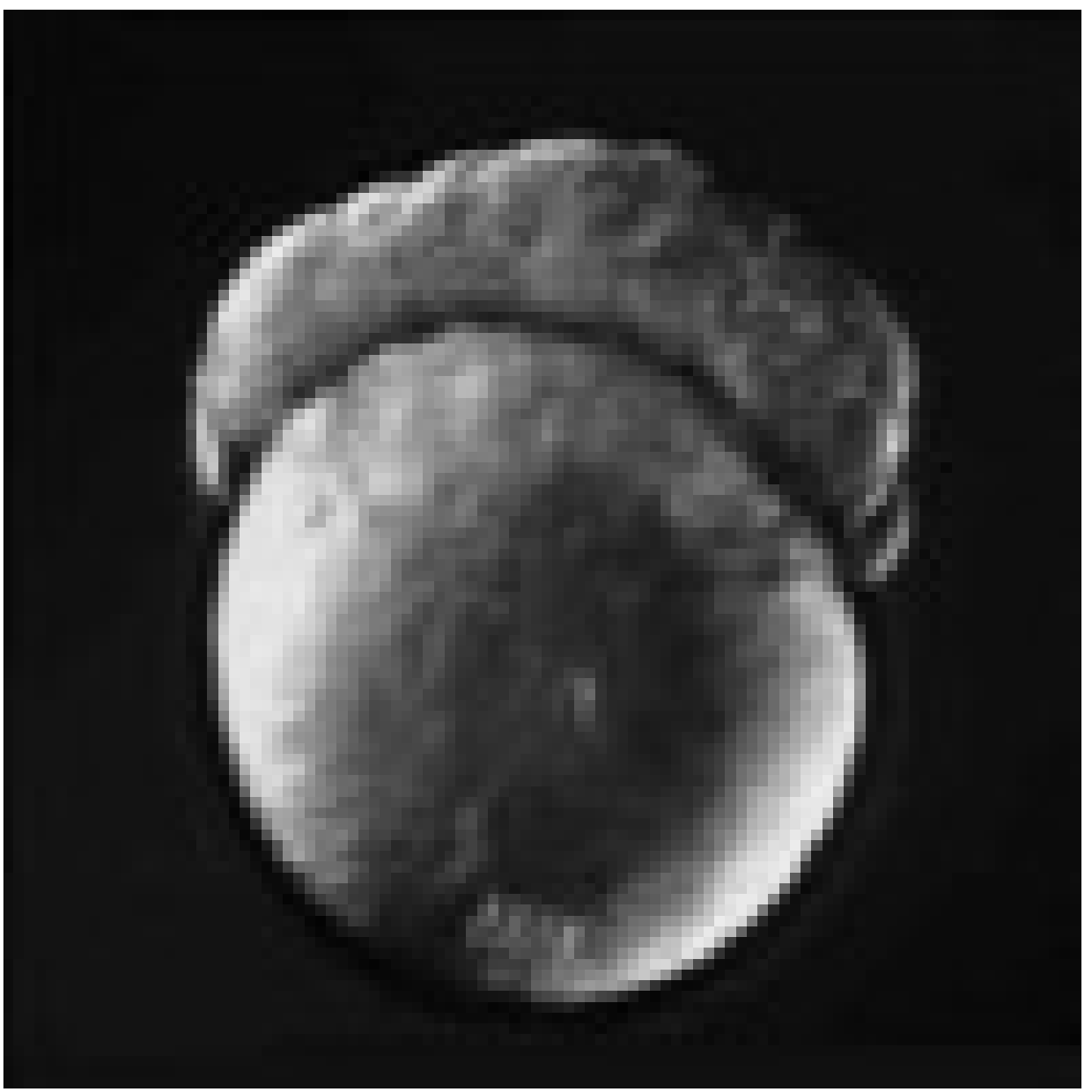} \\
\includegraphics[width = .08\textwidth]{zf_1-eps-converted-to.pdf} &
\includegraphics[width = .08\textwidth]{zf_2-eps-converted-to.pdf} & 
\includegraphics[width = .08\textwidth]{zf_3-eps-converted-to.pdf} & 
\includegraphics[width = .08\textwidth]{zf_4-eps-converted-to.pdf} &
\includegraphics[width = .08\textwidth]{zf_6-eps-converted-to.pdf} &
\includegraphics[width = .08\textwidth]{zf_7-eps-converted-to.pdf} &
\includegraphics[width = .08\textwidth]{zf_8-eps-converted-to.pdf} &
\includegraphics[width = .08\textwidth]{zf_101-eps-converted-to.pdf} &
\includegraphics[width = .08\textwidth]{zf_9-eps-converted-to.pdf} &
\includegraphics[width = .08\textwidth]{zf_10-eps-converted-to.pdf} 
\end{tabular}
\end{center}
\caption{The top row shows 10 images of zebrafish embryos in random order. The bottom row shows 10 images ordered by the found curve that minimizes \eqref{mppc}. 
% The images in the bottom row correspond to points 1, 10 ,20, 30, 40, 50, 60, 70, 80, and 90.
 }
\label{zf_images}
\end{figure}

%We note the problem in this example can essentially be reduced to finding the correct ordering of the data, and since there is no noise, a curve can then be obtained by simply connecting the data points in the correct order. 
%The problem of recovering the order of the data can be approached through nonlinear dimensionality reduction methods that provide low-dimensional embedding coordinates of the data. Indeed, a variant of the diffusion maps algorithm has been successfully applied to register and temporally order images, including the images in this example \cite{dsilva15}. Used there, vector diffusion maps allow for the added function of being able to factor out any existing rotations of the images. To illustrate the applicability of such approaches, we show the 1-dimensional embedding (ordering) of the data given by standard diffusion maps \cite{CoiLaf06} in our simpler case of already registered images. The found ordering is shown in Figure \ref{zf_color}, and, as the ordering induced by the curve found by our algorithm, corresponds to the true ordering of the images.

\begin{figure}[!htbp]
\centering
 \subfigure[ The curve found minimizing \eqref{mppc} ]{\includegraphics[width=0.44 \textwidth]{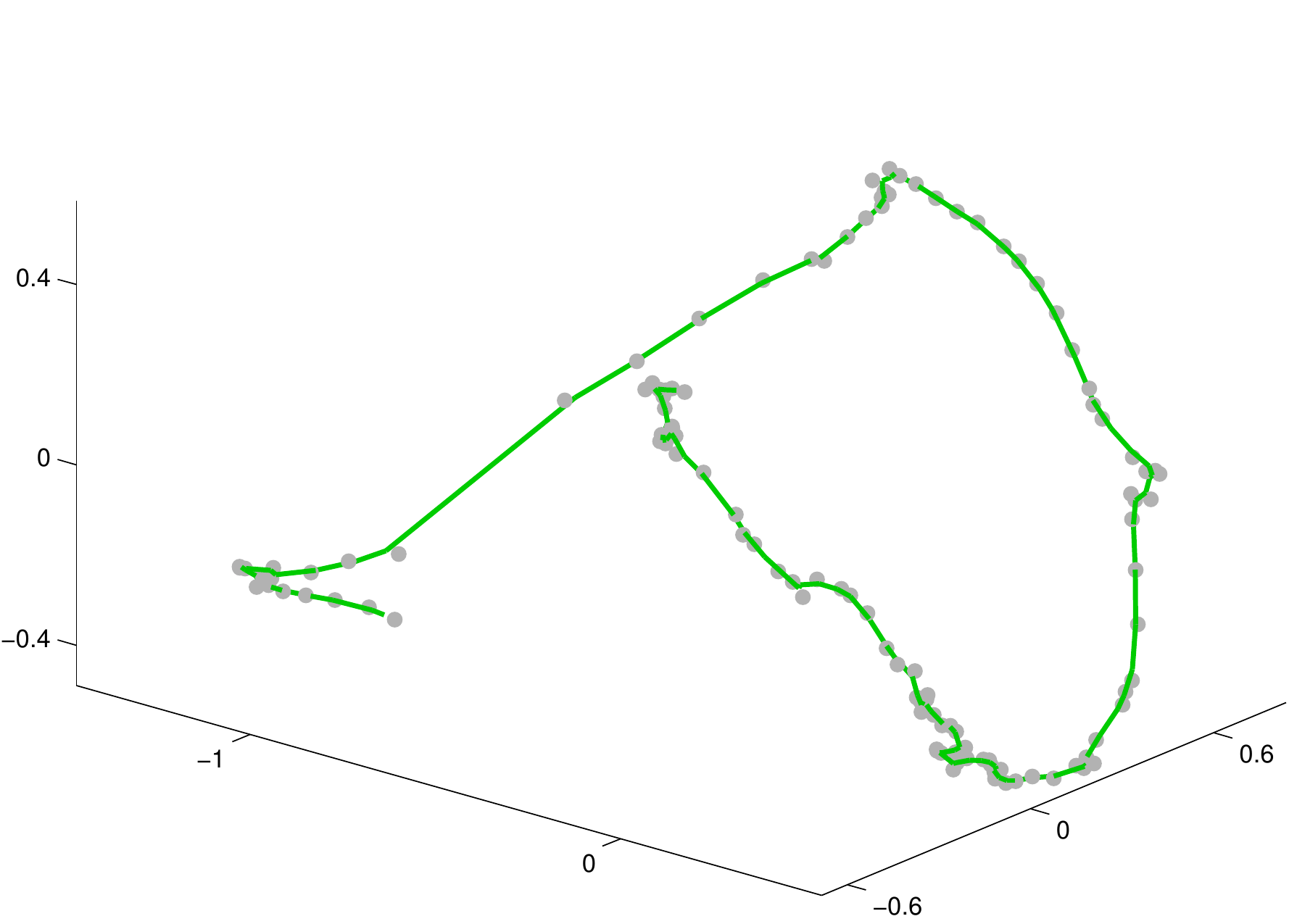} \label{zf_curve} }  
  \hspace*{16pt} \nolinebreak
 \subfigure[Color-coded first embedding coordinate of the diffusion map]{\includegraphics[width=0.44 \textwidth]{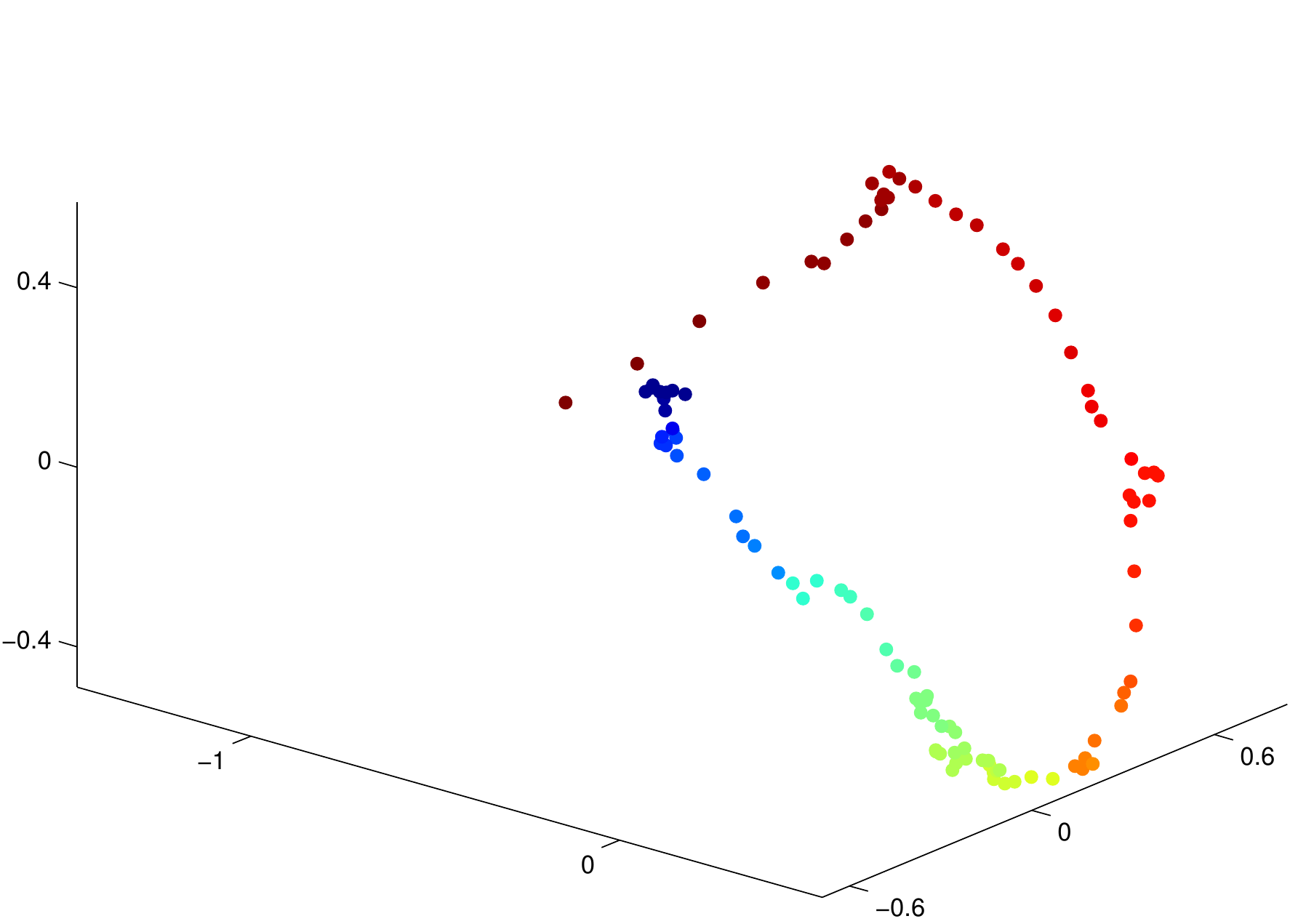}   \label{zf_color}} 
\caption{On both images the first three principal components are used for visualization.
The \eqref{mppc} algorithm was applied to all 120 images, while we applied diffusion maps to only the first 104 images due to a slight camera shift that resulted in relatively large euclidean distance between images 104 and 105. Both methods perfectly ranked their respective data, and some (simple) preprocessing  done in \cite{dsilva15} allows diffusion maps to work on the full 120 images.}
\end{figure}
\end{example}

\begin{example} \emph{Noisy spiral revisited.}
In the previous example we discussed the feasibility of using nonlinear dimensionality reduction techniques such as diffusion maps to order the data. Since the data in Example \ref{zebrafish} had almost no noise, one can obtain a good ordering using many different methods.
%f there is noise and an ordering not far from the truth is found, one can still obtain a curve approximating the data with standard regression. In essence, one can view the challenging part of our curve approximation problem as the problem of finding the intrinsic ordering of the data, as was formulated by Gerber and Whitaker in \cite{gw13}. 
Spectral dimensionality reduction techniques are often successful even when substantial noise is present. However, when there is significant overlap in the distribution of data whose generating points have large intrinsic distance, spectral methods can fail to recover the desired one-dimensional ordering. 
%However when noise is relatively large and the one dimensional structures present in data are long and curved spectral methods can fail to recover the desired one dimensional ordering. 
The example below illustrates this and indicates that in some situations minimizing \ref{mppc} gives better results in ordering the data than diffusion maps. 

We revisit the noisy spiral data considered in Example \ref{spiral}, and run the diffusion maps algorithm using a range of scaling parameters $\epsilon = (\frac{d}{c})^2$, where $d$ denotes the median of the pairwise distances of the data points. 
After testing a wide range of parameter values $c$, we found that for all values tested the spectral embedding fails to recover the desired one-dimensional ordering.
We display the typical results (which correspond to $c=0.5, 2,$ and $4$) in Figure \ref{sp_dm}.
Larger values of $\epsilon$ lead to an embedding that differentiates the data linearly from bottom left to top right, while smaller values lead to an embedding that separates outliers in the top from the rest of the data points. On the other hand minimizer of \eqref{mppc} can correctly recover the one dimensional structure, as shown in Figure \ref{ex_6b}. 
%
%Spectral dimensionality reduction methods behave globally in that they aim to preserve the sum of locally weighted pairwise distances. The example illustrates a scenario in which the interface between data points that are not visibly sufficiently close to each other is long enough that it becomes more advantageous to map them close to each other, than to differentiate them along the length of the spiral. In contrast, our approach is not affected by length of this interface, and is instead guided by the local density of points. 

\begin{figure}
\centering
 \subfigure[$\epsilon = 4d^2$]{\includegraphics[width=0.31 \textwidth]{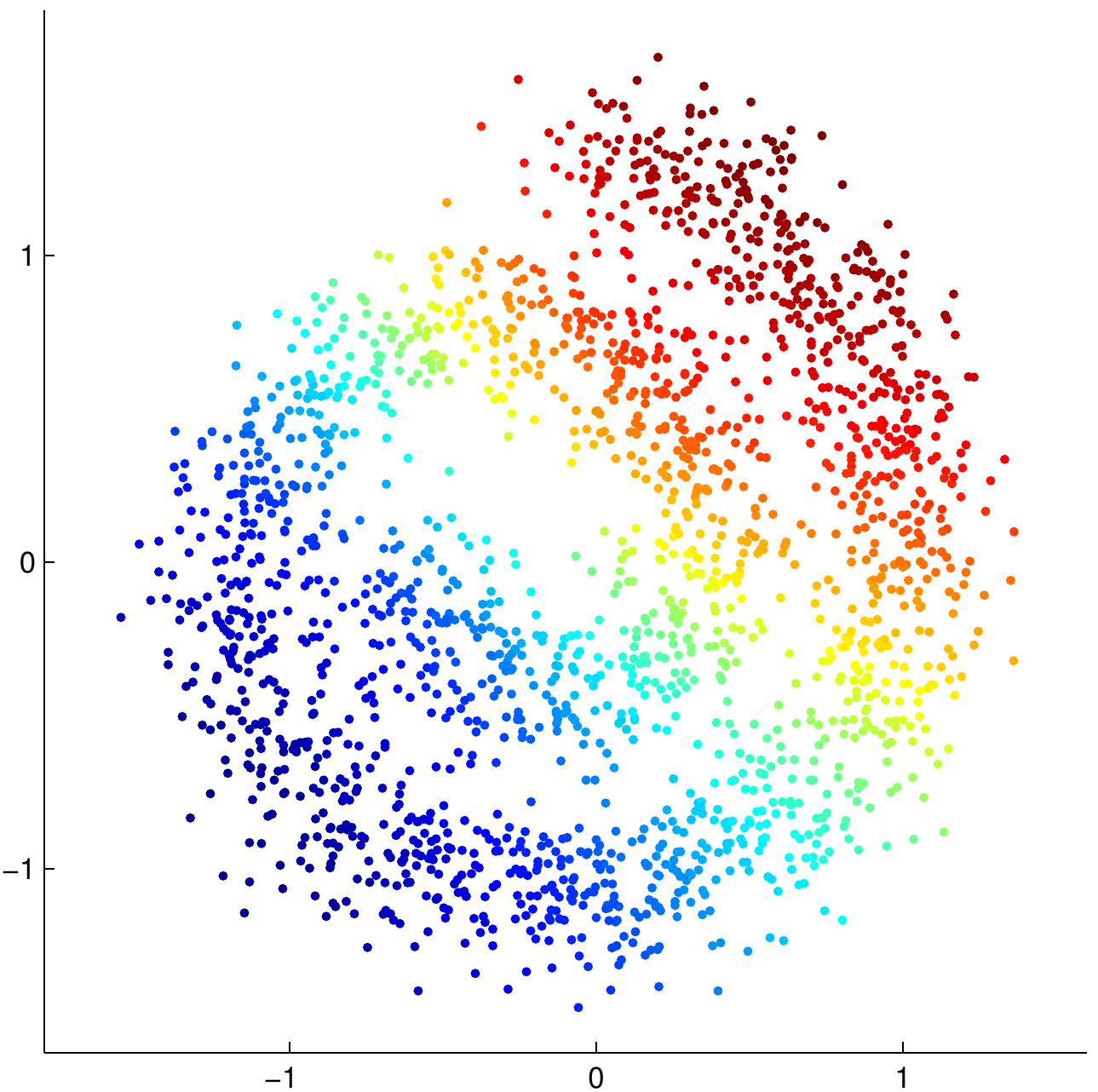} \label{sp1} }  
  \hspace*{3pt} \nolinebreak
 \subfigure[$\epsilon = \frac{d^2}{4}$]{\includegraphics[width=0.31 \textwidth]{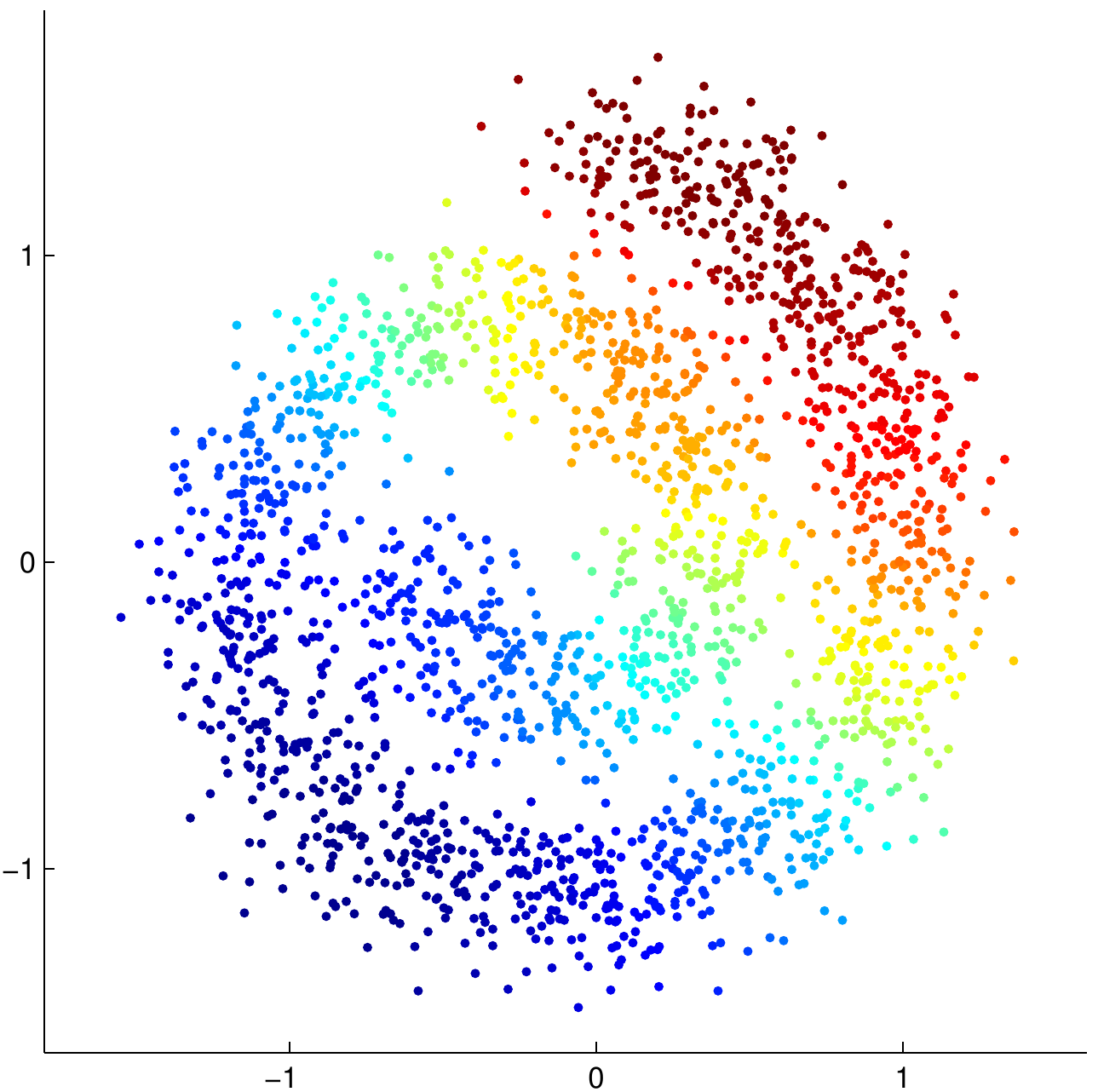}   \label{sp2}} 
   \hspace*{3pt} \nolinebreak
 \subfigure[$\epsilon = \frac{d^2}{16}$]{\includegraphics[width=0.31 \textwidth]{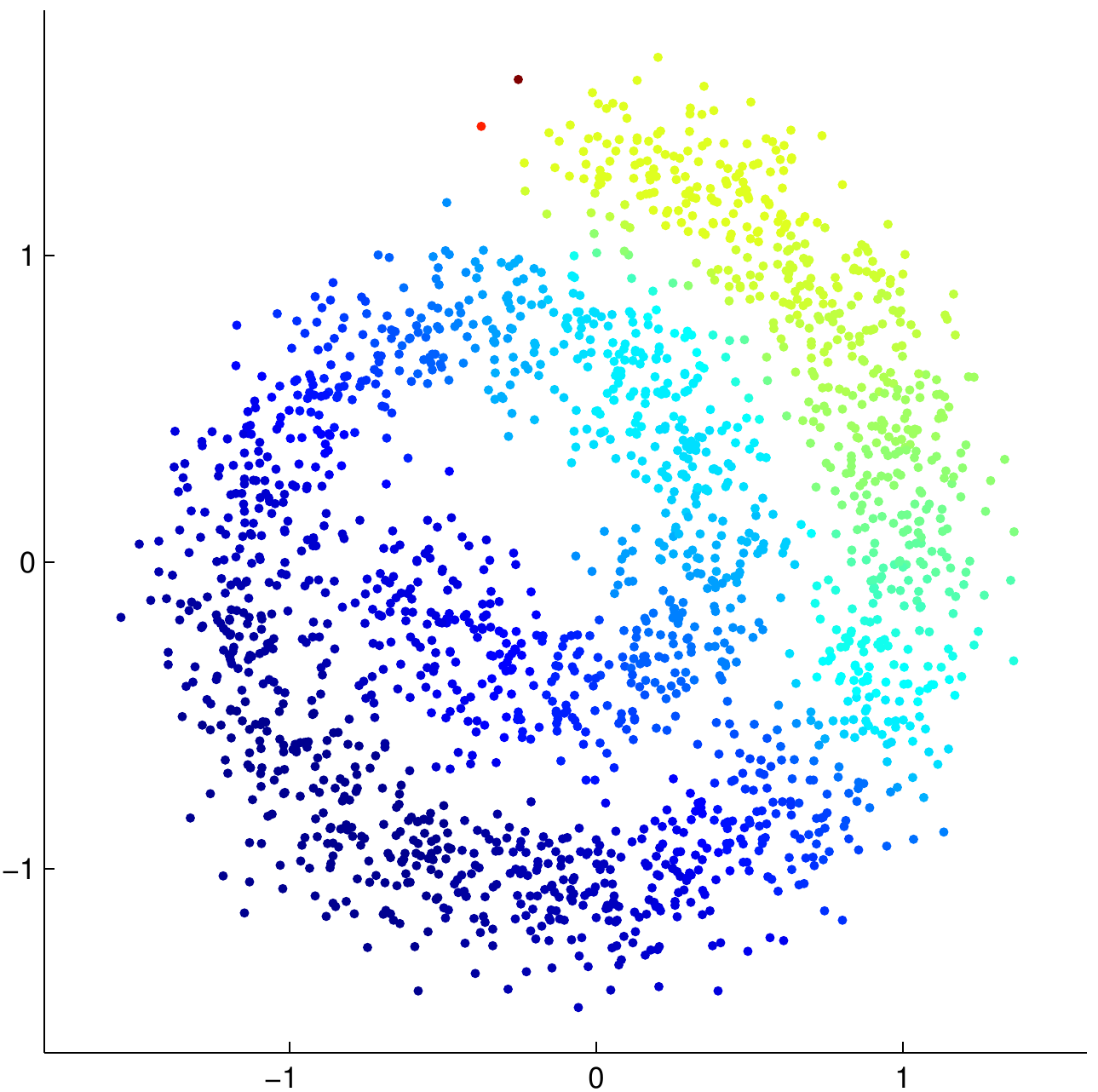}   \label{sp3}} 
\caption{Color-coded one-dimensional embeddings provided by the diffusion maps algorithm for three settings of the scaling parameter $\epsilon$. No values for $\epsilon$ recover the intrinsic ordering of the data. Larger values (left and middle) cannot detect the finer structure, while a smaller value (right) separates two outliers at the top (colored red and brown) from the rest of the data.}
\label{sp_dm}
\end{figure}

\end{example}
\nc

%%%%%%%%%%%%%%%%%%%%%%%%%%%%%%%%%%%%%%%%%%%%%%%%%%
\section{Discussion and Conclusions} \label{sec:dc}

\nc 
In this paper, we proposed a new objective functional \eqref{mppc} for finding one-dimensional structures in data that allows for representation consisting of several components. The functional introduced is based on the average-distance functional and can be seen as a regularization of principal curves. It penalizes the approximation error, total length of the curves, and the number of curves used. We have investigated the relationship between the data generated by one dimensional signal with noise, the parameters of the functional, and the minimizer. Our findings provide guidance for the choice of parameters, and can further be used for multi-scale representation of the data. In addition, we have demonstrated that the zeroth-order term helps energy descent based algorithms converge to desirable configurations. 
In particular, energy descent approaches for \eqref{ppc} very often end up in undesirable local minima. The main reason for this is of topological nature -- points on the approximate local minimizer represent the data points in an order which may be very different from the true ordering corresponding to the (unknown) generating curve. The added flexibility of being able to split and reconnect the curves provides a way for resolving such topological obstacles. 

Finally, we have developed a fast numerical algorithm for estimating minimizers of \eqref{mppc}. It has computational complexity $\mathcal{O}(mnd)$, where $n$ is the number of data points in $\R^d$, and $m$ is the number of points along the approximating curve(s). We demonstrated the effectiveness of the algorithm in recovering the underlying one-dimensional structure for real and synthetic data, in cases with significant noise and in very high dimensions.  

%  We give insight as to how solutions depend on the regularization parameters, and we provide a computationally efficient and robust numerical scheme for finding solutions, that scales linearly with both the number of data and the dimension. We demonstrate performance on several datasets and observe that results correspond to desirable zero and one-dimensional summaries. Our approach originates from the average-distance problem, which has been previously studied in the mathematical community. 

%As the proposed algorithm does not come with any guarantees for conditions under which a global minimizer will be recovered, this is an important direction for future research. In cases where only one curve is sought, a path-solve method for $\lambda_2$ might be effective. This would entail solving the problem for a sequence of increasing $\lambda_2$ values, using each result as initialization for the next step. The smallest $\lambda_2$ values would correspond to a configuration of singletons, and the final largest value would be the smallest value needed to obtain a single curve. 

\subsection{Relation to other approaches.}

We now briefly compare the proposed approach to other existing approaches for finding one-dimensional structures in data. 
The original principal curves are prone to overfitting, as carefully explained in \cite{gw13} and are difficult to compute numerically. The approach of Gerber and Whitaker \cite{gw13} offers a more stable notion and an effective numerical implementation in low dimensions. 
Experiments by the authors indicate that the algorithm often selects desirable minimizers.
However, the functional may still overfit noisy data, as there are many curves which minimize the functional but do not represent the data well.
%We should mention that the experiments by the authors indicate that the algorithm often selects desirable minimizers. 
We also note that since the functional does not measure the approximation error, its low values are not a measure of how well the data are approximated.  On the other hand if the data are sampled from a smooth curve the minimizers can be expected to recover the curve exactly.

A number of works inspired by principal curves treat the problem by considering objective functionals which regularize the principal curves problem. Among them are works of 
Tibshirani \cite{Tib92} (square curvature penalization)
Kegl, Krzyzak, Linder, and Zeger \cite{kegl} (length constraint), Biau and Fischer \cite{BiaFis12} (length constraint)
Smola, Mika, Sch\"olkopf, and  Williamson \cite{sms01} (a variety of penalizations including penalizing length as in \eqref{ppc}). These works are similar in spirit to our approach. The work of 
Biau and Fischer \cite{BiaFis12} 
%\footnote{This paper has many references to applications of principal components. Also their seismic data seems like a nice real data set to try. \nc} 
also discusses model-selection based automated ways to choose parameters of the given functional for the specific data set.
 Wang and Lee \cite{LeeWan06ap} also use model selection to select parameters, but ensure the regularity of the minimizer in a different way. Namely they model the points along the curve as an autoregressive series.

Regarding numerical approaches,  Kegl, Krzyzak, Linder, and Zeger \cite{kegl} proposed a polygonal-line algorithm that penalizes sharp angles. Feuersänger and Griebel employ sparse grids to minimize a functional similar to \eqref{ppc}, with length squared regularization \cite{FeuGri09} (as in  \cite{sms01}) for manifolds up to dimension three. While these approaches take measures against overfitting data, they do not address the problem of local minima,
%allow for splitting and reconnecting, 
resulting in performance that is very sensitive to the initialization of the algorithms. Verbeek, Vlassis and Kr\"ose \cite{VerVlaKro02} approach this issue by iteratively inserting, fitting, and connecting line segments in the data. This approach is effective in some situations where others exhibit poor performance (e.g. spiral in 2-d, some self-intersecting curves, and curvy data with little noise). However, in cases of higher noise the algorithm overfits if the number of segments is not significantly limited. A better understanding of the impact of the number of segments on the final configuration is still needed, despite some efforts to automatize selection of this parameter \cite{LeeWan06ap}.  

% Our approach differs in that  ADMM based optimization is in general significantly faster for these functionals and that the flexibility of multiple components often prevents the algorithm for being attracted to an undesirable local minimum.

%What we provide is a new functional (multicomponent), but also theoretical results on properties of minimizers,  precise understanding of parameters, and an efficient numerical scheme.

A different class of approaches to finding one dimensional structures is based on estimating the probability density function of the point cloud and then finding its ridges \cite{Eberly}. In particular  Ozertem and Erdogmus \cite{oe11} introduced the widely used Subspace Constrained Mean Shift (SCMS) algorithm, which a based on the Mean Shift algorithm of Comaniciu and Meer \cite{ComMee02}. 
Estimation of density ridges has been substantially developed and studied -- see works of 
Chen,  Genovese, and Wasserman \cite{CGW15},
   Genovese, Perone-Pacifico, Verdinelli, and Wasserman  \cite{genovese14}, and Pulkkinen \cite{p15}. The SCMS algorithm is often able to find the desired one-dimensional structure even when significant noise is present. However it does not automatically parameterize the found one dimensional structure, which consists of an (unordered) set of points.
An important difference between our approach and SCMC is that we seek the one-dimensional structure that best approximates the data, and measure the quality of approximation as part of the \eqref{mppc} finctional, while SCMC does not require the ridges found to approximate that data well. Let us also mention that in high dimensions SCMC  faces a combination of computational difficulties, the primary of which is accurately estimating the Hessian of the density function (found using a kernel density estimator), as is discussed in Section 3 of \cite{oe11}.

%
% This is a computationally expensive task and \sout{ with best, according to our knowledge, reported complexity $\mathcal{O}(n^2 d^2)$ \cite{p15},} can become intractable for high-dimensional data with many points. In particular, in all of these approaches one needs to calculate the Hessian of a kernel density estimator, and doing so at just a single point requires $\mathcal{O}(n d^2)$ operations. \nc

%\medskip

Recently there has been a significant effort to recover one-dimensional structures which are branching and intersecting and in particular the connectivity network of the data set. The ability to recover graph structures and the topology of the data is very valuable, and facilitates a number of data analysis tasks. 
Several notable works are based on Reeb graphs and related structures \cite{Cha14, GeSaBeWa11, SiMeCa07}. 
We note that these approaches are sensitive to noise and furthermore the presence of noise significantly slows down the algorithms.
 We believe our approach and algorithm could be valuable as a pre-processing step for simplifying the data prior to applying graph-based approaches that find the connectivity network of the data set. Recalling the data from Example \ref{grid} and Figure \ref{fig:window}, we see that although our approach does not recover the topological structure, it does identify and appropriately simplifies the one dimensional structure present in the data. The approaches mentioned here should work much better on the simplified (green or blue) data than on the original point cloud. 
 
Finally we mention the work of Arias-Castro, Donoho, and Huo \cite{ACDH06} who studied optimal conditions 
and algorithms for detecting (sufficiently smooth) one-dimensional structures with uniform noise in the background.

%
% Finally, we briefly mention nonlinear dimensionality reduction methods that are global in the sense that they do not require any kind of initialization. Methods such as locally linear embedding \cite{RowSau00}, Laplacian eigenmaps \cite{BelNiy03}, diffusion maps \cite{CoiLaf06} and several others find lower-dimensional representation of the data by aiming to preserve local distance information. These methods typically give desirable embeddings of the data, but do not provide an explicit curve or manifold that approximates the data in the original space. Subsequent manifold estimation is sometimes performed using regression, as in \cite{ZhaZha06} (where local tangent spaces are first aligned to provide the embedding coordinates). Due to their dependence on pairwise distance information, these methods can be computationally intensive when the number of data points is very large. 

\section*{Acknowledgements}
We are grateful to Xin Yang Lu, Ryan Tibshirani, and  Larry Wasserman for valuable discussions. 
The research for this work has been supported by the National
Science Foundation under grants  CIF 1421502 and DMS 1516677.
We are furthermore thankful to Center for Nonlinear Analysis (CNA) for its support.

%%%%%%%%%%%%%%%%%%%%%%%%%%%%%%%%%%%%%%%%%%%%%%%%%%%

\appendix
\section{Analysis of the uniformly distributed  data on a line segment} \label{apA}

Consider data uniformly distributed with density $\alpha$ on a line segment $[0,L]$. 
The functional \eqref{mppc} takes the form
\begin{equation}
\label{energy}
 E^{\lambda_1,\lambda_2}_\mu(\gamma) := \int_0^L d(x,\gamma)^p \alpha dx + \lambda_1 ( \length(\gamma)  +  \lambda_2 \comp(\gamma ) ) 
 \end{equation}
where $\comp(\gamma )$ is the number of components of $\gamma$ minus 1. We restrict ourselves to $\gamma$ such that $\{0,L\} \subset \te{range}(\gamma)$, so that $\gamma$ takes the form $\gamma = \bigcup_{i=1}^{\comp +1} [a_i,b_i]$, where $a_1=0$, and $b_{\comp+1}=L$. Define $\tau := \sum_{i=1}^{\comp +1} \tau_i$, $g:= \sum_{i=1}^{\comp } g_i$, where $\tau_i := b_i-a_i$ and $g_i := a_{i+1}-b_i$. We make the following observations:

\begin{lemma} \label{lem:rearrange}
The energy  $E^{\lambda_1,\lambda_2}_\mu$ is invariant under redistribution of total length of $\gamma$, assuming that the number of components is $\comp+1$, and that the gap sizes remain constant. More precisely, if $\bar{\gamma} =\bigcup_{i=1}^{\comp +1} [\bar{a_i},\bar{b_i}], \tilde{\gamma} = \bigcup_{i=1}^{\comp +1} [\tilde{a_i},\tilde{b_i}]$ and there exists a permutation $\sigma$ of $\{1, \dots, \comp \}$ such that $\bar{g}_i = \tilde{g}_{\sigma(i)}$ for $i=1,...,\comp $, then $E^{\lambda_1,\lambda_2}_\mu(\bar{\gamma}) = E^{\lambda_1,\lambda_2}_\mu(\tilde{\gamma})$.
\end{lemma}

\begin{lemma} \label{lem:unigap}
For $\comp>0$ fixed, the energy $E^{\lambda_1,\lambda_2}_\mu$ is minimized when the length of the gaps between components are uniform. More precisely, consider an arbitrary $\gamma = \bigcup_{i=1}^{\comp +1} [a_i,b_i]$, with total gap $g$ defined as above. Let $\tilde{\gamma}$ have $\comp+1$ components such that $\tilde{g}_i = g/\comp $, implying that $\tilde{g} = g$. Then $E^{\lambda_1,\lambda_2}_\mu(\tilde{\gamma}) \leq E^{\lambda_1,\lambda_2}_\mu(\gamma)$, with equality only if $g_i = \tilde{g}_i$.
\end{lemma}

\begin{proof}
The result is trivial for $\comp=0$. We prove the result for $\comp=1$. Consider $\gamma$ with $g = g_1 + g_2$. The fidelity part of the energy $E^{\lambda_1,\lambda_2}_\mu$, as a function of $g_1$ is
$$F(g_1) = 2\int_0^{g_1/2} x^p \alpha dx + 2\int_0^{(g-g_1)/2} x^p \alpha dx.$$
Thus
$$ \frac{dF}{dg_1} =  \frac{1}{2^{p-1}} \alpha \left ( g_1^p - (g-g_1)^p \right ) $$
and
$$ \frac{d^2F}{dg_1^2} =  \frac{p}{2^{p-1}} \alpha \left ( g_1^{p-1} + (g-g_1)^{p-1} \right ) \geq 0. $$
By these we see that $g_1$ minimizes the energy if and only if $g_1 = g/2 = g_2$. The result for $\comp>2$ follows since one can consider the above situation by looking at the gaps formed by three consecutive components.
\end{proof}

Using Lemma \ref{lem:rearrange} we may assume that each component not containing the endpoints $0$ or $L$ has the same length $l$, and that the two components containing the endpoints are of length $l/2$.
By Lemma \ref{lem:unigap}, the gaps between the components are 
$ \frac{L-\comp l}{\comp }. $
We first consider $ \comp >0 $ fixed, and minimize the energy w.r.t. $l$ in the range $l \in [0, \frac{L}{\comp })$. The energy 
\begin{align} \label{energy_new}
\begin{split}
E &= \lambda_1 \comp  l + 2\comp \int_0^{\frac{L-\comp l}{2\comp }} x^p \alpha dx + \lambda_1\lambda_2 \comp \\
& = \lambda_1 \comp  l + \frac{2}{p+1}\comp  \left ( \frac{L-\comp l}{2\comp } \right )^{p+1} \alpha + \lambda_1\lambda_2 \comp
\end{split}
\end{align}
is convex on $[0, \frac{L}{\comp })$.
Taking a derivative in $l$ we obtain
$$ \frac{dE}{dl} = \lambda_1 \comp  - \comp   \left ( \frac{L-\comp l}{2\comp } \right )^p \alpha. $$
Setting the derivative to zero and solving for $l$, and by noting that if there is no solution on $[0, \frac{L}{\comp })$ then $E$ is a nondecreasing function of $l$,  we get that the energy is minimized at
\begin{equation} \label{eq:lstar}
 l_{\comp}^* = 
 \begin{cases} \frac{L}{\comp } - 2\left( \frac{\lambda_1}{\alpha} \right)^{1/p} \qquad & \text{if } \quad  \comp  \leq \frac{L}{2\left( \frac{\lambda_1}{\alpha} \right)^{1/p}} =: \bar \comp \\
 0 & \textrm{else.}
\end{cases}
\end{equation}
as we indicate above let $\bar \comp = 1+\frac{L}{2\left( \frac{\lambda_1}{\alpha} \right)^{1/p}} $.
For $\comp$ between $1$ and $\bar \comp$,
plugging back into (\ref{energy_new}) we get that the minimal energy is 
\begin{equation} \label{energy_k}
E_{min}(\comp)  %\lambda_1 \left (L - 2\comp \left( \frac{\lambda_1}{\alpha} \right)^{1/p} \right) + \frac{2}{p+1}\comp  (\lambda_1/\alpha )^{\frac{p+1}{p}} \alpha + \lambda_1\lambda_2 \comp.  $$ 
%= \lambda_1 \left ( - \frac{2p}{p+1} \left( \frac{\lambda_1}{\alpha} \right)^{1/p} + \lambda_2 \right ) (\comp+1)+ \lambda_1 \left ( L +  \frac{2p}{p+1} \left( \frac{\lambda_1}{\alpha} \right)^{1/p} \right ) 
 = \lambda_1 L +\lambda_1 \lambda_2 \comp  - \frac{2p}{p+1} \left( \frac{\lambda_1}{\alpha} \right)^{1/p} \comp 
\end{equation}
By direct inspection we verify that \eqref{energy_k} is the (minimal) energy in the case that there is only one component (no breaks in the line).
We note that (\ref{energy_k}) is linear in $\comp$, and hence for $\comp$ between $0$ and $\bar \comp$,
the minimizing value is at a boundary:
\[ {\comp}^* = 
\begin{cases}
0 & \text{if} \quad  \lambda_2 \geq  \frac{2p}{p+1} \left( \frac{\lambda_1}{\alpha} \right)^{1/p}    \\
\left\lfloor \bar \comp \right\rfloor & \text{otherwise} \\
\end{cases} \]

We now consider $\comp > \bar \comp$ when all components have length zero ($l_{\comp}^* = 0$).
The energy in this case is
$$ E_{l=0}(\comp) := \frac{2}{p+1} \left ( \frac{L}{2} \right )^{p+1}  \frac{1}{\comp ^p} \alpha + \lambda_1\lambda_2 \comp$$
Considering $\comp$ as a real variable  we note that $E_{l=0}(\comp)$ is a convex function.
Taking a derivative in $\comp$ gives
$$ \frac{dE_{l=0}}{d\comp} = \frac{-2p}{p+1} \left ( \frac{L}{2\comp } \right )^{p+1} \alpha + \lambda_1\lambda_2. $$

If $ \lambda_2 \geq  \frac{2p}{p+1} \left( \frac{\lambda_1}{\alpha} \right)^{1/p}$ then 
$\frac{dE_{l=0}}{d\comp} \geq 0$ for $\comp > \bar \comp$.

If $ \lambda_2 <  \frac{2p}{p+1} \left( \frac{\lambda_1}{\alpha} \right)^{1/p}$ then the point where the minimum is reached 
$$ \bar{\comp}_{l=0}^* =  \frac{L}{2} \left ( \frac{(p+1)\lambda_1 \lambda_2}{2p \alpha} \right )^{-\frac{1}{p+1}} $$
satisfies $\bar{\comp}_{l=0}^* >  \bar \comp$ and thus belongs to the range considered. 
If $\bar{\comp}_{l=0}^* $ is an integer then it is the minimizer of the energy, otherwise the minimizer is 
in the set $\{ \lfloor \bar{\comp}_{l=0}^* \rfloor,  \lfloor \bar{\comp}_{l=0}^* \rfloor +1\}$. 
In all cases let us denote by $\comp_{l=0}^*$ the minimizer of the energy:
$ \comp_{l=0}^* = \argmin_{k=\lfloor \bar{\comp}_{l=0}^* \rfloor, \lceil \bar{\comp}_{l=0}^* \rceil} E_{l=0}(k)$.

We note that there is a special case that ${\comp}^*_{l=0}< \bar \comp$. In that case the minimizer of the energy 
with exactly ${\comp}^*_{l=0}+1$ components will be the one considered in the analysis of the $1 \leq \comp \leq \bar \comp$ case, and thus will have segments of positive length $l^*$ given by formula \eqref{eq:lstar}.

%
% We observe that $\comp_{l=0}^*$ will be preferred to $\bar \comp$ if and only if $E_{l=0}(\comp_{l=0}^*) < E_{l^*}({\comp}^*)$, where
%$$E_{l=0}(\comp_{l=0}^*) =  \frac{L\alpha}{p+1} \left ( \frac{(p+1)\lambda_1\lambda_2}{2p\alpha} \right )^{p/{p+1}} + \lambda_1 \lambda_2 \left [ \frac{L}{2} \left ( \frac{(p+1)\lambda_1\lambda_2}{2p\alpha} \right )^{-1/{p+1}} + 1 \right ] $$
%$$E_{l^*}({\comp}^*) = \lambda_1 \left ( - \frac{2p}{p+1} \left( \frac{\lambda_1}{\alpha} \right)^{1/p} + \lambda_2 \right ) \lfloor {1+(\lambda_1/\alpha)^{-1/p}L/2} \rfloor+ \lambda_1 \left ( L +  \frac{2p}{p+1} \left( \frac{\lambda_1}{\alpha} \right)^{1/p} \right ) $$

To summarize, the optimal number of components will be
\begin{equation} \label{app:nocomp}
\begin{cases}
1 & \text{if} \quad  \lambda_2 \geq  \frac{2p}{p+1} \left( \frac{\lambda_1}{\alpha} \right)^{1/p}    \\
\lfloor \bar \comp\rfloor +1  & \text{if} \quad  \lambda_2 <  \frac{2p}{p+1} \left( \frac{\lambda_1}{\alpha} \right)^{1/p},  \quad \text{and} \quad  {\comp}^*_{l=0} < \bar \comp\\
\comp_{l=0}^*+1 & \text{if} \quad  \lambda_2 <  \frac{2p}{p+1} \left( \frac{\lambda_1}{\alpha} \right)^{1/p},  \quad \text{and} \quad  {\comp}^*_{l=0} \geq \bar \comp. \\
\end{cases} 
\end{equation}
In the first case, there is just one single connected component. In the second case there are $ \lfloor \bar \comp \rfloor +1$ components, each with equal positive length.  We note that by Lemma 
\ref{lem:rearrange} there exists a configuration with the same energy where one of these components has positive length, while the rest have zero length. The third case is that  each of the components has length zero. We point out that if $\bar \comp$ is integer-valued and $ \lambda_2 <  \frac{2p}{p+1} \left( \frac{\lambda_1}{\alpha} \right)^{1/p}$, then the minimizer will have $\comp_{l=0}^*+1$ components.

We can now derive conclusions to the structure of minimizers if $L \gg 1$. From above we conclude that the minimizer will have one component (and be a continuous line) if $\lambda_2 \geq  \frac{2p}{p+1} \left( \frac{\lambda_1}{\alpha} \right)^{1/p}$, and break up into at least $ \lfloor \bar {\comp}^*_{l=0} \rfloor+1$ components otherwise. Rearranging, the condition also provides the critical density at which topological changes (gaps) in minimizers occur:
\begin{equation} \label{app:alpha}
 \alpha^* = \left ( \frac{2p}{p+1} \right )^{p}  \frac{\lambda_1}{\lambda_2^p}. 
\end{equation}

Finally we note that the typical gap length is $L/(\bar {\comp}^*_{l=0})$ that is 
\begin{equation} \label{app:L}
 L^* = 2  \left ( \frac{(p+1)\lambda_1 \lambda_2}{2p \alpha} \right )^{\frac{1}{p+1}}. 
\end{equation}

\bibliography{references}
\bibliographystyle{siam}

\end{document}